\documentclass{amsart}
\usepackage{amsmath,amsthm,amssymb,amsfonts,amscd}
\usepackage{mathrsfs}
\usepackage{bbding}
\usepackage{graphicx,latexsym}
\usepackage{hyperref}
\usepackage{geometry}
\usepackage{color}

\usepackage{chngcntr}

\counterwithout{subsubsection}{subsection}
\makeatletter
\let\c@equation\c@subsection

\let\c@subsubsection\c@subsection

\makeatother

\numberwithin{equation}{section}

\theoremstyle{plain}
\newtheorem{theorem}[equation]{Theorem}
\newtheorem{lemma}[equation]{Lemma}

\newtheorem{claim}[equation]{Claim}

\theoremstyle{definition}
\newtheorem{definition}[equation]{Definition}

\theoremstyle{remark}
\newtheorem{remark}[equation]{Remark}

\renewcommand{\Re}{\operatorname{Re}}
\renewcommand{\Im}{\operatorname{Im}}

\newcommand{\rad}{\operatorname{rad}}

\newcommand{\Res}{\operatorname{Res}}

\renewcommand{\mod}{\operatorname{mod}\ }

\newcommand{\cA}{\mathcal{A}}
\newcommand{\cB}{\mathcal{B}}

\newcommand{\cD}{\mathcal{D}}
\newcommand{\cE}{\mathcal{E}}
\newcommand{\cF}{\mathcal{F}}

\newcommand{\cH}{\mathcal{H}}

\newcommand{\cJ}{\mathcal{J}}
\newcommand{\cK}{\mathcal{K}}

\newcommand{\cM}{\mathcal{M}}
\newcommand{\cN}{\mathcal{N}}

\newcommand{\cS}{\mathcal{S}}

\newcommand{\cV}{\mathcal{V}}
\newcommand{\cW}{\mathcal{W}}

\newcommand{\sV}{\mathscr{V}}

\newcommand*{\bbC}{\ensuremath{\mathbb{C}}}

\newcommand*{\bbR}{\ensuremath{\mathbb{R}}}

\newcommand*{\bbN}{\ensuremath{\mathbb{N}}}

\newcommand{\ve}{\varepsilon}

\begin{document}

\title[Super-positivity of a family of L-functions]
        {Super-positivity of a family of L-functions}
\author{Dorian Goldfeld and Bingrong Huang}
\address{Department of Mathematics \\ Columbia University \\ New York \\NY 10027 \\USA}
\email{goldfeld@columbia.edu}
\address{School of Mathematics \\ Shandong University \\ Jinan \\Shandong 250100 \\China}
\email{brhuang@mail.sdu.edu.cn}
\thanks{Dorian Goldfeld is partially supported by NSA Grant H98230-16-1-0009.
 Bingrong Huang is supported in part by\break
\phantom{xx} NSFC grant 11531008
    and IRT\_16R43 from the Ministry of Education, China. Bingrong Huang also thanks the\break
\phantom{xx} China Scholarship Council for supporting his studies at Columbia University.}

\begin{abstract}
  Zhiwei Yun and Wei Zhang introduced the notion of ``super-positivity of self dual L-functions'' which specifies that all derivatives of the completed L-function (including Gamma factors and power of the conductor) at the central value $s = 1/2$ should be non-negative. They proved that the Riemann hypothesis implies super-positivity for self dual cuspidal automorphic L-functions on $GL(n)$. Super-positivity of the Riemann zeta function was established by P\' olya in 1927 and since then many other cases have been found by numerical computation.  In this paper we prove, for the first time, that there are infinitely many  L-functions associated to modular forms for $SL(2, \mathbb{Z})$ each of which has the super-positivity property. Our proof also establishes that all derivatives of the completed L-function at  any real point $\sigma > 1/2$ must be positive.
\end{abstract}

\maketitle


\section{Introduction} \label{sec:intr}


Let $F$ be a number field and let $\mathbb A_F$ be the ad\'ele ring of $F$ which is the restricted product $\prod_v F_v$ over the completions of $F$.  A cuspidal automorphic representation $\pi$ of $GL(n, \mathbb A_F)$ can be written as a tensor product  $\pi = \bigotimes \pi_v$ of local representations.  Then $\pi$
 has a Godement--Jacquet L-function
$$L(s, \pi) = \prod_v L(s, \pi_v),$$
where
$$L(s, \pi_v) = \begin{cases} \prod\limits_{j=1}^n \left( 1 - \frac{\alpha_j(v)}{N(v)^s}\right)^{-1}, & \text{if} \; v \; \text{is non-archimedean,}\\
 \prod\limits_{j=1}^n \Gamma_v\big(s - \mu_j(v)\big),&  \text{if} \; v \; \text{is archimedean,}\\
\end{cases}$$
with $\alpha_j(v), \; \mu_j(v) \in\mathbb C$ for $j = 1,2,\ldots,n,$ and
$$\Gamma_v(s) = \begin{cases}
\pi^{-\frac{s}{2}}\Gamma\left( \frac{s}{2}  \right), & \text{if} \; F_v \sim \mathbb R,\\
(2\pi)^{-s} \Gamma(s), & \text{if}\; F_v \sim \mathbb C.
\end{cases}
$$

Let $\widetilde\pi$ denote the
 contragredient representation. It is well known that  $L(s, \pi)$
 is an entire function of  order 1 and satisfies
a functional equation of the form (see \cite{GJ, goldfeld-hundley})
$$L(s,\pi) = \epsilon(s,\pi) L(1-s, \widetilde \pi),$$
with $$\epsilon(s,\pi) = \epsilon(\pi)  N_\pi^{1/2-s},$$
where
$N_\pi \geq 1$ is the conductor of $\pi$ and $\epsilon(\pi)$ is the root number satisfying $|\epsilon(\pi)|=1.$ If $\pi = \widetilde\pi$, i.e., $\pi$ is self dual, then
$\epsilon(\pi) = \pm 1.$

   Zhiwei Yun and Wei Zhang \cite{Y-Z} introduced the notion of super-positivity for self dual cuspidal automorphic representations $\pi$ of $GL(n, \mathbb A_F)$ which specifies that all derivatives   of the completed L-function,
  $$\Lambda(s, \pi) := N_\pi^{\frac{s-1/2}{2}} L(s,\pi) = \pm \Lambda(1-s, \pi),$$   at $s = 1/2$ should be greater or equal to zero. They proved that super-positivity holds for self dual cuspidal automorphic L-functions (over any global field) which satisfy the Riemann hypothesis. In the case that $F$ is a function field, the Riemann hypothesis is known
   by the theorem of Deligne on Weil's conjecture, and of Drinfeld and L. Lafforgue on the global Langlands correspondence, so it is now known that super-positivity holds for cuspidal automorphic automorphic representations of
 $GL(n, \mathbb A_F)$ where $F$ is a function field.

Super-positivity was  established   for the example of the  Riemann zeta function  by P\'olya in 1927 (see \cite{polya1927algebraisch} and \cite{csordas1986riemann}). We would like to thank Peter Sarnak for informing us that super-positivity  is also known in many other cases (including quadratic Dirichlet L-functions and L-functions of $GL(2)$ modular forms) when the L-function is {\it ``positive definite''} as defined by Sarnak \cite{Sarnak}. It is not hard to check numerically if an L-function is positive definite or not. For example, in the case of an L-function  associated to a holomorphic modular form $f$  for $SL(2,\mathbb Z)$ (with Fourier coefficients $a_n$) it is enough to check
 if $f(iy) = \sum\limits_{n=1}^\infty a_n e^{-2\pi ny}$ is positive for $y\ge 1.$

 \vskip 3pt
 It was shown by Jung \cite{Jung} that almost all L-functions in any reasonable family will not be positive definite. It is not known if there are infinitely many self dual automorphic L-functions which are positive definite.

  It seems to be infeasible to prove super-positivity for all Dirichlet L-functions  at this time  since it would  follow that there are no Siegel zeros (real zeros near $s=1$) for Dirichlet L-functions, which is known to be a notoriously difficult problem. Similarly, proving super-positivity for all cuspidal automorphic L-functions on  $GL(n, \mathbb A_F)$ (with $F$ = a number field and $n > 1$) also seems hopeless at present.

 \vskip 8pt
 The main aim of this paper is to prove that there are infinitely many examples of cuspidal automorphic L-functions for $GL(2, \mathbb A_{\mathbb Q})$ which have the super-positivity property. Following \cite{Y-Z} we will actually prove our results for the following expanded definition of super-positivity.

 \begin{definition} {\bf (Super-positivity)}\label{def:superpositivity}
Fix a  number field $F$. Let $\pi$ denote a self dual cuspidal automorphic representation of $GL(n, \mathbb A_F)$ with conductor $N_\pi.$   We say $\pi$ has the super-positivity property if
\begin{align*} &{\bf (1)} \;\Lambda^{(k)}(1/2,\pi) = \left(\frac{d}{ds}\right)^k \Lambda(s, \pi)\bigg|_{s=1/2}  \ge \; 0, \quad\big (\text{for all} \;k = 0,1,2,\ldots\big),\\
&{\bf (2)} \;\Lambda^{(k)}(\sigma,\pi) = \left(\frac{d}{ds}\right)^k \Lambda(s, \pi)\bigg|_{s=\sigma} > \; 0, \quad\big (\text{for $\sigma>1/2$ and all} \;k = 0,1,2,\ldots\big),\\
&{\bf (3)} \; \Lambda^{(k_0)}(1/2,\pi) \ne 0 \; \implies\;  \Lambda^{(k_0+2i)}(1/2,\pi) \ne 0, \quad \big(\text{for some $k_0 \ge 0$ and all}\; i = 0,1,2,\ldots\big).
\end{align*}
\end{definition}

\vskip 8pt
We now state our main results. Proofs will follow in subsequent sections.

\begin{theorem}\label{thm:superpositivity1}
Let $\pi$ denote a self dual  cuspidal automorphic representation of $GL(n, \mathbb A)$. For $s\in\mathbb C$ let $\Lambda(s,\pi)$ be the completed L-function of $\pi$   with functional equation
$$\Lambda(s,\pi) = \pm \Lambda(1-s,\pi).$$
Assume that
$\Lambda(s,\pi) \ne 0$ for $s = \sigma + i t$ where
$1/2 < \sigma < 1, \; |t| \leq \sigma - 1/2.$
Then $\pi$ has the super-positivity property as in definition (\ref{def:superpositivity}).

\end{theorem}

\begin{remark}
  The above theorem is essentially due to Stark--Zagier \cite{stark1980property},
  but we will give the short simple proof in \S \ref{sec:proof1} 
  for the convenience of the reader.
\end{remark}

\begin{remark} 
Theorem \ref{thm:superpositivity1} provides a method to manually check if an individual L-function has the super-positivity property. It is enough to check, for example, that  all the zeros of $\Lambda(s,\pi)$ with imaginary part $\leq 1/2$ are on the line $\Re(s) = 1/2.$ By examining the zeros of L-functions database \cite{LMFDB} one sees that all the self dual L-functions in this database do indeed satisfy the super-positivity property.
\end{remark}

Next, we apply the theorem \ref{thm:superpositivity1} to show that there is an infinite family of $GL(2)$ L-functions with the super-positivity property.
Let $S_k$ denote the space of holomorphic cusp forms $f$ of  weight $k$
for $SL(2, \mathbb Z)$. For $f \in S_k$, let $L(s,f)$ denote the L-function associated to $f$.  For ease of notation, we shall say $L(s, f)$ has the {\it ``super-positivity property''} if its associated cuspidal automorphic representation has the super-positivity property.

Let $H_k$ be the basis of forms of $S_k$ that are
eigenfunctions of all the Hecke operators.  Denote
\[
  \begin{split}
     \cH_{+}(K) & := \{ f\in H_k:\ K\leq k\leq 2K,\ k\equiv0\ (\mod 4) \}, \\
     \cH_{-}(K) & := \{ f\in H_k:\ K\leq k\leq 2K,\ k\equiv2\ (\mod 4) \}.
  \end{split}
\]

\begin{remark} \label{remark1} For $f \in  \cH_{+}(K)$, the form $f$ is even, the sign in the functional equation of $L(s,f)$ is positive, and  $L(s, f)$ must have a zero of even order at $s = 1/2.$ Similarly for  $f \in  \cH_{-}(K)$ the form $f$ is odd, the sign in the functional equation of $L(s,f)$ is negative, and  $L(s, f)$  must have a zero of odd order at $s = 1/2.$
\end{remark}

\begin{theorem}\label{thm:superpositivity2}
The number of $f \in  \cH_{-}(K)$ such that $L(s, f)$ has no zero in the region $s=\sigma+it$ with
$1/2 < \sigma < 1, \; |t| \leq \sigma - 1/2$
is $\gg  K^2/\log K.$
\end{theorem}

\begin{remark}
The proof of theorem \ref{thm:superpositivity2} is based on \cite{conrey2002real}, where Conrey and Soundararajan combined Selberg's lemma \ref{lemma:selberg} with the mollification method to prove that a positive proportion of quadratic Dirichlet L-functions have no zeros in $[0,1].$ The proof in \cite{conrey2002real} requires estimating the total number of zeros for the family of quadratic Dirichlet L-functions in a thin rectangular region. On the other hand, our proof of theorem \ref{thm:superpositivity2} uses many such regions to cover the triangle $1/2 < \sigma < 1, \; |t| \leq \sigma - 1/2$.
\end{remark}

By combining theorems \ref{thm:superpositivity1}, \ref{thm:superpositivity2}, we immediately obtain the following.

\begin{theorem}\label{thm:superpositivity3}
  There are infinitely many odd modular forms $f$ for $SL(2, \mathbb Z)$ such that
  $L(s,f)$ has the super-positivity property. In fact, the number of $f \in  \cH_{-}(K)$ which have the super-positivity property is $\gg  K^2/\log K.$
\end{theorem}

In the course of proving theorem \ref{thm:superpositivity3} we also obtained the following result as a
byproduct.
 \begin{theorem}\label{thm:rz}
  There are infinitely many  modular forms $f$ for $SL(2, \mathbb Z)$ such that
  $L(s,f)$ has no real zeros in the region $\Re(s) > 0$ except at $s = 1/2.$  In fact, the number of $f \in  \cH_{-}(K)$ (respectively  $f \in  \cH_{+}(K)$)  with
 this property is $\gg  K^2/\log K.$

\end{theorem}

\begin{remark}
  We learned from Ricotta \cite[p. 292]{ricotta2006real},
  that in 2003, Soundarajaran announced a result similar to theorem \ref{thm:rz} for $\cH_{+}(K)$ and $\cH_{-}(K)$. Recently, Soundararajan informed us that he and Conrey proved that for $K$ large, at least 71\% of $f \in  \cH_{-}(K)$ and at least 38\% of $f\in \cH_{+}(K)$ satisfy theorem \ref{thm:rz}. However, they never published this result.
 The proof for the case $\cH_{-}(K)$ is given in \S \ref{sec:thmrz}  because it is needed
  for the proof of theorem \ref{thm:superpositivity3}.
  This result is not optimal since we didn't try to remove the harmonic weight and our choice of the mollifier is not optimal.
   One may also give a similar proof for $\cH_{+}(K)$,
 but it is omitted here since this is not needed for our main results.\end{remark}

 Although we focussed on the family of Hecke cusp forms $H_k$ for $SL(2, \mathbb Z)$, the methods introduced in this paper can also be applied to other families of automorphic forms of varying level, weight and spectrum. We plan to continue these investigations in a future research paper.

\section{Proof of theorem \ref{thm:superpositivity1}} \label{sec:proof1}

\begin{proof}

Let $$\lambda(s,\pi) := \Lambda(s+1/2,\pi).$$ Then, as in \cite{Y-Z}, the function $\lambda(s,\pi)$ has the following properties:
\begin{align*}
&\bullet \; \lambda(-s) = \pm \lambda(s),
\\
&\bullet \; \lambda(\sigma) > 0 \; \text{for}\;  \sigma > 1/2,\\
&\bullet \; \lambda(s) \;\text{is entire and of  order 1,}\\
&\bullet \; \text{If}\; \lambda(\pm\rho_k) = 0\; (\text{with $\rho_k\ne0, \;\Im\rho_k \ge 0$}) \;\text{for}\; k=1,2,\ldots, \text{(denotes  the non central zeros of $\lambda(s)$)} \\
 & \phantom{xx}\text{then for some integer $m\ge 0$ and  $A > 0,$}
\end{align*}
\begin{equation}\label{eq:RHassume1}
\lambda(s) = s^m e^{A} \prod_{k=1}^\infty \left(1- \frac{s^2}{\rho_k^2}  \right), \qquad\qquad (\text{\rm Hadamard Product Formula} ).
\end{equation}

Now, let $\beta +i\gamma$ be a non-trivial zero of $\lambda(s)$ with $\beta, \gamma\in \mathbb R.$ Then by our assumptions either
\begin{align*} &\bullet \; \beta = 0 \;\text{and there are two zeros of} \;\lambda(s)\;\text{at} \; s = \pm i\gamma;\\
&\bullet \;\beta > 0, \; |\gamma| > \beta, \; \text{and there are four zeros of} \; \lambda(s)\;  \text{at} \; s = \beta+i\gamma,\,  -\beta+i\gamma,\,  \beta-i\gamma,\, -\beta-i\gamma.
\end{align*}
Consequently, we may rewrite (\ref{eq:RHassume1}) in the form
\begin{equation} \label{eq:RHassume2}
\lambda(s) =  s^m e^A \underset {\beta>0 \;\text{and}\; |\gamma|>\beta}{\prod_{\lambda(\beta+i\gamma)=0}} \left(1+ \frac{(2\gamma^2-2\beta^2) s^2 +s^4}{
(\gamma^2+\beta^2)^2}  \right) \underset{\beta = 0} {\prod_{\lambda(\beta+i\gamma)=0}} \left(1+\frac{s^2}{\gamma^2}   \right)
.\end{equation}

It immediately follows from (\ref{eq:RHassume2}) that all derivatives of $\lambda(s)$ at $s = 0$ must be greater or equal to zero and all derivatives of $\lambda(s)$ at $s = \sigma > 1/2$ must be positive.  Condition (3) of  definition \ref{def:superpositivity} for super-positivity of $\pi$  follows as in \cite{Y-Z}.
\end{proof}

\section{Requisite background material needed for the proof of theorem \ref{thm:superpositivity2} }

\subsection{Selberg's Lemma}

We will need the following version of the argument principle,
which is due to Selberg.

\begin{lemma}\label{lemma:selberg}
 For $W\in \mathbb R$, let $\phi(s)$ be a holomorphic function of a complex variable $s$ that does not vanish on a half-plane $\Re(s)\geq W$.
  Let $\cB$ be the rectangular box of vertices $W_0\pm iH$, $W_1\pm iH$, where $H>0$
  and $W_0<W<W_1$. Then we have
  \[
    \begin{split}
       &   4H\sum_{\substack{\beta+i\gamma\in\cB \\ \phi(\beta+i\gamma)=0}}
             \cos\left(\frac{\pi \gamma}{2H}\right) \sinh\left(\frac{\pi(\beta-W_0)}{2H}\right)
          = \int\limits\limits_{-H}^{H} \cos\left(\frac{\pi t}{2H}\right) \log|\phi(W_0+it)| \;dt \\
         & \hskip 137pt + \int\limits_{W_0}^{W_1} \sinh\left(\frac{\pi(\alpha-W_0)}{2H}\right) \log|\phi(\alpha+iH)\phi(\alpha-iH)| \;d\alpha \\
         & \hskip 167pt - \Re\left(\int\limits_{-H}^{H}\cos\left(\frac{\pi(W_1-W_0+it)}{2iH}\right)(\log \phi)(W_1+it) \; dt\right).
    \end{split}
  \]
\end{lemma}

\begin{proof}
  See Selberg \cite[Lemma 14]{selberg1946contributions2}
  or Conrey--Soundararajan \cite[Lemma 2.1]{conrey2002real}.
\end{proof}

\subsection{An average of the $J$-Bessel function}

\begin{lemma}\label{lemma:average-over-k}
  Suppose $\Phi\in C_0^\infty(\bbR^+)$ is a real valued function and $K\geq1$.
  For $x>0$, we have
  $$4\sum_{k\equiv2(4)}\Phi\left(\frac{k-1}{K}\right)J_{k-1}(x)  =  \Phi\left(\frac{x}{K}\right)
      + \frac{K}{\sqrt{x}} \Im\left(e^{-2\pi i/8}e^{ix}\check{\Phi}\left(\frac{K^2}{2x}\right)\right)
      + \mathcal{O}\left(\frac{x}{K^3}\int\limits_{-\infty}^{\infty}|v|^3|\hat{\Phi}(v)| \; dv\right),
   $$
  where
  $$
    \check{\Phi}(v) := \int_{0}^{\infty} \frac{\Phi(\sqrt{u})}{\sqrt{2\pi u}}e^{iuv}du,
  $$
  and $\hat{\Phi}$ is the Fourier transform of $\Phi$.
  The implied constant is absolute.
\end{lemma}

\begin{proof}
  See Iwaniec \cite[Lemma 5.8]{iwaniec1997topics},
  Iwaniec--Luo--Sarnak \cite[Proposition 8.1]{iwaniec2000low},
  and Khan \cite[Lemma 2.3]{khan2010non-vanishing}.
\end{proof}

\subsection{The approximate functional equation}
Let $f(z) = \sum\limits_{n=1}^\infty \lambda_f(n) (4\pi n)^{\frac{(k-1)}{2}}e^{2\pi i nz}$ (for $z$ in the upper half plane)
be a modular form  of weight $k$ for $SL(2, \mathbb Z)$ with associated L-function
$$L(s, f) := \sum_{n=1}^\infty \frac{\lambda_f(n)}{n^s},\qquad\qquad (\Re(s) > 1).$$

Fix a smooth function $H:\bbR^+\rightarrow\bbR^+$ satisfying
$H(x)=1$ for $x\in[0,1/2]$, and $H(x)+H(1/x)=1$ for $x\in\bbR^+$.
We know the Mellin transform $\widetilde{H}(s)=\int_{0}^{\infty}H(y)y^s\frac{dy}{y}$
has a single simple pole at $0$ of residue $1$, and is odd.
Furthermore, $\widetilde{H}(s)$ satisfies the bounds
$\widetilde{H}(s) \ll_A \frac{1}{|s(s+1)\cdots (s+A-1)|},$ $A=1,2,\ldots$,
and $\widetilde{H}(s)\ll 2^{\Re(s)}$ for $\Re(s)>1$.

\begin{lemma}\label{lemma:AFE}
  Let $K\geq1$, and $k\asymp K$ an even integer.
  Let $-\frac{B}{\log K}\leq \delta\leq \vartheta$ and $t\ll K$.
  Then for any modular form $f$
  of weight $k$ for $SL(2, \mathbb Z)$, we have
  \[
    |L(1/2+\delta+it,f)|^2 = \sum_{d=1}^{\infty}\frac{1}{d^{1+2\delta}}
       \sum_{n=1}^{\infty}\frac{\lambda_f(n)\eta_{it}(n)}{n^{1/2+\delta}}
       V_{k,\delta+it}(nd^2).
  \]
 Here $\eta_\nu(n):=\sum\limits_{ad=n}\left(\frac{a}{d}\right)^\nu$ is the generalized divisor function, and for any $y>0,$
  \[
    V_{k,\delta+it}(y) := \frac{1}{2\pi i} \int\limits_{3-i\infty}^{3+i\infty}
        \frac{\widetilde{H}(s+\delta)+\widetilde{H}(s-\delta)}{(4\pi^2y)^{s-\delta}}
        \frac{\Gamma(s+k/2+it)\Gamma(s+k/2-it)}{\Gamma(\delta+k/2+it)\Gamma(\delta+k/2-it)} \; ds
  \]
  is real valued, and satisfies the following:
  \[
    \begin{split}
      V_{k,\delta+it}(y) & = 1 + (4\pi^2y)^{2\delta}\frac{\Gamma(-\delta+k/2+it)\Gamma(-\delta+k/2-it)}{\Gamma(\delta+k/2+it)\Gamma(\delta+k/2-it)} + \mathcal{O}_{A}\left(\left(\frac{y}{k^2}\right)^{A}\right), \\
      V_{k,\delta+it}(y) & \ll_{A}  \left(\frac{k^2}{y}\right)^{A},
      \quad
      y^j V_{k,\delta+it}^{(j)}(y) \ll_j 1,
    \end{split}
  \]
  for  $A>0$ and any integer $j\geq0$.
  We also have that
  \[
    V_{k,\delta+it}(y) = \frac{1}{2\pi i} \int\limits_{\alpha-i\infty}^{\alpha+i\infty} \frac{\widetilde{H}(s+\delta)+\widetilde{H}(s-\delta)}{(16\pi^2)^{s-\delta}}
        \left(\frac{k^2}{y}\right)^{s-\delta} ds + \mathcal{O}_{\ve}\left(|t|^2y^{-\ve}k^{-1+\ve}\right)
  \]
  for any $\alpha>|\delta|$.
\end{lemma}

\begin{proof}
  See Iwaniec--Kowalski \cite[pp. 97--100]{iwaniec2004analytic}
  and Hough \cite[Proposition 3.7]{hough2012zero}.
\end{proof}

\subsection{The Petersson trace formula}

Each Hecke eigenform $f\in H_k$ has a Fourier expansion
$$f(z) = \sum\limits_{n=1}^\infty \lambda_f(n) (4\pi n)^{\frac{(k-1)}{2}}e^{2\pi i nz}, \qquad (z\in\mathbb C, \; \Im(z) > 0),$$
 where  $\lambda_f(n) \in\mathbb R$ for $n =1,2,\ldots$  We normalize $f$ by setting $\lambda_f(1) = 1$.  The Fourier coefficients of $f$ satisfy the relation
\begin{equation}\label{eqn:HR}
  \lambda_f(m)\lambda_f(n)= \sum_{d|(m,n)}\lambda_f\left(\frac{mn}{d^2}\right).
\end{equation}
The Petersson trace formula is given by the following basic orthogonality relation on $H_k$.

\begin{lemma}\label{lemma:PTF}
  Let $m,n\geq1$. Then
  \[  \sum_{f\in H_k} \omega_f\cdot  \lambda_f(m)\lambda_f(n)
 \;   = \; \delta_{m,n} + 2\pi i^{-k} \sum_{c=1}^{\infty}
                 \frac{S(m,n;c)}{c}J_{k-1}\left(\frac{4\pi\sqrt{mn}}{c}\right).
  \]
  where $$\omega_f = \frac{12\zeta(2)}{(k-1)} \cdot \frac{1}{ L(1,\operatorname{sym}^2 f)  }$$
  is termed the harmonic weight of $f$.
\end{lemma}
\begin{proof}
  See e.g. Iwaniec \cite[Theorem 3.6]{iwaniec1997topics} and
  Blomer--Khan--Young \cite[\S2.1]{blomer2013distribution}.
\end{proof}
By appealing to the well known estimate
$J_{k-1}(x) \ll \min\left(x^{k-1}, \; x^{-\frac12}   \right)$
it easily follows that
\begin{equation}\label{eq:harmonicsum}
\sum_{f\in H_k} \omega_f = 1  \;+ \;\mathcal O\left(2^{-k} \right).
\end{equation}

\subsection{The Voronoi summation formula for Eisenstein series}

Let $\nu\in\mathbb C.$ The generalized divisor function
$$\eta_\nu(n)=\sum\limits_{ad=n}\left(\frac{a}{d}\right)^\nu$$
occurs in the Fourier expansion of Eisenstein series.
We have the following version of the Voronoi summation formula.

\begin{lemma}\label{lemma:VSF}
  Let $g:\bbR^+\rightarrow\bbR^+$ be a smooth and compactly supported function.
  Let $c\geq1$ and $(a,c)=1$ with $ad\equiv1(c)$. Then
  \[
    \begin{split}
          \sum_{n=1}^\infty \eta_{it}(n)  g(n) e^{2\pi i\frac{an}{c}} & = c^{2it-1}\zeta(1-2it)\int_{0}^{\infty}g(x)x^{-it}dx +
             c^{-2it-1}\zeta(1+2it)\int_{0}^{\infty}g(x)x^{it}dx \\
         & \hskip 50pt + \frac{1}{c} \sum_{n=1}^\infty \eta_{it}(n) e\left(-\frac{dn}{c}\right)
                 \int_{0}^{\infty}g(x)J_{2it}^{+}\left(\frac{4\pi\sqrt{nx}}{c}\right)dx \\
         & \hskip 90pt + \frac{1}{c} \sum_{n=1}^\infty \eta_{it}(n) e\left(\frac{dn}{c}\right)
                 \int_{0}^{\infty}g(x)K_{2it}^{+}\left(\frac{4\pi\sqrt{nx}}{c}\right)dx,
    \end{split}
  \]
  where
  \[
    J_\nu^+(x) := \frac{-\pi}{\sin\frac{\pi \nu}{2}}\left(J_\nu(x)-J_{-\nu}(x)\right),
    \quad
    K_\nu^+(x) := 4\cos\frac{\pi \nu}{2} K_\nu(x).
  \]
\end{lemma}

\begin{proof}
  See e.g. Hough \cite[Lemma 3.3]{hough2012zero}.
\end{proof}

\section{The twisted second moment near the critical point}\label{sec:tsm}

Recall that $H_k$ denotes a basis for the space of holomorphic Hecke cusp forms of weight $k\ge 12$ for $SL(2,\mathbb Z).$
Let $\mathcal H = \bigcup_k H_k$.
Assume that for  all $f \in \mathcal H$  there is some uniquely defined $\alpha_f\in\mathbb C.$ Consider the set $$\{\alpha_f\} := \{\alpha_f\}_{f\in \mathcal H}.$$

The basic objects of study for the rest of this paper are given in the following definition.

\begin{definition} \label{defSums} Let
$$\omega_f := \frac{12\zeta(2)}{(k-1)}\cdot \frac{1}{ L(1,\operatorname{sym}^2 f)}$$ denote the harmonic weight of $f\in H_k.$
Let $\Phi:\mathbb R\to\mathbb R_{\ge 0}$ be a fixed smooth non-negative function supported on $[1,2]$ and
let $$\mathcal R_{1/2} := \big\{\beta+i\gamma \; \big | \; \beta \in (1/2,1) \; \text{and} \; |\gamma|\leq\beta - 1/2\big\}.$$

  For $K > 0,$
 define the following sums
\begin{align*}
& \mathcal A\big(\{\alpha_f\};  K, \Phi\big) \; :=  \sum_{k\equiv2(4)}\Phi\left(\frac{k-1}{K}\right)   \sum_{f\in H_k}  \omega_f \cdot\alpha_f,\\
& \mathcal A\big(K,\Phi\big) \; := \sum_{k\equiv2(4)}\Phi\left(\frac{k-1}{K}\right)  \sum_{f\in H_k} \omega_f,\\
& \cM(K, \Phi) \; := \sum_{k\equiv2(4)}\Phi\left(\frac{k-1}{K}\right)  \underset{L(s,f)\, \ne \,0 \; \text{for} \; s \,\in \,\mathcal R_{1/2}} {\sum_{f\in H_k}} \hskip-20pt \omega_f,\\
& \cN(K, \Phi) \; := \sum_{k\equiv2(4)}\Phi\left(\frac{k-1}{K}\right) \underset{ \text{one zero in}  \,\mathcal R_{1/2}} {\underset{L(s,f)\; \text{has at least} } {\sum_{f\in H_k}}}  \hskip -10pt\omega_f.
\end{align*}

\end{definition}

It is clear that
$$
  \cM(K;\Phi)+\cN(K;\Phi)=\cA(K;\Phi).
$$
The key strategy for proving  theorem \ref{thm:superpositivity2} is to try to show that $\cM(K, \Phi)$ is large compared to
$\mathcal A\big(K,\Phi\big)$.
To achieve this goal we will use the mollification method
which leads us to first consider the following twisted second moment of $L(s, f)$ at the special value $s = 1/2+\delta+it$.

\begin{theorem}\label{thm:tsm}
  Let $-\frac{B}{\log K}\leq \delta \leq \vartheta$ and $t\ll K^{\theta}$, with
  $0<2\theta\leq \vartheta\leq 1/100$ two small positive constants.
  Let $\ell\leq K^{2-4\vartheta}$. 
  We have the following asymptotic formula.
 \begin{align*}
       &  \cA\Big(\left\{\lambda_f(\ell)|L(1/2+\delta+it,f)|^2\right\}; \, K, \,\Phi\Big) \\
       & \hskip 40pt = \zeta(1+2\delta) \frac{\eta_{it}(\ell)}{\ell^{1/2+\delta}} \frac{K}{4} \int\limits_{0}^{\infty}\Phi(u)du
             + \zeta(1-2\delta) \frac{\eta_{it}(\ell)}{\ell^{1/2-\delta}} \left(\frac{K}{4\pi}\right)^{-4\delta}
              \frac{K}{4} \int\limits_{0}^{\infty} \Phi(u)u^{-4\delta}du  \\
       & \hskip 120pt - 2\Re\Bigg\{ \zeta(1+2it)\frac{\eta_\delta(\ell)}{\ell^{1/2+it}} \left(\frac{K}{4\pi}\right)^{-2\delta+2it}
                \frac{K}{4}\int\limits_{0}^{\infty}\Phi(u)u^{-2\delta+2it} \; du \Bigg\}\\
       & \hskip 200pt + \; \mathcal{O}_{\ve}\left((1+|t|)^2\ell^{-\delta}K^{\ve} + (1+|t|)^4\ell^{1/2}K^{-1+\ve}\right).
    \end{align*}

\end{theorem}

We now begin the proof of theorem \ref{thm:tsm}.
From the approximate functional equation, we have
\begin{align*}
       &   \cA\Big(\left\{\lambda_f(\ell)|L(1/2+\delta+it,f)|^2\right\}; \, K, \,\Phi\Big) = \sum_{k\equiv2(4)}\Phi\left(\frac{k-1}{K}\right)
            \underset{f\in H_k}{{\sum}}\omega_f\cdot \lambda_f(\ell)|L(1/2+\delta+it,f)|^2 \\
       &\hskip 55pt = \sum_{k\equiv2(4)}\Phi\left(\frac{k-1}{K}\right) \sum_{d=1}^{\infty}
       \sum_{n=1}^{\infty}\frac{\eta_{it}(n)}{d^{1+2\delta}n^{1/2+\delta}}
        V_{k,\delta+it}(nd^2) \underset{f\in H_k}{{\sum}}\omega_f\cdot \lambda_f(\ell)\lambda_f(n).
\end{align*}

Applying the Petersson trace formula we obtain that
\begin{equation}\label{eqn:D+F}
  \cA\Big(\left\{\lambda_f(\ell)|L(1/2+\delta+it,f)|^2\right\}; \, K, \,\Phi\Big) = \cD + \cF,
\end{equation}
where we have the diagonal term
\begin{equation*}
  \cD := \frac{\eta_{it}(\ell)}{\ell^{1/2+\delta}} \sum_{d=1}^{\infty}\frac{1}{d^{1+2\delta}}
         \sum_{k\equiv2(4)}\Phi\left(\frac{k-1}{K}\right) V_{k,\delta+it}(\ell d^2),
\end{equation*}
and the off-diagonal term
\begin{equation*}
  \begin{split}
    \cF & :=  -2\pi \sum_{d=1}^{\infty}
         \sum_{n=1}^{\infty}\frac{\eta_{it}(n)}{d^{1+2\delta}n^{1/2+\delta}}
         \sum_{c=1}^{\infty}\frac{S(n,\ell;c)}{c} \\
        & \qquad\qquad\qquad \cdot \sum_{k\equiv2(4)}\Phi\left(\frac{k-1}{K}\right) V_{k,\delta+it}(nd^2)
         J_{k-1}\left(\frac{4\pi\sqrt{\ell n}}{c}\right).
  \end{split}
\end{equation*}

\subsection{The diagonal term}

From now on, we let $\vartheta$ be a fixed positive real number less than $1/100$.
We first handle the case $-\frac{B}{\log K}\leq \delta \leq \vartheta$
and $\delta\neq0$. Note that for the remaining case $\delta=0$,
we can just view it as the limitation of $\delta\rightarrow0$.
Introducing the integral defining $V$, we have
\[
  \begin{split}
    \cD & =  \frac{\eta_{it}(\ell)}{\ell^{1/2+\delta}} \frac{1}{2\pi i}
            \int\limits_{\delta+\ve -i\infty}^{\delta+\ve +i\infty} \frac{\zeta(1+2s)}{(4\pi^2\ell)^{s-\delta}}
            \left[\widetilde{H}(s-\delta)+\widetilde{H}(s+\delta)\right] \\
        & \hskip 80pt\cdot \sum_{k\equiv2(4)}\Phi\left(\frac{k-1}{K}\right) \frac{\Gamma(s+k/2+it)\Gamma(s+k/2-it)}{\Gamma(\delta+k/2+it)\Gamma(\delta+k/2-it)} \; ds.
  \end{split}
\]
Since we have
$$\left|\frac{\Gamma(s+k/2+it)\Gamma(s+k/2-it)}
  {\Gamma(\delta+k/2+it)\Gamma(\delta+k/2-it)}\right|
  \ll
  \frac{\Gamma(\delta+\ve+k/2)\Gamma(\delta+\ve+k/2)}
  {\Gamma(\delta+k/2)\Gamma(\delta+k/2)}\ll k^{2\ve}$$
for $\Re(s)=\delta+\ve$ and rapid decay of $\widetilde{H}(s-\delta)+\widetilde{H}(s+\delta)$
on the vertical line $\Re(s)=\delta+\ve$,
we can restrict the integral above to $|\Im(s)|\leq K^\ve$ with
an error of $O_{\ve,B}(\ell^{-1/2-\delta}K^{-B})$.
For $\Re(s)=\delta+\ve$ and $|\Im(s)|\leq K^\ve$,
it follows from Stirling's formula that
\[
  \frac{\Gamma(z+u)}{z} = z^u \left(1+ \mathcal O\left(\frac{|u|^2}{|z|}\right)\right),
\]
so we have (recalling that $t\ll K^\theta$)
\[
  \frac{\Gamma(s+k/2+it)\Gamma(s+k/2-it)}{\Gamma(\delta+k/2+it)\Gamma(\delta+k/2-it)}
  = \left(\frac{k-1}{2}\right)^{2(s-\delta)}  \Big(1+ \mathcal O\big((1+|t|)^2 k^{-1+2\ve}\big)\Big);
\]
and then, together with Poisson summation formula
(cf. Iwaniec--Kowalski \cite[eq. (4.24)]{iwaniec2004analytic})
we obtain
  \begin{align*}
       & \sum_{k\equiv2(4)}\Phi\left(\frac{k-1}{K}\right) \frac{\Gamma(s+k/2+it)\Gamma(s+k/2-it)}{\Gamma(\delta+k/2+it)\Gamma(\delta+k/2-it)} \\
       & \hskip 80pt = \sum_{k\equiv2(4)}\Phi\left(\frac{k-1}{K}\right) \left(\frac{k-1}{2}\right)^{2(s-\delta)}
             \Big(1+ \mathcal O\left((1+|t|)^2k^{-1+2\ve}\right)\Big) \\
       & \hskip 80pt = \left(\frac{K}{2}\right)^{2(s-\delta)} \sum_{k\equiv2(4)}\Phi\left(\frac{k-1}{K}\right)
             \left(\frac{k-1}{K}\right)^{2(s-\delta)}  +\;\;\; \mathcal O\Big((1+|t|)^2K^{4\ve}\Big) \\
       &  \hskip 80pt = \left(\frac{K}{2}\right)^{2(s-\delta)}
             \left(\frac{K}{4} \int\limits_{0}^{\infty}\Phi(u)u^{2(s-\delta)}du + O_B\left(K^{-B}\right)\right)
             + \; \mathcal O\Big((1+|t|)^2K^{4\ve}\Big)\\
       & \hskip 80pt =  \left(\frac{K}{2}\right)^{2(s-\delta)} \frac{K}{4} \int\limits_{0}^{\infty}\Phi(u)u^{2(s-\delta)}du
             + \; \mathcal O\Big((1+|t|)^2K^{4\ve}\Big).
  \end{align*}

Hence we have
\[
  \begin{split}
    \cD & = \; \frac{\eta_{it}(\ell)}{\ell^{1/2+\delta}} \frac{1}{2\pi i}
            \int\limits_{\delta+\ve-iK^\ve}^{\delta+\ve+iK^\ve} \frac{\zeta(1+2s)}{(4\pi^2\ell)^{s-\delta}}
            \left[\widetilde{H}(s-\delta)+\widetilde{H}(s+\delta)\right] \\
        & \hskip 30pt \cdot \left[\left(\frac{K}{2}\right)^{2(s-\delta)} \frac{K}{4} \int\limits_{0}^{\infty}\Phi(u)u^{2(s-\delta)}du
             + \mathcal O\Big((1+|t|)^2 K^{4\ve}\Big)\right] ds
          + \; \mathcal O_{\ve,B}\Big(\ell^{-1/2-\delta}K^{-B}\Big) \\
        & = \; \frac{\eta_{it}(\ell)}{\ell^{1/2+\delta}} \frac{K}{4}
             \frac{1}{2\pi i} \int\limits_{\delta+\ve-iK^\ve}^{\delta+\ve+iK^\ve}
             \frac{\zeta(1+2s)}{(4\pi^2\ell)^{s-\delta}}
            \left[\widetilde{H}(s-\delta)+\widetilde{H}(s+\delta)\right]  \\
        & \hskip 30pt \cdot \left(\frac{K}{2}\right)^{2(s-\delta)}
             \int\limits_{0}^{\infty}\Phi(u)u^{2(s-\delta)} \;du\, ds
           \; + \; \mathcal O_\ve\Big((1+|t|)^2\ell^{-1/2-\delta}K^{4\ve}\Big)
          \;  + \; \mathcal O_{\ve,B}\Big(\ell^{-1/2-\delta}K^{-B}\Big) \\
        & = \; \frac{\eta_{it}(\ell)}{\ell^{1/2+\delta}} \frac{K}{4}
             \frac{1}{2\pi i} \int\limits_{\delta+\ve-i\infty}^{\delta+\ve-i\infty}\frac{\zeta(1+2s)}{(4\pi^2\ell)^{s-\delta}}
            \left[\widetilde{H}(s-\delta)+\widetilde{H}(s+\delta)\right]  \\
        & \hskip 80pt \cdot \left(\frac{K}{2}\right)^{2(s-\delta)}
             \int\limits_{0}^{\infty}\Phi(u)u^{2(s-\delta)} \;du \,ds
             +\; \mathcal  O_\ve\Big((1+|t|)^2\ell^{-1/2-\delta}K^{4\ve}\Big).
  \end{split}
\]
By shifting the contour to the line $\Re(s)=-1/2+\delta+\ve$, we obtain
\begin{equation}\label{eqn:cD=}
  \begin{split}
    \cD & = \zeta(1+2\delta) \frac{\eta_{it}(\ell)}{\ell^{1/2+\delta}} \frac{K}{4} \int\limits_{0}^{\infty}\Phi(u) \;du
           \;  + \; \zeta(1-2\delta) \frac{\eta_{it}(\ell)}{\ell^{1/2-\delta}} \left(\frac{K}{4\pi}\right)^{-4\delta}
              \frac{K}{4} \int\limits_{0}^{\infty} \Phi(u)u^{-4\delta}\;du \\
        & \hskip -5pt  +\;\mathcal  O\left(\frac{\tau(\ell ) K}{\ell^{1/2+\delta}}  \left | \,\int\limits_{-\frac12+\delta+\ve-i\infty}^{-\frac12+\delta+\ve+i\infty}
                \frac{\zeta(1+2s)}{(4\pi \ell)^{s-\delta}}
                \left[\widetilde{H}(s-\delta)+\widetilde{H}(s+\delta)\right]  \cdot \left(\frac{K}{2}\right)^{2(s-\delta)}
                 \int\limits_{0}^{\infty}\Phi(u)u^{2(s-\delta)}\; du \,ds\,\right |\;\right)\\
                 &
                  \hskip 200pt  + \; \mathcal O\Big((1+|t|)^2\ell^{-1/2-\delta}K^{4\ve}\Big) \\
        &\hskip -10pt = \zeta(1+2\delta) \frac{\eta_{it}(\ell)}{\ell^{1/2+\delta}} \frac{K}{4} \int\limits_{0}^{\infty}\Phi(u)\; du
           \;  + \; \zeta(1-2\delta) \frac{\eta_{it}(\ell)}{\ell^{1/2-\delta}} \left(\frac{K}{4\pi}\right)^{-4\delta}
              \frac{K}{4} \int\limits_{0}^{\infty} \Phi(u)u^{-4\delta}\;du  \\
                 &
                  \hskip 200pt  + \; \mathcal O_\ve\Big((1+|t|)^2\ell^{-\delta}K^{4\ve}\Big).
  \end{split}
\end{equation}

\subsection{The off-diagonal term}
We assume $t\neq0$. The case $t=0$, can be viewed  as the
limitation of the case $t\neq0$. Of course, we can also use
the Voronoi summation formula for $\tau(n)$, the divisor function,
(cf. Iwaniec--Kowalski \cite[eq. (4.49)]{iwaniec2004analytic})
instead of lemma \ref{lemma:VSF} to do the estimation.
Note that $|J_k(x)|\leq (x/2)^{k-1}$ for $k>4$, and $0<x<1$.
By lemma \ref{lemma:AFE}, we can truncate the sums over $c$
and $d$ in $\cF$ at $cd\leq \sqrt{\ell}K^{1+\ve}$ with a negligible error.
(Indeed, to prove this we consider two cases depending on the size of $nd^2$.
For the case $nd^2>K^{2+\ve}$, by lemma \ref{lemma:AFE}, we know
$V_{k,\delta+it}(nd^2)\ll \frac{1}{n^2d^4}K^{-B}$;
and then we just use the uniform bound
$J_{k-1}(\frac{4\pi\sqrt{\ell n}}{c})\ll k^{-1/3}$ when $c\leq 4\pi\sqrt{\ell n}$,
and we use the bound
$J_{k-1}(\frac{4\pi\sqrt{\ell n}}{c})\ll \frac{\ell n}{c^2}$ when $c>4\pi\sqrt{\ell n}$,
so we can bound the total contribution by $K^{-B'}$.
For the case $nd^2\leq K^{2+\ve}$, since $cd\geq\sqrt{\ell}K^{1+\ve}$,
we have $\frac{4\pi\sqrt{\ell n}}{c}\ll K^{-\ve/2}$.
So we have $J_{k-1}(\frac{4\pi\sqrt{\ell n}}{c})\ll \frac{\ell n}{c^2} K^{-B}$,
and again the contribution to $\cF$ will be bounded by $K^{-B'}$.)

Let $x=\frac{4\pi\sqrt{\ell n}}{c}$.
By lemma \ref{lemma:average-over-k}, we have
\[
  \begin{split}
      &  4\sum_{k\equiv2(4)}\Phi\left(\frac{k-1}{K}\right)  V_{k,\delta+it}(nd^2) J_{k-1}\left(x\right) \\
      & \hskip 30pt= \Phi\left(\frac{x}{K}\right) V_{x+1,\delta+it}(nd^2)
          \;  + \; \frac{K}{\sqrt{x}} \Im\left(e^{-2\pi i/8}e^{ix} \int\limits_{0}^{\infty}\frac{\Phi(\sqrt{u})}{\sqrt{2\pi u}} \, V_{\sqrt{u}K+1,\delta+it}(nd^2) \, e^{iu\frac{K^2}{2x}}\; du\right)  \\
      & \hskip 150pt + \; \mathcal O\left(\frac{x}{K^3}\int_{-\infty}^{\infty}|v|^3 \left|\int\limits_{0}^{\infty}\Phi(u) \, V_{uK+1,\delta+it}(nd^2) \, e^{iuv}\; du\right|dv\right).
  \end{split}
\]
Hence
\begin{equation}\label{eqn:FtoFE}
  \cF = \cF_1 + \cF_2 + \cE \; + \; \mathcal O\left(K^{-B}\right),
\end{equation}
where
 $$
     \cF_1  := -\frac{\pi}{2} \sum_{cd\leq \sqrt{\ell}K^{1+\ve}}
                \frac{1}{cd^{1+2\delta}}
             \underset{(a,c)=1}  {\sum_{a=1}^c}               e^{\frac{2\pi i a\ell}{c}}
                \sum_{n=1}^{\infty}  \frac{\eta_{it}(n) e^{\frac{2\pi i\bar{a}n}{c}} }{n^{1/2+\delta}} \Phi\bigg(\frac{4\pi\sqrt{\ell n}}{cK}\bigg)
                V_{\frac{4\pi\sqrt{\ell n}}{c}+1,\delta+it}(nd^2),
                $$

  \begin{align*}
     \cF_2 & := -\frac{\pi^{1/2}K}{4\ell^{1/4}} \sum_{cd\leq \sqrt{\ell}K^{1+\ve}}
                \frac{1}{c^{1/2}d^{1+2\delta}}
               \underset{(a,c)=1}  {\sum_{a=1}^c}
               e^{\frac{2\pi i a\ell}{c}}
                \sum_{n=1}^{\infty}
                \frac{\eta_{it}(n) e^{\frac{2\pi i \bar{a}n}{c}}}{n^{3/4+\delta}} \\
           &  \hskip 126pt \cdot
                \Im\left(e^{-2\pi i/8} \, e^{i\frac{4\pi\sqrt{\ell n}}{c}}
                \int\limits_{0}^{\infty}\frac{\Phi(\sqrt{u})}{\sqrt{2\pi u}} \,
                V_{\sqrt{u}K+1,\delta+it}(nd^2) \, e^{iu\frac{cK^2}{8\pi\sqrt{\ell n}}}\; du\right),
                \end{align*}

   $$  \cE   \ll \frac{\ell^{1/2}}{K^3} \sum_{d=1}^{\infty} \sum_{n=1}^{\infty}\frac{\tau(n)}{d^{1+2\delta}n^{\delta}}
         \sum_{\substack{c\leq \sqrt{\ell}K^{1+\ve}/d}}\frac{|S(n,\ell;c)|}{c^2}  \cdot \int\limits_{-\infty}^{\infty}|v|^3 \left|\int\limits_{0}^{\infty}\Phi(u) \,V_{uK+1,\delta+it}(nd^2) \,e^{iuv}\; du\right|dv.
$$

\subsubsection{The estimate of $\cF_1$}

As in lemma \ref{lemma:AFE}, by Stirling's formula, for $x\asymp K$ we have
\[
  \begin{split}
    V_{x+1,\delta+it}(y)  & = \frac{1}{2\pi i} \int\limits_{\alpha-i\infty}^{\alpha+i\infty}
        \frac{\widetilde{H}(s+\delta)+\widetilde{H}(s-\delta)}{(16\pi^2)^{s-\delta}}
        \left(\frac{x^2}{y}\right)^{s-\delta}\left(1+\frac{2(s^2-\delta^2)}{x}\right) ds\\
                 &  \hskip 200pt
                  + \; \mathcal  O\left((1+|t|)^4y^{-\ve}K^{-2+\ve}\right)
  \end{split}
\]
for any $\alpha>|\delta|$.
By lemma \ref{lemma:AFE}, we can restrict the $c$-sum and $d$-sum
to $cd\leq \sqrt{\ell}K^{\ve}$ with a negligible error.
(Indeed, since $h$ has compact support, we only need to sum over $n$
which satisfies the conditon $\sqrt{\ell n}/(cK)\asymp1$.
So for the sum over $c,d$ with $cd\geq \sqrt{\ell}K^{\ve}$,
we have $nd^2\asymp (cdK)^2/\ell \gg K^{2+2\ve}$.
Note that by lemma \ref{lemma:AFE}, we have
$V_{\frac{4\pi\sqrt{\ell n}}{c}+1,\delta+it}(nd^2) \ll (K^2/(nd^2))^{B} \ll K^{-B'}(nd^2)^{-3}$.
Hence this contributes to $\cF_1$ with $O(K^{-B''})$ for any $B''>0$,
here the implied constant depends on both $\ve$ and $B''$.)
And then by using Weil's bound for the Kloosterman sum,
the contribution of the error term of
$V_{\frac{4\pi\sqrt{\ell n}}{c}+1,\delta+it}(nd^2)$
as above to $\cF_1$ is
\begin{equation*}
  \begin{split}
    & \ll (1+|t|)^4 K^{-2+\ve} \sum_{d=1}^{\infty}\frac{1}{d^{1+2\delta}}
        \sum_{\substack{c\\ cd\ll \sqrt{\ell}K^{\ve}}} \frac{(c,\ell)^{1/2}}{c^{1/2}}
        \sum_{n\asymp c^2K^2/\ell} \frac{\tau(n)}{n^{1/2+\delta}} \\
    & \ll (1+|t|)^4 \ell^{-1/2}K^{-1+\ve} \sum_{d\leq\sqrt{\ell}K^\ve}\frac{1}{d^{1+2\delta}}
        \sum_{\substack{c\ll \sqrt{\ell}K^{\ve}}} (c,\ell)^{1/2}c^{1/2} \\
    & \ll (1+|t|)^4 \ell^{1/4}K^{-1+2\ve} \sum_{c\ll \sqrt{\ell}K^{\ve}} (c,\ell)^{1/2} \\
    & \ll (1+|t|)^4 \ell^{1/4}K^{-1+2\ve} \sum_{c\ll \sqrt{\ell}K^{\ve}} (c,\ell)\\
    & \ll (1+|t|)^4 \ell^{1/4}K^{-1+2\ve} \sum_{d|\ell}d\sum_{\substack{d|c\\ c\ll \sqrt{\ell}K^{\ve}}}1
      \ll (1+|t|)^4 \ell^{1/4}K^{-1+3\ve}.
  \end{split}
\end{equation*}
Hence we have
\[
  \begin{split}
    \cF_1 & = -\frac{\pi}{2} \sum_{cd\leq \sqrt{\ell}K^\ve}
                \frac{1}{cd^{1+2\delta}}
                 \underset{(a,c)=1}  {\sum_{a=1}^c}               e^{\frac{2\pi i a\ell}{c}}
                \sum_{n=1}^{\infty}\eta_{it}(n) e^{\frac{2\pi i\bar{a}n}{c}} g_{c,d}(n)
            \;   + \; \mathcal O\Big((1+|t|)^4 \ell^{1/4}K^{-1+3\ve}\Big) \\
          & = -\frac{\pi}{2} \sum_{d=1}^{\infty}\sum_{c=1}^{\infty}
                \frac{1}{cd^{1+2\delta}}
               \underset{(a,c)=1}  {\sum_{a=1}^c}               e^{\frac{2\pi i a\ell}{c}}
                \sum_{n=1}^{\infty}\eta_{it}(n) e^{\frac{2\pi i\bar{a}n}{c}} g_{c,d}(n)
            \;   + \; \mathcal O\Big((1+|t|)^4 \ell^{1/4}K^{-1+3\ve}\Big),
  \end{split}
\]
where
\[
  \begin{split}
    g_{c,d}(y)  := \frac{1}{y^{1/2+\delta}} \Phi\left(\frac{4\pi\sqrt{\ell y}}{cK}\right)
            \frac{1}{2\pi i} \int\limits_{3-i\infty}^{3+i\infty} \left[\widetilde{H}(s+\delta)+\widetilde{H}(s-\delta)\right]
        \left(\frac{\ell}{c^2d^2}\right)^{s-\delta}
        \left(1+\frac{c(s^2-\delta^2)}{2\pi\sqrt{\ell y}}\right) ds.
  \end{split}
\]
By lemma \ref{lemma:VSF}, we have
\begin{equation}\label{eqn:F1toMJK1}
  \cF_1 = \cM_1 + \cJ_1 + \cK_1 + \mathcal{O}\Big((1+|t|)^4 \ell^{1/4}K^{-1+3\ve}\Big),
\end{equation}
where
\begin{equation*}
  \begin{split}
    \cM_1 & := -\frac{\pi}{2}  \zeta(1-2it) \sum_{d=1}^{\infty}\sum_{c=1}^{\infty}
                \frac{S(0,\ell;c)}{c^{2-2it}d^{1+2\delta}}  \int\limits_{0}^{\infty}g_{c,d}(y)y^{-it}dy \\
          & \hskip 100pt -\frac{\pi}{2} \zeta(1+2it) \sum_{d=1}^{\infty}\sum_{c=1}^{\infty}
                \frac{S(0,\ell;c)}{c^{2+2it}d^{1+2\delta}} \int\limits_{0}^{\infty}g_{c,d}(y)y^{it}dy, \\
    \cJ_1 & := -\frac{\pi}{2} \sum_{d=1}^{\infty}\sum_{c=1}^{\infty} \sum_{n=1}^{\infty}
                \frac{\eta_{it}(n)S(0,\ell-n;c)}{c^2d^{1+2\delta}}
                \int\limits_{0}^{\infty} g_{c,d}(y) J_{2it}^+\left(\frac{4\pi\sqrt{ny}}{c}\right)dy, \\
    \cK_1 & := -\frac{\pi}{2} \sum_{d=1}^{\infty}\sum_{c=1}^{\infty} \sum_{n=1}^{\infty}
                \frac{\eta_{it}(n)S(0,\ell+n;c)}{c^2d^{1+2\delta}}
                \int\limits_{0}^{\infty} g_{c,d}(y) K_{2it}^+\left(\frac{4\pi\sqrt{ny}}{c}\right)dy.
  \end{split}
\end{equation*}

We first deal with $\cM_1$. Note that
\[
  \cM_1 = -\pi \Re\left\{ \zeta(1-2it) \sum_{d=1}^{\infty}\sum_{c=1}^{\infty}
     \frac{S(0,\ell;c)}{c^{2-2it}d^{1+2\delta}}  \int\limits_{0}^{\infty}g_{c,d}(y)y^{-it}dy \right\}.
\]
Introducing the definition of $g_{c,d}(y)$, and making a substitution
$u=\frac{4\pi\sqrt{\ell y}}{cK}$, we have
\[
  \begin{split}
    \cM_1 & = -\pi \Re\left \{ \zeta(1-2it)
            \sum_{d=1}^{\infty}\sum_{c=1}^{\infty}
            \frac{S(0,\ell;c)}{c^{2-2it}d^{1+2\delta}}
            \int\limits_{0}^{\infty}\frac{1}{y^{1/2+\delta+it}}
            \Phi\left(\frac{4\pi\sqrt{\ell y}}{cK}\right) \right.\\
          & \hskip 70pt
          \left.
          \cdot \frac{1}{2\pi i} \int\limits_{3-i\infty}^{3+i\infty}
             \left[\widetilde{H}(s+\delta)+\widetilde{H}(s-\delta)\right]
             \left(\frac{\ell}{c^2d^2}\right)^{s-\delta}
             \left(1+\frac{c(s^2-\delta^2)}{2\pi\sqrt{\ell y}}\right)\, ds\, dy \right\} \\
          & = -\pi \Re\left \{ \zeta(1-2it) \frac{2K^{1-2\delta-2it}}{(4\pi)^{1-2\delta-2it}\ell^{1/2-\delta-it}}
            \int\limits_{0}^{\infty}\Phi(u)u^{-2\delta-2it} \right.  \\
          & \hskip 30pt \left.
           \cdot \frac{1}{2\pi i} \int\limits_{3-i\infty}^{3+i\infty}
             \big[\widetilde{H}(s+\delta)+\widetilde{H}(s-\delta)\big]
             \ell^{s-\delta}
             \bigg(1+\frac{2(s^2-\delta^2)}{Ku}\bigg)
             \sum_{d=1}^{\infty}\sum_{c=1}^{\infty}
            \frac{S(0,\ell;c)}{c^{1+2s}d^{1+2s}} ds \,du \right \}.
  \end{split}
\]
Note that for $\Re(s)>0$ and $\ell\geq1$ we have
\begin{equation}\label{eqn:GenSer-Ram}
  \sum_{d=1}^{\infty}\sum_{c=1}^{\infty}
  \frac{S(0,\ell;c)}{c^{1+2s}d^{1+2s}}
 \; = \; \ell^{-s}\eta_s(\ell).
\end{equation}
Hence we get
\[
  \begin{split}
    \cM_1 & = -\pi \Re\left \{ \zeta(1-2it) \frac{2K^{1-2\delta-2it}}{(4\pi)^{1-2\delta-2it}\ell^{1/2-it}}
            \int\limits_{0}^{\infty}\Phi(u)u^{-2\delta-2it} \right.  \\
          & \hskip 80pt \left.
           \cdot \frac{1}{2\pi i} \int\limits_{3-i\infty}^{3+i\infty}
             \left[\widetilde{H}(s+\delta)+\widetilde{H}(s-\delta)\right] \eta_s(\ell)
             \left(1+\frac{2(s^2-\delta^2)}{Ku}\right) ds \, du \right\}.
  \end{split}
\]
For the innermost integral, we change $s$ to $-s$.
Recall that $\widetilde{H}(-s)=-\widetilde{H}(s)$ and $\eta_s(\ell)=\eta_{-s}(\ell)$.
So that for any $u\asymp1$ we have
  \begin{align*}
       & \frac{1}{2\pi i} \int\limits_{3-i\infty}^{3+i\infty}
             \left[\widetilde{H}(s+\delta)+\widetilde{H}(s-\delta)\right] \eta_s(\ell)
             \left(1+\frac{2(s^2-\delta^2)}{Ku}\right) ds \\
       &\hskip 80pt = -\frac{1}{2\pi i} \int\limits_{-3+i\infty}^{-3-i\infty}
             \left[\widetilde{H}(-s+\delta)+\widetilde{H}(-s-\delta)\right] \eta_{-s}(\ell)
             \left(1+\frac{2(s^2-\delta^2)}{Ku}\right) ds\\
       & \hskip 80pt = -\frac{1}{2\pi i} \int\limits_{-3-i\infty}^{-3+i\infty}
             \left[\widetilde{H}(s+\delta)+\widetilde{H}(s-\delta)\right] \eta_{s}(\ell)
             \left(1+\frac{2(s^2-\delta^2)}{Ku}\right) ds \\
       & \hskip 80pt = \frac{1}{2}\frac{1}{2\pi i} \left( \int\limits_{3-i\infty}^{3+i\infty} -  \int\limits_{-3-i\infty}^{-3+i\infty}\right)
             \left[\widetilde{H}(s+\delta)+\widetilde{H}(s-\delta)\right] \eta_{s}(\ell)
             \left(1+\frac{2(s^2-\delta^2)}{Ku}\right) ds\\
             &\\
       & \hskip 80pt = \eta_\delta(\ell).
  \end{align*}

Therefore, we get
\begin{equation}\label{eqn:cM_1=}
  \begin{split}
    \cM_1 & = - \frac{1}{2}\Re \left\{ \zeta(1-2it)\eta_\delta(\ell) \frac{K^{1-2\delta-2it}}{(4\pi)^{-2\delta-2it}\ell^{1/2-it}}
            \int\limits_{0}^{\infty}\Phi(u)u^{-2\delta-2it} du \right\}. \\
  \end{split}
\end{equation}

To handle the $\cJ_1$ and $\cK_1$ terms, we need some asymptotic evaluations
regarding the Bessel functions
(cf. Hough \cite[eq. (3.5) and (3.6)]{hough2012zero})
\begin{equation}\label{eqn:J&K}
  \begin{split}
     J_\nu^+(x) & = -\sqrt{\frac{2\pi}{x}}\sin(x-\pi/4)\left[1-\frac{16\nu^4-40\nu^2+9}{128x^2}\right] \\
                & \hskip 90pt - \pi \cos(x-\pi/4)\frac{\nu^2-1/4}{2x} \; + \; \mathcal O\left(\frac{1+|\nu|^6}{x^3}\right), \\
     K_\nu(x)   & = \sqrt{\frac{\pi}{2x}} e^{-x} \left[1+ \mathcal O\left(\frac{1+|\nu|^2}{x}\right)\right].
  \end{split}
\end{equation}
Substituting $u=\frac{4\pi\sqrt{\ell y}}{cK}$ we obtain
\[
  \begin{split}
    \cJ_1 & = -\pi \left(\frac{K}{4\pi}\right)^{1-2\delta} \frac{1}{\ell^{1/2-\delta}}
                \sum_{d=1}^{\infty}\sum_{c=1}^{\infty} \sum_{n=1}^{\infty}
                \frac{\eta_{it}(n)S(0,\ell-n;c)}{c^{1+2\delta}d^{1+2\delta}}
                \int\limits_{0}^{\infty}  \Phi(u) u^{-2\delta}
                J_{2it}^+\left(\frac{\sqrt{n}Ku}{\sqrt{\ell}}\right) \\
          & \hskip 90pt \cdot \frac{1}{2\pi i}
           \int\limits_{3-i\infty}^{3+i\infty} \left[\widetilde{H}(s+\delta)+\widetilde{H}(s-\delta)\right]
              \left(\frac{\ell}{c^2d^2}\right)^{s-\delta}\left(1+\frac{2(s^2-\delta^2)}{Ku}\right) ds \, du, \\
    \cK_1 & = -\pi \left(\frac{K}{4\pi}\right)^{1-2\delta} \frac{1}{\ell^{1/2-\delta}}
                \sum_{d=1}^{\infty}\sum_{c=1}^{\infty} \sum_{n=1}^{\infty}
                \frac{\eta_{it}(n)S(0,\ell+n;c)}{c^{1+2\delta}d^{1+2\delta}}
                \int\limits_{0}^{\infty}  \Phi(u) u^{-2\delta}
                K_{2it}^+\left(\frac{\sqrt{n}Ku}{\sqrt{\ell}}\right) \\
          & \hskip 90pt \cdot \frac{1}{2\pi i}   \int\limits_{3-i\infty}^{3+i\infty}\left[\widetilde{H}(s+\delta)+\widetilde{H}(s-\delta)\right]
              \left(\frac{\ell}{c^2d^2}\right)^{s-\delta}\left(1+\frac{2(s^2-\delta^2)}{Ku}\right) ds \, du.
  \end{split}
\]
Since $S(0,0;c)=\phi(c)$ where $\phi$ is the Euler function,
and
\begin{equation}\label{eqn:GenSer-Eul}
  \sum_{d=1}^{\infty}\sum_{c=1}^{\infty}
  \frac{\phi(c)}{c^{1+2s}d^{1+2s}}
  = \zeta(2s), \quad \quad  (\textrm{for}\;\Re(s)>1/2),
\end{equation}
together with \eqref{eqn:GenSer-Ram}, we have
$$\cJ_1  = \cJ_{11}+\cJ_{12},$$
and
\begin{equation} \label{eqn:cJ_1=}
  \begin{split}
    \cK_1 & = -\pi \left(\frac{K}{4\pi}\right)^{1-2\delta} \frac{1}{\ell^{1/2-\delta}}
                \sum_{n=1}^{\infty} \eta_{it}(n)  \int\limits_{0}^{\infty}  \Phi(u) u^{-2\delta}
                K_{2it}^+\left(\frac{\sqrt{n}Ku}{\sqrt{\ell}}\right) \\
          & \qquad\qquad \cdot \frac{1}{2\pi i}  \int\limits_{3-i\infty}^{3+i\infty} \big[\widetilde{H}(s+\delta)+\widetilde{H}(s-\delta)\big]
              \frac{\ell^{s-\delta}}{(\ell+n)^{s}}\eta_{s}(\ell+n)
              \bigg(1+\frac{2(s^2-\delta^2)}{Ku}\bigg) ds \, du,
  \end{split}
\end{equation}
where
\begin{equation*}
  \begin{split}
    \cJ_{11} & := -\pi \left(\frac{K}{4\pi}\right)^{1-2\delta} \frac{1}{\ell^{1/2-\delta}}
                \sum_{n\neq \ell} \eta_{it}(n) \int\limits_{0}^{\infty}  \Phi(u) u^{-2\delta}
                J_{2it}^+\left(\frac{\sqrt{n}Ku}{\sqrt{\ell}}\right) \\
          & \qquad\qquad \cdot \frac{1}{2\pi i}  \int\limits_{3-i\infty}^{3+i\infty} \left[\widetilde{H}(s+\delta)+\widetilde{H}(s-\delta)\right]
              \ell^{s-\delta}\sum_{a|(\ell-n)}\frac{1}{a^{2s}}\left(1+\frac{2(s^2-\delta^2)}{Ku}\right) ds \, du, \\
    \cJ_{12} & := -\pi \left(\frac{K}{4\pi}\right)^{1-2\delta} \frac{\eta_{it}(\ell)}{\ell^{1/2-\delta}}
                 \int\limits_{0}^{\infty}  \Phi(u) u^{-2\delta} J_{2it}^+(Ku) \\
          & \qquad\qquad \cdot \frac{1}{2\pi i} \int\limits_{3-i\infty}^{3+i\infty} \left[\widetilde{H}(s+\delta)+\widetilde{H}(s-\delta)\right]
              \ell^{s-\delta}\zeta(2s)\left(1+\frac{2(s^2-\delta^2)}{Ku}\right) ds \, du.
  \end{split}
\end{equation*}
Note that $\ell\leq K^{2-2\vartheta}$ and the $K$-Bessel function is exponentially small for large variable.
It follows that $\cK_1$ is extremely small.

Now we consider $\cJ_{11}$.  By shifting the contour of the
integral in $\cJ_{11}$ to $\Re(s)=\delta+\ve$, the contribution coming from
the error of the asymptotic expansion of
$J_{2it}^+\left(\frac{\sqrt{n}Ku}{\sqrt{\ell}}\right)$ is
\begin{equation}\label{eqn:cJ_110<<}
  \begin{split}
     &  \ll (1+|t|)^6 \left(\frac{K}{\sqrt{\ell}}\right)^{1-2\delta}
                \sum_{n\neq \ell} \tau(n)  \int\limits_{0}^{\infty}  \Phi(u) u^{-2\delta}
                \left(\frac{\sqrt{n}Ku}{\sqrt{\ell}}\right)^{-3} \\
     & \qquad\qquad\cdot \left|\frac{1}{2\pi i} \int\limits_{\delta+\ve-i\infty}^{\delta+\ve+i\infty} \left[\widetilde{H}(s+\delta)+\widetilde{H}(s-\delta)\right]
              \ell^{s-\delta}\sum_{a|(\ell-n)}\frac{1}{a^{2s}}
              \left(1+\frac{2(s^2-\delta^2)}{Ku}\right) ds\right|\, du,\\
     & \ll (1+|t|)^6 \ell K^{-2+\ve}. 
  \end{split}
\end{equation}
And similarly, by shifting the contour of the
integral in $\cJ_{12}$ to $\Re(s)=1/2+\ve$, the contribution coming from
the error of the asymptotic expansion of
$J_{2it}^+(Ku)$ is $\ll (1+|t|)^6 \ell^{-1/2}K^{-2+\ve}$.
We only show how to bound the contribution from the main term
\begin{equation}\label{eqn:J^+1}
  -\left({\frac{2\pi}{Ku}}\right)^{1/2}\left(\frac{\ell}{n}\right)^{1/4}
  \sin\left(Ku\left(\frac{n}{\ell}\right)^{1/2}-\frac{\pi}{4}\right)
\end{equation}
to $\cJ_{11}$; the rest of the main terms can be handled in the same way,
and their contribution is smaller.
The contribution to $\cJ_{11}$ from integrating
against \eqref{eqn:J^+1} is
\[
  \begin{split}
     \cJ_{111}  & :=  \pi \left(\frac{K}{4\pi}\right)^{1-2\delta} \frac{1}{\ell^{1/2-\delta}}
                \left({\frac{2\pi}{K}}\right)^{1/2}\ell^{1/4}
                \sum_{n\neq \ell} \frac{\eta_{it}(n)}{n^{1/4}} \int\limits_{0}^{\infty}  \Phi(u) u^{-1/2-2\delta}
                \sin\left(K\left(\frac{n}{\ell}\right)^{1/2}u-\frac{\pi}{4}\right) \\
       & \hskip 80pt\cdot \frac{1}{2\pi i}
       \int\limits_{3-i\infty}^{3+i\infty} \left[\widetilde{H}(s+\delta)+\widetilde{H}(s-\delta)\right]
              \ell^{s-\delta}\sum_{a|(\ell-n)}\frac{1}{a^{2s}}\left(1+\frac{2(s^2-\delta^2)}{Ku}\right) ds \,du.
  \end{split}
\]
Define
$$
  g_n(u):=\frac{1}{2\pi i}  \int\limits_{3-i\infty}^{3+i\infty} \left[\widetilde{H}(s+\delta)+\widetilde{H}(s-\delta)\right]
              \ell^{s-\delta}\sum_{a|(\ell-n)}\frac{1}{a^{2s}}\left(1+\frac{2(s^2-\delta^2)}{Ku}\right) ds.
$$
Then by integrating by parts twice, we obtain
\[
  \begin{split}
     \cJ_{111}  & \ll  \left(\frac{K}{\sqrt{\ell}}\right)^{1/2-2\delta}
                \sum_{n\neq \ell} \frac{\tau(n)}{n^{1/4}}\; \left|\int\limits_{0}^{\infty}  \Phi(u) u^{-1/2-2\delta}g_n(u)
                \sin\left(K\left(\frac{n}{\ell}\right)^{1/2}u-\frac{\pi}{4}\right)du \,\right| \\
                & \ll  \left(\frac{K}{\sqrt{\ell}}\right)^{-3/2}
                \sum_{n\neq \ell} \frac{\tau(n)}{n^{5/4}} \;\left|\int\limits_{0}^{\infty}
                \left( \Phi(u) u^{-1/2-2\delta}g_n(u)\right)''
                \sin\left(K\left(\frac{n}{\ell}\right)^{1/2}u-\frac{\pi}{4}\right)du \,\right|.
  \end{split}
\]
By shifting the contour of the integral in $g_n(u)$ to the line $\Re(s)=\delta+\ve$,
and then bounding everything trivially, we get
\begin{equation}\label{eqn:cJ_111<<}
  \cJ_{111} \ll \ell^{3/4}K^{-3/2+\ve}.
\end{equation}
Hence, by \eqref{eqn:F1toMJK1}, \eqref{eqn:cM_1=}, \eqref{eqn:cJ_1=}, \eqref{eqn:cJ_110<<},
and \eqref{eqn:cJ_111<<}, we obtain that
\begin{equation}\label{eqn:cF_1=}
  \begin{split}
    \cF_1 & = - \frac{1}{2}\Re \left\{ \zeta(1-2it)\eta_\delta(\ell)
            \frac{K^{1-2\delta-2it}}{(4\pi)^{-2\delta-2it}\ell^{1/2-it}}
            \int\limits_{0}^{\infty}\Phi(u)u^{-2\delta-2it} du \right\} \\
          & \hskip 100pt + \mathcal O\left((1+|t|)^4\ell^{1/4}K^{-1+3\ve}+\ell^{3/4}K^{-3/2+\ve} + (1+|t|)^6 \ell K^{-2+\ve}\right).
  \end{split}
\end{equation}

\subsubsection{The treatment of $\cF_2$ and $\cE$}
Define
\[
  \cW_{K}(y,v)= \cW_{K}(y,v,\Phi) := \int\limits_{0}^{\infty}\frac{\Phi(\sqrt{u})}{\sqrt{2\pi u}}\, V_{\sqrt{u}K+1,\delta+it}(y) \,e^{iuv} \;du,
\]
for $y\geq1$ and $0<v<K^5$.
Using the integral definition of $V_{k,\delta+it}$ we have
\[
  \begin{split}
    \cW_{K}(y,v) & = \frac{1}{2\pi i} \int\limits_{\alpha-i\infty}^{\alpha+i\infty}
        \frac{\widetilde{H}(s+\delta)+\widetilde{H}(s-\delta)}{(4\pi^2y)^{s-\delta}} \\
       &\qquad\qquad\qquad \cdot  \int\limits_{0}^{\infty} \,\frac{\Gamma(s+\frac{\sqrt{u}K+1}{2}+it)\Gamma(s+\frac{\sqrt{u}K+1}{2}-it)}
        {\Gamma(\delta+\frac{\sqrt{u}K+1}{2}+it)\Gamma(\delta+\frac{\sqrt{u}K+1}{2}-it)} \,
        \frac{\Phi(\sqrt{u})}{\sqrt{2\pi u}} \,e^{iuv} \;du\, ds,
  \end{split}
\]
for any $\alpha>|\delta|$. Taking $\alpha=A+\delta$ for any
large number $A>1$. For $|\Im(s)|\geq K^\ve$, then by the properties
of $\widetilde{H}$ we have $\cW_K(y,v)\ll \big(\frac{K^2}{y}\big)^A K^{-B}$ for
any $B>0$. And for $|\Im(s)|\leq K^\ve$, by integrating by parts several times
and asymptotic formulas for Gamma function and the polygamma functions,
we have
\[
  \int\limits_{0}^{\infty} \, \frac{\Gamma(s+\frac{\sqrt{u}K+1}{2}+it)\Gamma(s+\frac{\sqrt{u}K+1}{2}-it)}
        {\Gamma(\delta+\frac{\sqrt{u}K+1}{2}+it)\Gamma(\delta+\frac{\sqrt{u}K+1}{2}-it)}
      \,  \frac{\Phi(\sqrt{u})}{\sqrt{2\pi u}} \,e^{iuv} \;du \;
  \ll_{A,B} \; \frac{1+|s|}{v^B}K^{2A}.
\]
Hence we obtain
\begin{equation}\label{eqn:W_k<<}
  \cW_K(y,v) \;\ll_{A,B} \; \left(\frac{K^2}{y}\right)^A v^{-B}.
\end{equation}
Thus we can truncate the sum over $n$ in $\cF_2$
at $nd^2\leq K^{2+\ve}$ with a negligible error.
Note that $\ell\leq K^{2-2\vartheta}$, so we have
$\frac{cK^2}{8\pi\sqrt{\ell n}}\gg K^{\vartheta/2}$.
Hence by \eqref{eqn:W_k<<}, for any $B>0$ we get
\begin{equation}\label{eqn:cF_2<<}
  \cF_2  \;\ll_B \; K^{-B}.
\end{equation}
Similarly, we can show that
\[
  \int\limits_{-\infty}^{\infty}|v|^3 \left|\int\limits_{0}^{\infty}\Phi(u)V_{uK+1,\delta+it}(y) e^{iuv}du\right|dv
\;  \ll_{A} \; \left(\frac{K^2}{y}\right)^A.
\]
So we can truncate the sum over $n$ and $d$ in $\cE$ at $nd^2\leq K^{2+\ve}$
with a negligible error. Hence we get
\begin{equation}\label{eqn:cE<<}
  \cE \; \ll \; \frac{\ell^{1/2}}{K^3} \sum_{d=1}^{\infty}
         \sum_{nd^2\leq K^{2+\ve}}^{\infty}\frac{\tau(n)}{d^{1+2\delta}n^{\delta}}
         \sum_{\substack{c\\cd\leq \sqrt{\ell}K^{1+\ve}}}\frac{|S(n,\ell;c)|}{c^2}\;
         +\; \mathcal  O\left(K^{-B}\right)
     \; \ll \; \ell^{1/2}K^{-1+2\ve}.
\end{equation}
Combining \eqref{eqn:FtoFE}, \eqref{eqn:cF_1=}, \eqref{eqn:cF_2<<}, and \eqref{eqn:cE<<}, we have
\begin{equation}\label{eqn:cF=}
  \begin{split}
    \cF & = - \frac{1}{2}\Re \left\{ \zeta(1-2it)\eta_\delta(\ell)
            \frac{K^{1-2\delta-2it}}{(4\pi)^{-2\delta-2it}\ell^{1/2-it}}
            \int\limits_{0}^{\infty}\Phi(u)u^{-2\delta-2it} du \right\} \; + \; \mathcal O\left((1+|t|)^4\ell^{1/2}K^{-1+\ve}\right).
  \end{split}
\end{equation}
Using \eqref{eqn:D+F}, \eqref{eqn:cD=}, and \eqref{eqn:cF=},
this completes the proof of theorem \ref{thm:tsm}.

\section{Mollification near the critical point}\label{sec:msm}

\subsection{Choosing the mollifier}

We will take the same mollifier as in Hough \cite[\S5]{hough2012zero}.
Let
$$L(s,f)^{-1} \; := \;\sum_{n=1}^\infty \frac{a_f(n)}{n^s}, \qquad (\Re(s) > 1).$$
 The coefficients $a_f(n)$  are supported
on cube-free numbers. Now, for $m,n$ square-free, $(m,n)=1$,
we have $a_f(mn^2)=\mu(m)\lambda_f(m)$.

Fix $0 < \Upsilon < 1$ and  $M\leq K^{1-2\vartheta}$. Define
\begin{equation*}
  F_{\Upsilon,M}(x) := \left\{\begin{array}{ll}
                   1, & \textrm{if}\ 0\leq x\leq M^{1-\Upsilon}, \\
                   P\left(\frac{\log(M/x)}{\log M}\right), & \textrm{if}\ M^{1-\Upsilon}\leq x\leq M, \\
                   0, & \textrm{if}\ x\geq M,
                 \end{array}\right.
\end{equation*}
with $P(t)$, a degree 3 polynomial, that satisfies $P(\Upsilon)=1$ and $P'(\Upsilon)=P(0)=P'(0)=0$.

 We define the mollifier for
$L(s,f)$ by
\begin{equation}\label{eqn:M(s)}
  M(s,f) := \sum_{n=1}^{\infty} \frac{a_f(n)F_{\Upsilon,M}(\rad(n))}{n^s}
          = \sum_{\ell=1}^{\infty} \frac{x_\ell(s)}{\ell^s} \lambda_f(\ell),
\end{equation}
where $\rad(n)$ denotes the   product of the distinct prime numbers dividing $n$ and
\begin{equation*}
  x_\ell(s) := \sum_{n=1}^{\infty}\frac{\mu^2(\ell n)\mu(\ell)F_{\Upsilon,M}(\ell n)}{n^{2s}}.
\end{equation*}

We now set $\omega:=\delta+it$.
By \eqref{eqn:HR}, we have
\begin{equation}\label{eqn:|M|^2}
  \begin{split}
     |M(1/2+\omega,f)|^2
       & =  \sum_{\ell_1=1}^\infty\sum_{\ell_2=1}^\infty \frac{x_{\ell_1}(1/2+\omega)x_{\ell_2}(1/2+\bar{\omega})}
            {\ell_1^{1/2+\omega}\ell_2^{1/2+\bar{\omega}}}
            \lambda_{f}(\ell_1)\lambda_f(\ell_2) \\
       & =  \sum_{\ell_1=1}^\infty\sum_{\ell_2=1}^\infty  \sum_{d=1}^\infty \frac{1}{d^{1+\omega+\bar{\omega}}} \frac{x_{d\ell_1}(1/2+\omega)x_{d\ell_2}(1/2+\bar{\omega})}
            {\ell_1^{1/2+\omega}\ell_2^{1/2+\bar{\omega}}}
            \lambda_{f}(\ell_1\ell_2).
  \end{split}
\end{equation}
We may always write $L(s,f)M(s,f)$ as $LM(s,f)$.
Hence we get
  \begin{align*}
    &  \cA\Big(\left\{|LM(1/2+\omega,f)|^2\right\};\; K, \;\Phi\Big)\\
    & \hskip 30pt=  \sum_{\ell_1=1}^\infty\sum_{\ell_2=1}^\infty  \sum_{d=1}^\infty \frac{1}{d^{1+\omega+\bar{\omega}}}
            \frac{x_{d\ell_1}(1/2+\omega)x_{d\ell_2}(1/2+\bar{\omega})}
            {\ell_1^{1/2+\omega}\ell_2^{1/2+\bar{\omega}}}
            \cA\Big(\left\{\lambda_f(\ell_1\ell_2)|L(1/2+\omega,f)|^2\right\}; \; K, \;\Phi\Big).
  \end{align*}

There are three cases we need to consider when  $\omega=\delta+it$:
\begin{equation*}
  \mathbf{(I)}\left\{\begin{split}
       & \frac{-B}{\log K}\leq \delta \leq \frac{C\log\log K}{\log K}
       \quad\textrm{and}\quad |\delta|\geq \frac{C}{\log K \log\log K}, \\
       & \frac{C}{\log K \log\log K}\leq |t| \leq \frac{C\log\log K}{\log K};
    \end{split}\right. 
\end{equation*}
\begin{equation*}
  \mathbf{(II)}\left\{\begin{split}
       & |\delta|\leq \frac{C}{\log K \log\log K}, \\
       & \frac{C_1}{\log K} \leq |t| \leq \frac{C_2}{\log K};
    \end{split}\right. \qquad \qquad\qquad\qquad
    \mathbf{(III)}\left\{\begin{split}
       & \frac{A}{\log K} \leq |\delta|\leq \frac{C\log\log K}{\log K} , \\
       & |t|\leq \frac{C}{\log K \log\log K},
    \end{split}\right.
\end{equation*}
where $A$, $B$, and $C$ are some constants.
We will focus on the first case $ \mathbf{(I)}$, which we will assume
in this section from now on. Note that the other cases
can be handled by combining the method of
Conrey--Soundararajan \cite[\S6]{conrey2002real}.
Define
\begin{equation}\label{eqn:nu}
  \nu_{(\alpha,\beta)}(\ell):= \frac{\eta_{(\alpha-\beta)/2}(\ell)}{\ell^{(\alpha+\beta)/2}},
\end{equation}
and
\begin{equation}\label{eqn:ve}
  \ve(\alpha,\beta) := \left\{\begin{array}{ll}
                               \hskip 18pt 1, & \qquad \mathrm{if}\ (\alpha,\beta)=(\omega,\bar{\omega}), \\
                                -\left(\frac{K}{4\pi}\right)^{-2\omega}, & \qquad\mathrm{if}\ (\alpha,\beta)=(-\omega,\bar{\omega}), \\
                                -\left(\frac{K}{4\pi}\right)^{-2\bar{\omega}}, & \qquad\mathrm{if}\ (\alpha,\beta)=(\omega,-\bar{\omega}), \\
                              \hskip 10pt  \left(\frac{K}{4\pi}\right)^{-2\omega-2\bar{\omega}}, & \qquad\mathrm{if}\ (\alpha,\beta)=(-\omega,-\bar{\omega}).
                              \end{array}\right.
\end{equation}
By theorem \ref{thm:tsm}, we obtain
\begin{equation}\label{eqn:AtoS}
  \begin{split}
    &\cA\Big(\left\{|LM(1/2+\omega,f)|^2\right\};\; K, \;\Phi\Big) \\
    & =  \sum_{\ell_1=1}^\infty\sum_{\ell_2=1}^\infty  \sum_{d=1}^\infty \frac{1}{d^{1+\omega+\bar{\omega}}}
            \frac{x_{d\ell_1}(1/2+\omega)x_{d\ell_2}(1/2+\bar{\omega})}
            {\ell_1^{1/2+\omega}\ell_2^{1/2+\bar{\omega}}} \cdot \left[     \Bigg( \zeta(1+2\delta) \frac{\eta_{it}(\ell_1\ell_2)}{(\ell_1\ell_2)^{1/2+\delta}}
            \frac{K}{4} \int\limits_{0}^{\infty}\Phi(u)du     \right. \\
    & \qquad\qquad\quad\quad\quad  + \zeta(1-2\delta) \frac{\eta_{it}(\ell_1\ell_2)}{(\ell_1\ell_2)^{1/2-\delta}} \left(\frac{K}{4\pi}\right)^{-4\delta}
              \frac{K}{4} \int\limits_{0}^{\infty} \Phi(u)u^{-4\delta}du  \\
    & \qquad\quad \left. - 2\Re \Bigg\{\zeta(1-2it)\frac{\eta_\delta(\ell_1\ell_2)}{(\ell_1\ell_2)^{1/2-it}}
                \left(\frac{K}{4\pi}\right)^{-2\delta-2it}
                \frac{K}{4}\int\limits_{0}^{\infty}\Phi(u)u^{-2\delta-2it} du \Bigg\} \right] \; + \; \mathcal O\Big(MK^{\ve}\Big) \\
    & = \sum_{(\alpha,\beta)=(\pm\omega,\pm\bar{\omega})}\Psi(\alpha,\beta)\cV_{(\alpha,\beta)}(\omega) \;
            + \; \mathcal O\Big(MK^{\ve}\Big),
  \end{split}
\end{equation}
\vskip 4pt
\noindent
where for any $(\alpha,\beta)=(\pm\omega,\pm\bar{\omega})$,
\begin{equation}\label{eqn:Psi}
  \Psi(\alpha,\beta) := \frac{K}{4} \zeta(1+\alpha+\beta)\, \widetilde{\Phi} \,(1-\omega-\bar{\omega}+\alpha+\beta)\,\varepsilon(\alpha,\beta),
\end{equation}
and
\begin{equation}\label{eqn:cV}
  \cV_{(\alpha,\beta)}(\omega) := \sum_{\ell=1}^\infty\nu_{(\alpha,\beta)}(\ell)
  \sum_{\ell_1\ell_2=\ell} \;\sum_{d=1}^\infty \frac{1}{d^{1+\omega+\bar{\omega}}}
            \frac{x_{d\ell_1}(1/2+\omega)x_{d\ell_2}(1/2+\bar{\omega})}
            {\ell_1^{1+\omega}\ell_2^{1+\bar{\omega}}}.
\end{equation}

By the multiplicative property of the divisor function
\begin{equation}\label{eqn:eta-multi}
  \eta_\nu(mn) = \sum_{d|(m,n)}\mu(d)\eta_{\nu}(m/d)\eta_{\nu}(n/d),
\end{equation}
we can now obtain the following lemma.
\begin{lemma}
  Let $\omega\in\bbC$, and let $(\alpha,\beta)=(\pm\omega,\pm\bar{\omega})$.
  We have
  \begin{equation*}
    \cV_{(\alpha,\beta)}(\omega) = \sum_{r=1}^\infty\frac{1}{r^{1+\omega+\bar{\omega}}}
    \tau_{(\alpha,\beta)}(r) S_{(\alpha,\beta)}(r;\omega)S_{(\alpha,\beta)}(r;\bar{\omega}),
  \end{equation*}
  where, for $z\in\{\omega,\bar{\omega}\}$,
  \begin{equation}\label{eqn:tau&S}
    \begin{split}
       \tau_{(\alpha,\beta)}(r) & := \sum_{d|r} \frac{\mu(d)}{d^{1+\alpha+\beta}}, \\
       S_{(\alpha,\beta)}(r;z) & := \sum_{\ell=1}^\infty \frac{\nu_{(\alpha,\beta)}(\ell)x_{\ell r}(1/2+z)}{\ell^{1+z}}.
    \end{split}
  \end{equation}
\end{lemma}

\begin{proof}
  By \eqref{eqn:eta-multi} we have
  \[
    \begin{split}
          \quad\ \cV_{(\alpha,\beta)}(\omega)
         & = \sum_{\ell_1=1}^\infty\sum_{\ell_2=1}^\infty  \sum_{d=1}^\infty \; \sum_{a|(\ell_1,\ell_2)}
            \frac{\mu(a)}{d^{1+\omega+\bar{\omega}}}
            \frac{\eta_{(\alpha-\beta)/2}(\ell_1/a)\eta_{(\alpha-\beta)/2}(\ell_2/a)}
            {(\ell_1\ell_2)^{(\alpha+\beta)/2}}
            \frac{x_{d\ell_1}(1/2+\omega)x_{d\ell_2}(1/2+\bar{\omega})}
            {\ell_1^{1+\omega}\ell_2^{1+\bar{\omega}}} \\
         & =  \sum_{\ell_1=1}^\infty\sum_{\ell_2=1}^\infty  \sum_{d=1}^\infty\sum_{a|\ell_1}\sum_{a|\ell_2}
            \frac{\mu(a)}{d^{1+\omega+\bar{\omega}}a^{\alpha+\beta}}
            \nu_{(\alpha,\beta)}(\ell_1/a)\nu_{(\alpha,\beta)}(\ell_2/a)
            \frac{x_{d\ell_1}(1/2+\omega)x_{d\ell_2}(1/2+\bar{\omega})}
            {\ell_1^{1+\omega}\ell_2^{1+\bar{\omega}}} \\
         & = \sum_{d=1}^\infty\sum_{a=1}^\infty\sum_{\ell_1=1}^\infty\sum_{\ell_2=1}^\infty
            \frac{\mu(a)}{d^{1+\omega+\bar{\omega}}a^{2+\omega+\bar{\omega}+\alpha+\beta}}
            \nu_{(\alpha,\beta)}(\ell_1)\nu_{(\alpha,\beta)}(\ell_2)
            \frac{x_{ad\ell_1}(1/2+\omega)x_{ad\ell_2}(1/2+\bar{\omega})}
            {\ell_1^{1+\omega}\ell_2^{1+\bar{\omega}}} \\
         & = \sum_{r=1}^\infty\frac{1}{r^{1+\omega+\bar{\omega}}}
            \sum_{a|r}\frac{\mu(a)}{a^{1+\alpha+\beta}}
            S_{(\alpha,\beta)}(r;\omega)S_{(\alpha,\beta)}(r;\bar{\omega}).
    \end{split}
  \]
  This completes the proof of the lemma.
\end{proof}

For any $(\alpha,\beta)=(\pm\omega,\pm\bar{\omega})$, we define
\begin{equation}\label{eqn:cV<>}
  \begin{split}
     \cV_{(\alpha,\beta)}^{\leq}(\omega) & := \sum_{1\leq r\leq M^{1-\Upsilon}}
        \frac{1}{r^{1+\omega+\bar{\omega}}}
        \tau_{(\alpha,\beta)}(r) S_{(\alpha,\beta)}(r;\omega)S_{(\alpha,\beta)}(r;\bar{\omega}), \\
     \cV_{(\alpha,\beta)}^{>}(\omega) & := \sum_{M^{1-\Upsilon}<r\leq M}
        \frac{1}{r^{1+\omega+\bar{\omega}}}
        \tau_{(\alpha,\beta)}(r) S_{(\alpha,\beta)}(r;\omega)S_{(\alpha,\beta)}(r;\bar{\omega}),
  \end{split}
\end{equation}
and we refer to these as the summation of the \emph{short-range} and \emph{long-range terms}, respectively.
Note that we have $S_{(\alpha,\beta)}(r;z)=0$ if $r>M$, for $z\in\{\omega,\bar{\omega}\}$.

\subsection{Treatment of $S_{(\alpha,\beta)}(r;z)$}

For $z\in\{\omega,\bar{\omega}\}$, any integer $r\geq1$,
any real $y>0$, any complex number $s\in\bbC$ with $\Re(s)>3|\delta|$,
and any polynomial $R$, we define
\begin{equation}\label{eqn:T(s)}
  T_{(\alpha,\beta)}(s;r;z) := \sum_{\ell\geq1} \frac{\nu_{(\alpha,\beta)}(\ell)}{\ell^{1+s+z}}
  \sum_{n\geq1} \frac{\mu^{2}(\ell nr)\mu(\ell r)}{n^{1+s+2z}},
\end{equation}
\begin{equation}\label{eqn:T(yR)}
  T_{(\alpha,\beta)}(y,R;r;z) := \sum_{\ell\geq1} \frac{\nu_{(\alpha,\beta)}(\ell)}{\ell^{1+z}}
  \sum_{1\leq n\leq y/(\ell r)} \frac{\mu^{2}(\ell nr)\mu(\ell r)}{n^{1+2z}} R\left(\frac{\log(y/(\ell nr))}{\log y}\right).
\end{equation}
Let
\begin{equation}\label{eqn:Q}
  Q(x) := 1-P\Big(\Upsilon+(1-\Upsilon)x\big).
\end{equation}
We remark that
\begin{equation}\label{eqn:S=T+T}
  S_{(\alpha,\beta)}(r;z) = T_{(\alpha,\beta)}(M,P;r;z)
    +T_{(\alpha,\beta)}(M^{1-\Upsilon},Q;r;z),
\end{equation}
if $r\leq M^{1-\Upsilon}$; and
\begin{equation}\label{eqn:S=T}
  S_{(\alpha,\beta)}(r;z) = T_{(\alpha,\beta)}(M,P;r;z),
\end{equation}
if $M^{1-\Upsilon}<r\leq M$.

\begin{lemma} \label{lemma:T}
  Let $z\in\{\omega,\bar{\omega}\}$ and
  $(\alpha,\beta)=(\pm\omega,\pm\bar{\omega})$.
 For 
  $\Re(s)>-\delta+|\delta|$, we have
  \[
    T_{(\alpha,\beta)}(s;r;z) = \frac{\mu(r) \, G_{(\alpha,\beta)}(s;r;z) \, \zeta(1+s+2z)}
            {\zeta(1+s+z+\alpha) \, \zeta(1+s+z+\beta)},
  \]
  where $G_{(\alpha,\beta)}(s;r;z):=\prod_p G_{p,(\alpha,\beta)}(s;r;z)$ with
  \[
    G_{p,(\alpha,\beta)}(s;r;z):= \left\{\begin{array}{ll}
                                           \left(1-p^{-1-s-z-\alpha}\right)^{-1}
                                           \left(1-p^{-1-s-z-\beta}\right)^{-1}
                                           \left(1-p^{-1-s-2z}\right), & \textrm{if}\ p|r, \\
                                           &
                                           \\
                                           \begin{array}{c}
                                             \left(1-p^{-1-s-z-\alpha}\right)^{-1}
                                             \left(1-p^{-1-s-z-\beta}\right)^{-1}
                                             \left(1-p^{-1-s-2z}\right) \\
                                             \cdot \left(1-p^{-1-s-z-\alpha}- p^{-1-s-z-\beta}+ p^{-1-s-2z}\right),
                                           \end{array} & \textrm{if}\ p\nmid r,
                                         \end{array}\right.
  \]
  so that $G_{(\alpha,\beta)}(s;r;z)$ is holomorphic in $\Re(s)>-1/2+\max\{-2\Re(\omega),0\}$.
  Let $1\leq y\leq M$ and $R$ be a polynomial with $R(0)=R'(0)=0$. Set
  $$E(r):=\prod_{p|r}(1+1/\sqrt{p}).$$
  For any positive integer $r\leq y$ we have
  \[
    \begin{split}
      T_{(\alpha,\beta)}(y,R;r;z) & = \underset{s=0}{\Res} \;\frac{\mu(r) \, G_{(\alpha,\beta)}(s;r;z) \,\zeta(1+s+2z)}
                {s\zeta(1+s+z+\alpha) \, \zeta(1+s+z+\beta)}
                \sum_{j=0}^{\infty} \frac{1}{(s\log y)^j}
                R^{(j)}\left(\frac{\log y/r}{\log y}\right) \\
       &\hskip -46pt + \delta_{z\not\in\{\alpha,\beta\}} \frac{\mu(r) \, G_{(\alpha,\beta)}(-2z;r;z)}
                {\zeta(1-z+\alpha) \,\zeta(1-z+\beta)}
                \left(\frac{y}{r}\right)^{-2z}
                \sum_{j=2}^{\infty} \frac{R^{(j)}(0)}{(-2z)^{j+1} \log^j y} \;  + \; \mathcal O\left(\frac{E(r)}{\log^2 y}e^{-A_0\sqrt{\log y/r}}\right),
    \end{split}
  \]
  where $\delta_{z\not\in\{\alpha,\beta\}}$ is $1$ if $z\not\in\{\alpha,\beta\}$,
  and $0$, otherwise.
\end{lemma}

\begin{proof}
  The first equation follows from comparing two Euler products.
  Indeed, we have
  \[
    \begin{split}
      T_{(\alpha,\beta)}(s;r;z) & = \underset{m\geq1}{{\sum}^\flat} \frac{1}{m^{1+s+z}}
        \sum_{\ell n=m} \frac{\nu_{(\alpha,\beta)}(\ell)\, \mu^{2}(\ell nr) \, \mu(\ell r)}{n^{z}} \\
        & = \mu(r) \underset{\substack{m\geq1\\ (m,r)=1}}{{\sum}^\flat} \frac{1}{m^{1+s+z}}
        \sum_{\ell n=m} \frac{\eta_{(\alpha-\beta)/2}(\ell) \,\mu(\ell)}{\ell^{(\alpha+\beta)/2}n^{z}} \\
        & = \mu(r) \prod_{p\nmid r} \left(1-\frac{1}{p^{1+s+z+\alpha}}- \frac{1}{p^{1+s+z+\beta}}+\frac{1}{p^{1+s+2z}}\right),
    \end{split}
  \]
  where $\flat$ means that we sum over square-free integers.
  To prove the other statement, we will use the following identity
  \[
    \frac{1}{2\pi i} \int\limits_{3-i\infty}^{3+i\infty} y^s\frac{ds}{s^{k+2}} \;
    = \;\delta_{y\geq1} \frac{\log^{k+1}(y)}{(k+1)!},
  \]
  which is standard using suitable contour shifts.
  By the Taylor expansion $R(x)=\sum\limits_{j=2}^{\infty}\frac{R^{(j)}(0)}{j!}x^j$,
  (here we use the assumption $R(0)=R'(0)=0$), we see that
  \[
    \begin{split}
       T_{(\alpha,\beta)}(y,R;r;z) & = \sum_{\ell\geq1} \frac{\nu_{(\alpha,\beta)}(\ell)}{\ell^{1+z}}
                \sum_{1\leq n\leq y/(\ell r)} \frac{\mu^{2}(\ell nr)\mu(\ell r)}{n^{1+2z}} \sum_{j=2}^{\infty} \frac{R^{(j)}(0)}{j!}\left(\frac{\log(y/(\ell nr))}{\log y}\right)^j \\
         & = \sum_{j=2}^{\infty} \frac{R^{(j)}(0)}{j!\log^j y} \sum_{\ell\geq1}
                \frac{\nu_{(\alpha,\beta)}(\ell)}{\ell^{1+z}} \sum_{1\leq n\leq y/(\ell r)}
                \frac{\mu^{2}(\ell nr)\mu(\ell r)}{n^{1+2z}} \log^j(y/(\ell nr)) \\
         & = \sum_{j=2}^{\infty} \frac{R^{(j)}(0)}{\log^j y} \frac{1}{2\pi i}  \int\limits_{3-i\infty}^{3+i\infty} \sum_{\ell\geq1}
                \frac{\nu_{(\alpha,\beta)}(\ell)}{\ell^{1+s+z}}
                \sum_{n\geq1} \frac{\mu^{2}(\ell nr)\mu(\ell r)}{n^{1+s+2z}}
                \left(\frac{y}{r}\right)^s \frac{ds}{s^{j+1}}  \\
         & = \sum_{j=2}^{\infty} \frac{R^{(j)}(0)}{\log^j y} \frac{1}{2\pi i}  \int\limits_{3-i\infty}^{3+i\infty}
                \frac{\mu(r) \, G_{(\alpha,\beta)}(s;r;z)\,\zeta(1+s+2z)}
                {\zeta(1+s+z+\alpha) \, \zeta(1+s+z+\beta)}
                \left(\frac{y}{r}\right)^s \frac{ds}{s^{j+1}}.
    \end{split}
  \]
  We may evaluate the above integral by a standard procedure.
First shift the contour to $\Re(s)=\frac{B}{\log(2y/r)}$, and truncate
  the integral to the line segment from $\frac{B}{\log(2y/r)}-iT$ to $\frac{B}{\log(2y/r)}+iT$
  where $B$ is a constant such that $\frac{B}{\log(2y/r)}>-2\delta$, and
  $$T:=\exp\left(\sqrt{\log(2y/r)}\right).$$
  The error involved in doing so is
  $$
    \ll \; \frac{E(r)}{\log^2 y} \log^3(2y/r) T^{-2}.
  $$
  Next, shift the integral on this line segment to the left onto
  the line segment $-c/\log T$, where $c$ is a positive constant such
  that $\zeta(1+s)$ has no zeros in the region $\Re(s)>-c/\log T$
  and $\Im(s)\leq T$.
  We encounter a multiple pole at $s=0$, and another simple pole
  at $s=-2z$ if $z\not\in\{\alpha,\beta\}$.
  The integrals on the three other sides are bounded using standard estimates
  for $1/\zeta(s)$ in the zero-free region,
  and contribute an amount
  $$
    \ll \; \frac{E(r)}{\log^2 y} \bigg(\log\frac{Ty}{r}\bigg)^3
    \left(T^{-2}+(y/r)^{-c/\log T}\right).
  $$
  We conclude that for an appropriate positive constant $A_0$, we have
  \[
    \begin{split}
       T_{(\alpha,\beta)}(y,R;r;z) & = \underset {s=0}{\Res}\; \frac{\mu(r) \, G_{(\alpha,\beta)}(s;r;z) \, \zeta(1+s+2z)}
                {s \, \zeta(1+s+z+\alpha) \, \zeta(1+s+z+\beta)}
                \sum_{j=2}^{\infty} \frac{R^{(j)}(0)}{s^j\log^j y}
                \left(\frac{y}{r}\right)^s \\
       & \hskip -45pt + \underset {s=-2z}{\Res} \;\frac{\mu(r) \, G_{(\alpha,\beta)}(s;r;z) \, \zeta(1+s+2z)}
                {\zeta(1+s+z+\alpha) \, \zeta(1+s+z+\beta)}
                \left(\frac{y}{r}\right)^s
                \sum_{j=2}^{\infty} \frac{R^{(j)}(0)}{s^{j+1}\log^j y}\; + \; \mathcal O\left( \frac{E(r)}{\log^2 y}e^{-A_0\sqrt{\log y/r}} \right).
    \end{split}
  \]
  Now we will follow the same argument as in
  Conrey--Soundararajan \cite[p. 38]{conrey2002real} for the residue at $s=0$.
  Indeed, we may replace $\sum\limits_{j=2}^{\infty} \frac{R^{(j)}(0)}{s^j\log^j y} \left(\frac{y}{r}\right)^s$
  with
  \[
    \begin{split}
      \sum_{j=2}^{\infty} \frac{R^{(j)}(0)}{s^j\log^j y}
        \bigg(\sum_{l\leq j}\frac{(s\log y/r)^l}{l!}\bigg)
      & = \sum_{k=0}^{\infty} \frac{1}{(s\log y)^k} \sum_{l=0}^{\infty} \frac{R^{(k+l)}(0)}{l!}
         \bigg(\frac{\log(y/r)}{\log y}\bigg)^l \\
      & = \sum_{j=0}^{\infty} \frac{1}{(s\log y)^j}R^{(j)}\bigg(\frac{\log(y/r)}{\log y}\bigg),
    \end{split}
  \]
  by grouping terms according to $k=j-l$, and using the fact $R(0)=R'(0)=0$.
  This completes the proof of the lemma.
\end{proof}



\subsection{Contribution of the short-range terms}


Recall that $0<\Upsilon<1$ is fixed. We first use lemma \ref{lemma:T} and equation \eqref{eqn:S=T+T}
to deduce the estimate for $S_{(\alpha,\beta)}(r;z)$
when $r\leq M^{1-\Upsilon}$.

\subsubsection{The case $(\alpha,\beta)=(\omega,\bar{\omega})$}
In this case we have
\[
  S_{(\omega,\bar{\omega})}(r;z) =
    \frac{\mu(r) \, G_{(\omega,\bar{\omega})}(0;r;z)\,\zeta(1+2z)}
    {\zeta(1+z+\omega) \, \zeta(1+z+\bar\omega)}
    +\mathcal  O\left(\frac{E(r)}{\log^2 K}e^{-A_0\sqrt{\frac{\log M^{1-\Upsilon}}{r}}}\right).
\]
Note that the main term above came form the $j=0$ contribution
in the applications of lemma \ref{lemma:T},
and that the contributions from $j\geq1$ in the two applications cancel each other.
By \eqref{eqn:cV<>}, we have
\[
  \begin{split}
    \cV_{(\omega,\bar{\omega})}^{\leq}(\omega)
    & = \sum_{1\leq r\leq M^{1-\Upsilon}}   \frac{\mu(r)^2}{r^{1+\omega+\bar{\omega}}}
        \tau_{(\omega,\bar{\omega})}(r)  \prod_{z\in\{\omega,\bar{\omega}\}} \frac{G_{(\omega,\bar{\omega})}(0;r;z) \,\zeta(1+2z)}
        {\zeta(1+z+\omega) \, \zeta(1+z+\bar\omega)}
        \; + \; \mathcal O\left( \frac{|\omega+\bar\omega|}{(\log K)^2}  \right).
  \end{split}
\]
Indeed, here we use the fact
\[
  \underset{1\leq r\leq x}{{\sum}^\flat} E(r)^3 \ll x,
  \quad \textrm{and, therefore,} \quad
  \underset{1\leq r\leq x}{{\sum}^\flat} \frac{E(r)^3}{r}e^{-c\sqrt{\log x/r}} \ll 1.
\]
Hence, by \eqref{eqn:ve} and \eqref{eqn:Psi}, we have
\[
  \begin{split}
    &  \Psi(\omega,\bar{\omega})\cV_{(\omega,\bar{\omega})}^{\leq}(\omega) \\
    &\hskip 36pt = \frac{K}{4} \frac{\widetilde{\Phi}(1)}{\zeta(1+\omega+\bar{\omega})}
        \sum_{1\leq r\leq M^{1-\Upsilon}}
        \frac{\mu(r)^2 \tau_{(\omega,\bar{\omega})}(r)  G_{(\omega,\bar{\omega})}(0;r;\omega)
        G_{(\omega,\bar{\omega})}(0;r;\bar{\omega})}{r^{1+\omega+\bar{\omega}}}
        + \, \mathcal O\left( \frac{K}{(\log K)^2}   \right).
  \end{split}
\]
To deal with the innermost sum above, we will use
Perron's formula  to prove the following lemma.

\begin{lemma}\label{lemma:r-sum++}
  Assume that $\log x\asymp\log K$ and $\omega$ satisfies $\mathbf{(I)}$.
  We have
  \[
    \begin{split}
      &  \sum_{1\leq r\leq x}
        \frac{\mu(r)^2 \, \tau_{(\omega,\bar{\omega})}(r)  \,G_{(\omega,\bar{\omega})}(0;r;\omega)\,
        G_{(\omega,\bar{\omega})}(0;r;\bar{\omega})}{r^{1+\omega+\bar{\omega}}} \\
      & \hskip 80pt = \zeta(1+\omega+\bar{\omega})\left(1-x^{-(\omega+\bar{\omega})}\right)(1+\mathcal{O}(|\omega+\bar{\omega}|))
      \; + \; \mathcal{O}\left(x^{-(\omega+\bar{\omega})}\right).
    \end{split}
  \]
  Furthermore, if $1\leq y\leq x$ and $\log y\asymp\log x$ then for any smooth
  function $R$ on $[0,1]$, we have
  \[
    \begin{split}
      &  \sum_{y\leq r\leq x}
        \frac{\mu(r)^2  \, \tau_{(\omega,\bar{\omega})}(r) \, G_{(\omega,\bar{\omega})}(0;r;\omega) \,
        G_{(\omega,\bar{\omega})}(0;r;\bar{\omega})}{r^{1+\omega+\bar{\omega}}}
        R\left( \frac{\log r}{\log x} \right) \\
      &\hskip 100pt = \big(1+ \mathcal O(|\omega+\bar{\omega}|)\big) \int_{y}^{x} R\left( \frac{\log t}{\log x}
        \right)\frac{dt}{t^{1+\omega+\bar{\omega}}} \; + \; \mathcal O_R\big(1\big).
    \end{split}
  \]
\end{lemma}

\begin{proof}
  Recalling the definition of $G_{(\omega,\bar{\omega})}(0;r;\omega)$
  in lemma \ref{lemma:T}, for any square-free $r$, we have
  \[
    \begin{split}
      G_{(\omega,\bar{\omega})}(0;r;z)
      & = G_{(\omega,\bar{\omega})}(0;1;z) \prod_{p|r} \left(1-\frac{1}{p^{1+\omega+\bar{\omega}}}\right)^{-1}
        = \prod_{p|r} \left(1-\frac{1}{p^{1+\omega+\bar{\omega}}}\right)^{-1},
    \end{split}
  \]
  where we use the fact that $G_{(\omega,\bar{\omega})}(0;1;z)=1$
  for $z\in\{\omega,\bar{\omega}\}$. Hence we have
  \[
    \begin{split}
       \sum_{r}
        \frac{\mu(r)^2 \tau_{(\omega,\bar{\omega})}(r)  G_{(\omega,\bar{\omega})}(0;r;\omega)
        G_{(\omega,\bar{\omega})}(0;r;\bar{\omega})}{r^{1+s+\omega+\bar{\omega}}} & = \sum_{r}
        \frac{\mu(r)^2 \tau_{(\omega,\bar{\omega})}(r)  \prod_{p|r} \left(1-\frac{1}{p^{1+\omega+\bar{\omega}}}\right)^{-2}}{r^{1+s+\omega+\bar{\omega}}} \\
      & \hskip -90pt = \prod_p \Bigg( 1 + \frac{\left(1-\frac{1}{p^{1+\omega+\bar{\omega}}}\right)
            \left(1-\frac{1}{p^{1+\omega+\bar{\omega}}}\right)^{-2}}
            {p^{1+s+\omega+\bar{\omega}}} \Bigg) \\
      & \hskip -90pt  = \zeta(1+s+\omega+\bar{\omega}) \prod_p
        \left(1-\frac{1}{p^{1+s+\omega+\bar{\omega}}}\right)
        \left( 1 + \sum_{k=1}^{\infty}\frac{1}{p^{s+k(1+\omega+\bar{\omega})}} \right) \\
      & \hskip -90pt  =: \zeta(1+s+\omega+\bar{\omega}) H_{(\omega,\bar{\omega})}(s),
    \end{split}
  \]
  We know that $H_{(\omega,\bar{\omega})}(s)$ is holomorphic for $\Re(s)>-1+2|\delta|$;
  and $H_{(\omega,\bar{\omega})}(0)=1$.
  Now we can use the Perron's formula to conclude our first statement.
  And the second statement will follow by the partial summation formula.
  Indeed, by Perron's formula (see \cite[Lemma 3.12]{titchmarsh1986theory}),
  we have
  \[
    \begin{split}
      &  \sum_{1\leq r\leq x}
        \frac{\mu(r)^2 \tau_{(\omega,\bar{\omega})}(r)  G_{(\omega,\bar{\omega})}(0;r;\omega)
        G_{(\omega,\bar{\omega})}(0;r;\bar{\omega})}{r^{1+\omega+\bar{\omega}}} \\
      & \hskip 70pt = \frac{1}{2\pi i} \int\limits_{-2\delta+\frac{A}{\log x}-iT}^{-2\delta+\frac{A}{\log x}+iT}
            \zeta(1+s+\omega+\bar{\omega})\, H_{(\omega,\bar{\omega})}(s)\; x^{s}\; \frac{ds}{s}
        \;    + \; \mathcal O\left( \frac{x^{-2\delta}(\log x)^B}{T} \right),
    \end{split}
  \]
  where constants $A,B>0$ satisfy that $-2\delta+\frac{A}{\log x}>0$ and
  $$\tau_{(\omega,\bar{\omega})}(r) \, G_{(\omega,\bar{\omega})}(0;r;\omega) \, G_{(\omega,\bar{\omega})}(0;r;\bar{\omega}) \; = \; \mathcal O\left((\log r)^B\right).$$
  We further choose $T=e^{\sqrt{\log x}}$. Now  shift the contour to the line segment $-2\delta-\frac{C}{\log T}-iT$ to $-2\delta-\frac{C}{\log T}+iT$, getting
  \[
    \begin{split}
      &  \sum_{1\leq r\leq x}
        \frac{\mu(r)^2 \, \tau_{(\alpha,\beta)}(r) \, G_{(\omega,\bar{\omega})}(0;r;\omega) \,
        G_{(\omega,\bar{\omega})}(0;r;\bar{\omega})}{r^{1+\omega+\bar{\omega}}} \\
      &\hskip 50pt = \;\underset {s=0} {\Res} \;  +  \; \underset{s=-(\omega+\bar\omega)}
{\Res} \;             + \; \mathcal O\left( \; \left |\frac{1}{2\pi i} \int\limits_{-2\delta-\frac{C}{\log T}-iT}^{-2\delta-\frac{C}{\log T}+iT}
            \zeta(1+s+\omega+\bar{\omega}) H_{(\omega,\bar{\omega})}(s) x^{s} \frac{ds}{s}\right | \; \right) \\
      & \hskip 117pt + \; \mathcal O\left( \; \left|\frac{1}{2\pi i} \int\limits_{-2\delta-\frac{C}{\log T}\pm iT}^{-2\delta+\frac{A}{\log x}\pm iT}
            \zeta(1+s+\omega+\bar{\omega}) H_{(\omega,\bar{\omega})}(s) x^{s} \frac{ds}{s}\right | \; \right) \;
            + \; \mathcal O\left( x^{-2\delta} \right) \\
      & \hskip 50pt  = \;\zeta(1+\omega+\bar{\omega})\left(1-x^{-(\omega+\bar{\omega})}\right)(1+\mathcal{O}(|\omega+\bar{\omega}|)) \;
       + \; \mathcal{O}\left(x^{-(\omega+\bar{\omega})}\right) \\
      & \hskip 50pt  = \;\zeta(1+\omega+\bar{\omega})\left(1-x^{-(\omega+\bar{\omega})}\right) \; + \; \mathcal O(1).
    \end{split}
  \]
  This proves the first claim of the lemma.
  For our second assertion, we use partial summation, obtaining
  \[
    \begin{split}
      &  \sum_{y\leq r\leq x}
        \frac{\mu(r)^2 \, \tau_{(\omega,\bar{\omega})}(r) \, G_{(\omega,\bar{\omega})}(0;r;\omega) \,
        G_{(\omega,\bar{\omega})}(0;r;\bar{\omega})}{r^{1+\omega+\bar{\omega}}}
        R\left( \frac{\log r}{\log x} \right) \\
      &\hskip 70pt = \int\limits_{y}^{x} R\left( \frac{\log t}{\log x} \right) d \sum_{1\leq r\leq t}
        \frac{\mu(r)^2 \, \tau_{(\omega,\bar{\omega})}(r)  \,
        G_{(\omega,\bar{\omega})}(0;r;\omega) \,
        G_{(\omega,\bar{\omega})}(0;r;\bar{\omega})}{r^{1+\omega+\bar{\omega}}} \\
        &
        \\
      &\hskip 70pt = R\left( \frac{\log t}{\log x} \right)
            \bigg( \sum_{1\leq r\leq t}
        \frac{\mu(r)^2 \, \tau_{(\omega,\bar{\omega})}(r)  \,
        G_{(\omega,\bar{\omega})}(0;r;\omega) \,
        G_{(\omega,\bar{\omega})}(0;r;\bar{\omega})}{r^{1+\omega+\bar{\omega}}} \bigg)\bigg|_{y}^{x}\\
      &\hskip 120pt - \int\limits_{y}^{x} \bigg( \sum_{1\leq r\leq t}
        \frac{\mu(r)^2 \, \tau_{(\omega,\bar{\omega})}(r)  \,
        G_{(\omega,\bar{\omega})}(0;r;\omega) \,
        G_{(\omega,\bar{\omega})}(0;r;\bar{\omega})}{r^{1+\omega+\bar{\omega}}} \bigg)
            d R\left( \frac{\log t}{\log x} \right) \\
            &
            \\
      &\hskip 70pt = R\left( \frac{\log t}{\log x} \right)
            \zeta(1+\omega+\bar{\omega})\left(1-t^{-(\omega+\bar{\omega})}\right)
            \bigg|_{y}^{x} \\
      &\hskip 120pt  - \int\limits_{y}^{x} \zeta(1+\omega+\bar{\omega})\left(1-t^{-(\omega+\bar{\omega})}\right)
            d R\left( \frac{\log t}{\log x} \right) \; + \; \mathcal O_R(1) \\
      &\hskip 70pt = (1+ \mathcal O\big(|\omega+\bar{\omega}|)\big) \int\limits_{y}^{x} R\left( \frac{\log t}{\log x}
        \right)\frac{dt}{t^{1+\omega+\bar{\omega}}} \; + \; \mathcal O_R(1).
    \end{split}
  \]
  This completes the proof of the lemma.
\end{proof}

As a consequence of lemma \ref{lemma:r-sum++}, we obtain
\begin{equation}\label{eqn:PsiV<++}
  \begin{split}
    \Psi(\omega,\bar{\omega})\cV_{(\omega,\bar{\omega})}^{\leq}(\omega)
    & = \frac{K}{4} \widetilde{\Phi}(1)
        \left(1-M^{-(1+\Upsilon)(\omega+\bar{\omega})}\right)(1+ \mathcal O\big(|\omega+\bar{\omega}|)\big)
       \;  + \; \mathcal O\left( \frac{K}{(\log K)^2}  \right).
  \end{split}
\end{equation}

\subsubsection{The case $(\alpha,\beta)=(-\omega,-\bar{\omega})$}

Here we have
\[
  \begin{split}
    S_{(-\omega,-\bar{\omega})}(r;z) &   =
    \frac{\mu(r)G_{(-\omega,-\bar{\omega})}(-2z;r;z)r^{2z}}
    {\zeta(1-z-\omega)\zeta(1-z-\bar{\omega})}
    \Bigg( M^{-2z} \sum_{j=2}^{\infty} \frac{P^{(j)}(0)}{(-2z)^{j+1}\log^j M} \\
    & \quad\ \ \qquad + M^{-2z(1-\Upsilon)} \sum_{j=2}^{\infty}
    \frac{Q^{(j)}(0)}{(-2z)^{j+1}\log^j M^{1-\Upsilon}} \Bigg)
   \; + \; \mathcal O\left(\frac{E(r)}{\log^2 K}e^{-A_0\sqrt{\frac{\log M^{1-\Upsilon}}{r}}}\right).
  \end{split}
\]
It follows from \eqref{eqn:cV<>} that
\[
  \begin{split}
    \cV_{(-\omega,-\bar{\omega})}^{\leq}(\omega)
    & = \sum_{1\leq r\leq M^{1-\Upsilon}}   \frac{\mu(r)^2 \tau_{(-\omega,-\bar{\omega})}(r)  G_{(-\omega,-\bar{\omega})}(-2\omega;r;\omega)G_{(-\omega,-\bar{\omega})}(-2\bar\omega;r;\bar\omega) }{r^{1-\omega-\bar{\omega}}\zeta(1-2\omega)\zeta(1-2\bar\omega)\zeta(1-\omega-\bar\omega)^2} \\
    & \qquad\qquad \cdot \Bigg| M^{-2\omega} \sum_{j=2}^{\infty} \frac{P^{(j)}(0)}{(-2\omega)^{j+1}\log^j M}
        + M^{-2\omega(1-\Upsilon)} \sum_{j=2}^{\infty}
        \frac{Q^{(j)}(0)}{(-2\omega)^{j+1}\log^j M^{1-\Upsilon}} \Bigg|^2 \\
    & \hskip 132pt  +\; \mathcal  O\left( \frac{|\omega+\bar\omega|(\log\log K)^3}{(\log K)^2}  \right).
  \end{split}
\]
Hence
\[
  \begin{split}
    \Psi(-\omega,-\bar{\omega}) & \cV_{(-\omega,-\bar{\omega})}^{\leq}(\omega)
     = \frac{K}{4} \left(\frac{K}{4\pi}\right)^{-2\omega-2\bar\omega}
            \frac{\widetilde{\Phi}(1)(1+\mathcal{O}(|\omega|))}
            {\zeta(1-\omega-\bar{\omega})}
            \\
    & \qquad \cdot  \sum_{1\leq r\leq M^{1-\Upsilon}}
        \frac{\mu(r)^2 \, \tau_{(-\omega,-\bar{\omega})}(r) \,G_{(-\omega,-\bar{\omega})}(-2\omega;r;\omega) \,G_{(-\omega,-\bar{\omega})}(-2\bar\omega;r;\bar\omega) }{r^{1-\omega-\bar{\omega}}} \\
    & \qquad\qquad \cdot \Bigg| M^{-2\omega} \sum_{j=2}^{\infty} \frac{P^{(j)}(0)}{(-2\omega)^{j}\log^j M}
        + M^{-2\omega(1-\Upsilon)} \sum_{j=2}^{\infty}
        \frac{Q^{(j)}(0)}{(-2\omega)^{j}\log^j M^{1-\Upsilon}} \Bigg|^2 \\
    & \hskip 130pt +\; \mathcal  O\left( \frac{(\log\log K)^3}{(\log K)^2} K  \right).
  \end{split}
\]
To deal with the $r$-sum above, we will use
Perron's formula again to prove the following lemma.

\begin{lemma}\label{lemma:r-sum--}
  Assume that $\log x\asymp\log K$ and $\omega$ satisfies $\mathbf{(I)}.$
  We have
  \[
    \begin{split}
      &  \sum_{1\leq r\leq x} \frac{\mu(r)^2 \, \tau_{(-\omega,-\bar{\omega})}(r) \, G_{(-\omega,-\bar{\omega})}(-2\omega;r;\omega) \, G_{(-\omega,-\bar{\omega})}(-2\bar\omega;r;\bar\omega) }{r^{1-\omega-\bar{\omega}}}  \\
      &\hskip 90pt = \; \zeta(1-\omega-\bar{\omega})\left(1-x^{\omega+\bar{\omega}}\right)(1\; + \; \mathcal O(|\omega|))
     \;  + \; \mathcal  O\left( e^{-c\sqrt{\log x}}\right).
    \end{split}
  \]
  Furthermore, if $1\leq y\leq x$ and $\log y\asymp\log x$ then for any smooth
  function $R$ on $[0,1]$, we have
  \[
    \begin{split}
      &  \sum_{y\leq r\leq x}
        \frac{\mu(r)^2 \, \tau_{(-\omega,-\bar{\omega})}(r) \, G_{(-\omega,-\bar{\omega})}(-2\omega;r;\omega) \, G_{(-\omega,-\bar{\omega})}(-2\bar\omega;r;\bar\omega) }{r^{1-\omega-\bar{\omega}}}
        R\left( \frac{\log r}{\log x} \right) \\
      & \hskip 110pt = (1+\mathcal{O}(|\omega+\bar\omega|)) \int\limits_{y}^{x} R\left( \frac{\log t}{\log x} 
        \right)\frac{dt}{t^{1-\omega-\bar{\omega}}} \; + \; \mathcal O_R\left(x^{\omega+\bar{\omega}}\right).
    \end{split}
  \]
  Similarly, we have
  \[
    \begin{split}
         &  \sum_{y<r\leq x}
        \frac{\mu(r)^2 \, \tau_{(-\omega,-\bar\omega)}(r) \, G_{(-\omega,-\bar\omega)}(0;r;\omega) \, G_{(-\omega,-\bar\omega)}(0;r;\bar\omega)}
        {r^{1+\omega+\bar{\omega}}} R\left( \frac{\log r}{\log x} \right) \\
         & \hskip 90pt = \big(1+\mathcal O(|\omega+\bar\omega|)\big) \int\limits_{y}^{x} R\left( \frac{\log t}{\log x} 
        \right)\frac{dt}{t^{1+\omega+\bar{\omega}}} \; + \; \mathcal O_R\left((\log\log K)^2\right),
    \end{split}
  \]
  and
  \[
    \begin{split}
      & \sum_{y<r\leq x}
        \frac{\mu(r)^2 \, \tau_{(-\omega,-\bar\omega)}(r) \, G_{(-\omega,-\bar\omega)}(0;r;\omega) \, G_{(-\omega,-\bar{\omega})}(-2\bar\omega;r;\bar\omega)}
        {r^{1+\omega-\bar{\omega}}} R\left( \frac{\log r}{\log x} \right) \\
      & \hskip 90pt = \big(1+ \mathcal O(|\omega-\bar\omega|)\big) \int\limits_{y}^{x} R\left( \frac{\log t}{\log x} \right)
        \frac{dt}{t^{1+\omega-\bar{\omega}}} \; + \; \mathcal O_R\left((\log\log K)^2\right).
    \end{split}
  \]
\end{lemma}

\begin{proof}
  Recalling the definition of $G_{(-\omega,-\bar{\omega})}(-2\omega;r;\omega)$
  in lemma \ref{lemma:T}, we have
  \[
    \begin{split}
      G_{(-\omega,-\bar{\omega})}(-2z;r;z)
      & = G_{(-\omega,-\bar{\omega})}(-2z;1;z) \prod_{p|r} \left(1-\frac{1}{p^{1-z-\omega}}-\frac{1}{p^{1-z-\bar{\omega}}}+\frac{1}{p}\right)^{-1},
    \end{split}
  \]
  where  $G_{(-\omega,-\bar{\omega})}(-2z;1;z)=1+ \mathcal O\big(|\omega|\big)$
  for $z\in\{\omega,\bar{\omega}\}$. Hence
  \[
    \begin{split}
      & \sum_{r} \frac{\mu(r)^2 \, \tau_{(-\omega,-\bar{\omega})}(r) \, G_{(-\omega,-\bar{\omega})}(-2\omega;r;\omega) \, G_{(-\omega,-\bar{\omega})}(-2\bar\omega;r;\bar\omega) }{r^{1+s-\omega-\bar{\omega}}}  \\
      & \hskip 20pt= \prod_{z\in\{\omega,\bar\omega\}} G_{(-\omega,-\bar{\omega})}(-2z;1;z)
          \sum_{r} \frac{\mu(r)^2 \tau_{(-\omega,-\bar{\omega})}(r)
          \prod\limits_{z\in\{\omega,\bar\omega\}}\prod\limits_{p|r} \left(1-\frac{1}{p^{1-z-\omega}}-\frac{1}{p^{1-z-\bar{\omega}}}+\frac{1}{p}\right)^{-1} }{r^{1+s-\omega-\bar{\omega}}}  \\
      & \hskip 20pt = \prod_{z\in\{\omega,\bar\omega\}} G_{(-\omega,-\bar{\omega})}(-2z;1;z)
      \prod_p \bigg(1+\frac{\tau_{(-\omega,-\bar{\omega})}(p)
      \prod\limits_{z\in\{\omega,\bar\omega\}} \left(1-\frac{1}{p^{1-z-\omega}}-\frac{1}{p^{1-z-\bar{\omega}}}+\frac{1}{p}\right)^{-1} }{p^{1+s-\omega-\bar{\omega}}} \bigg) \\
      &
      \\
      & \hskip 20pt =: \zeta(1+s-\omega-\bar{\omega}) \, H_{(-\omega,-\bar{\omega})}(s),
    \end{split}
  \]
  where $H_{(-\omega,-\bar{\omega})}(s)$ is holomorphic for $\Re(s)>-1+2|\delta|$;
  and $H_{(-\omega,-\bar{\omega})}(0)=1+ \mathcal O(|\omega|)$.
  Now we can use the Perron's formula to obtain the first statement in lemma \ref{lemma:r-sum--}.
 The second statement  follows by the partial summation formula.
\end{proof}

It follows from lemma \ref{lemma:r-sum--} that
\begin{equation}\label{eqn:PsiV<--}
  \begin{split}
    \Psi(-\omega,-\bar{\omega})  \cV_{(-\omega,-\bar{\omega})}^{\leq}(\omega)
    & = \frac{K}{4} \left(\frac{K}{4\pi}\right)^{-2\omega-2\bar\omega}
        \widetilde{\Phi}(1) \left(1-M^{(\omega+\bar{\omega})(1-\Upsilon)}\right) (1+\mathcal{O}(|\omega|))    \\
    & \hskip 30pt \cdot \Bigg| M^{-2\omega} \sum_{j=2}^{\infty} \frac{P^{(j)}(0)}{(-2\omega)^{j}\log^j M}
        + M^{-2\omega(1-\Upsilon)} \sum_{j=2}^{\infty}
        \frac{Q^{(j)}(0)}{(-2\omega)^{j}\log^j M^{1-\Upsilon}} \Bigg|^2 \\
    & \hskip 110pt + \;\mathcal O\left( \frac{(\log\log K)^3}{(\log K)^2} K  \right).
  \end{split}
\end{equation}

\subsubsection{The cases $(\alpha,\beta)=(\omega,-\bar{\omega})$ and $(\alpha,\beta)=(-\omega,\bar{\omega})$}


We first consider $(\alpha,\beta)=(\omega,-\bar{\omega})$.
By lemma \ref{lemma:T}, we have
\[
  S_{(\omega,-\bar{\omega})}(r;\omega) =
    \frac{\mu(r) \, G_{(\omega,-\bar{\omega})}(0;r;\omega)}
    {\zeta(1+\omega-\bar\omega)}
   \; + \; \mathcal O\left(\frac{E(r)}{\log^2 K}e^{-A_0\sqrt{\frac{\log M^{1-\Upsilon}}{r}}}\right),
\]
and
\[
  \begin{split}
    S_{(\omega,-\bar{\omega})}(r;\bar\omega) &  =
    \frac{\mu(r) \, G_{(\omega,-\bar{\omega})}(-2\bar\omega;r;\bar\omega) \, r^{2\bar\omega}}
    {\zeta(1-\bar\omega+\omega)\zeta(1-2\bar{\omega})}
    \Bigg( M^{-2\bar\omega} \sum_{j=2}^{\infty} \frac{P^{(j)}(0)}{(-2\bar\omega)^{j+1}(\log M)^j} \\
    & \hskip 40pt \ + M^{-2\bar\omega(1-\Upsilon)} \sum_{j=2}^{\infty}
    \frac{Q^{(j)}(0)}{(-2\bar\omega)^{j+1}(\log M^{1-\Upsilon})^j} \Bigg)
  \;  + \; \mathcal O\left(\frac{E(r)}{\log^2 K}e^{-A_0\sqrt{\frac{\log M^{1-\Upsilon}}{r}}}\right).
  \end{split}
\]
By \eqref{eqn:cV<>}, we have
\[
  \begin{split}
    \cV_{(\omega,-\bar{\omega})}^{\leq}(\omega)
    & = \sum_{1\leq r\leq M^{1-\Upsilon}}   \frac{\mu(r)^2\, \tau_{(\omega,-\bar{\omega})}(r) \, G_{(\omega,-\bar{\omega})}(0;r;\omega) \,
    G_{(\omega,-\bar{\omega})}(-2\bar\omega;r;\bar\omega) }{r^{1+\omega-\bar{\omega}}\zeta(1-2\bar\omega)\zeta(1+\omega-\bar\omega)^2} \\
    & \qquad \cdot \Bigg( M^{-2\bar\omega} \sum_{j=2}^{\infty} \frac{P^{(j)}(0)}{(-2\bar\omega)^{j+1}(\log M)^j} + M^{-2\bar\omega(1-\Upsilon)} \sum_{j=2}^{\infty}
    \frac{Q^{(j)}(0)}{(-2\bar\omega)^{j+1}(\log M^{1-\Upsilon})^j} \Bigg) \\
    & \hskip 120pt  + \; \mathcal O\left( \frac{|\omega-\bar\omega|(\log\log K)^3}{(\log K)^2}  \right).
  \end{split}
\]
Hence
\[
  \begin{split}
  &  \Psi(\omega,-\bar{\omega}) \cV_{(\omega,-\bar{\omega})}^{\leq}(\omega)
     = - \frac{K}{4} \left(\frac{K}{4\pi}\right)^{-2\bar\omega}
            \frac{\widetilde{\Phi}(1)}{\zeta(1+\omega-\bar{\omega})}
            \left(1+ \mathcal O\big(|\omega|\big)\right) \\
    & \hskip 50pt \cdot  \sum_{1\leq r\leq M^{1-\Upsilon}}   \frac{\mu(r)^2 \, \tau_{(\omega,-\bar{\omega})}(r)  \, G_{(\omega,-\bar{\omega})}(0;r;\omega) \,
    G_{(\omega,-\bar{\omega})}(-2\bar\omega;r;\bar\omega)} {r^{1+\omega-\bar{\omega}}} \\
    & \hskip 75pt \cdot \Bigg( M^{-2\bar\omega} \sum_{j=2}^{\infty} \frac{P^{(j)}(0)}{(-2\bar\omega)^{j}(\log M)^j} + M^{-2\bar\omega(1-\Upsilon)} \sum_{j=2}^{\infty}
    \frac{Q^{(j)}(0)}{(-2\bar\omega)^{j}(\log M^{1-\Upsilon})^j} \Bigg) \\
    & \hskip 170pt + \; \mathcal  O\left( \frac{(\log\log K)^3}{(\log K)^2} K  \right).
  \end{split}
\]
To deal with the $r$-sum above, we will use
Perron's formula again to prove the following lemma.

\begin{lemma}\label{lemma:r-sum+-}
  Assume that $\log x\asymp\log K$ and $\omega$ satisfies $\mathbf{(I)}$.
  We have
  \[
    \begin{split}
      &  \sum_{1\leq r\leq x} \frac{\mu(r)^2 \,\tau_{(\omega,-\bar{\omega})}(r) \,
        G_{(\omega,-\bar{\omega})}(0;r;\omega) \,
        G_{(\omega,-\bar{\omega})}(-2\bar\omega;r;\bar\omega) }
        {r^{1+\omega-\bar{\omega}}} \\
      &\hskip 100pt = \zeta(1+\omega-\bar{\omega})\left(1-x^{-(\omega-\bar{\omega})}\right)
            \Big(1+ \mathcal O(|\omega|)\Big)
      \; + \; \mathcal O\big(1\big).
    \end{split}
  \]
  Furthermore, if $1\leq y\leq x$ and $\log y\asymp\log x$ then for any smooth
  function $R$ on $[0,1]$, we have
  \[
    \begin{split}
      & \sum_{y\leq r\leq x}
        \frac{\mu(r)^2 \, \tau_{(\omega,-\bar{\omega})}(r)\,
        G_{(\omega,-\bar{\omega})}(0;r;\omega)\,
        G_{(\omega,-\bar{\omega})}(-2\bar\omega;r;\bar\omega) }{r^{1+\omega-\bar{\omega}}}
        R\left( \frac{\log r}{\log x} \right) \\
      &\hskip 100pt = \big(1+ \mathcal O\big(|\omega-\bar\omega|\big)\big) \int\limits_{y}^{x} R\left( \frac{\log t}{\log x}
        \right)\frac{dt}{t^{1+\omega-\bar{\omega}}} \; + \; \mathcal  O_R\left((\log\log K)^2\right).
    \end{split}
  \]
  Similarly, we have
  \[
    \begin{split}
      & \sum_{y<r\leq x}
        \frac{\mu(r)^2 \, \tau_{(\omega,-\bar\omega)}(r) \, G_{(\omega,-\bar\omega)}(0;r;\omega) \, G_{(\omega,-\bar\omega)}(0;r;\bar\omega)}
        {r^{1+\omega+\bar{\omega}}} R\left( \frac{\log r}{\log x} \right) \\
      &\hskip 100pt = \big(1+ \mathcal O(|\omega+\bar\omega|)\big) \int\limits_{y}^{x} R\left( \frac{\log t}{\log x} 
        \right)\frac{dt}{t^{1+\omega+\bar{\omega}}} \; + \; \mathcal O_R\left((\log\log K)^2\right).
    \end{split}
  \]
\end{lemma}

\begin{proof}
  Recalling the definition of $G_{(\omega,-\bar{\omega})}(0;r;\omega)$ and $G_{(\omega,-\bar{\omega})}(-2\bar\omega;r;\bar\omega)$
  in lemma \ref{lemma:T}, we have
  \[
    G_{(\omega,-\bar{\omega})}(0;r;\omega) \; = \;
    G_{(\omega,-\bar{\omega})}(0;1;\omega)
    \prod_{p|r} \left( 1- \frac{1}{p^{1+\omega-\bar\omega}}\right)^{-1},
  \]
  \[
    \begin{split}
      G_{(\omega,-\bar{\omega})}(-2\bar\omega;r;\bar\omega)
      & \; = \; G_{(\omega,-\bar{\omega})}(-2\bar\omega;1;\bar\omega) \prod_{p|r} \left(1-\frac{1}{p^{1-\bar\omega+\omega}}-\frac{1}{p^{1-2\bar{\omega}}}+\frac{1}{p}\right)^{-1},
    \end{split}
  \]
  where  $G_{(\omega,-\bar{\omega})}(0;1;\omega)=1$ and
  $G_{(\omega,-\bar{\omega})}(-2\bar\omega;1;\bar\omega)=1+ \mathcal O(|\omega|)$.
  It follows that
    \[
    \begin{split}
      &\ \sum_{r} \frac{\mu(r)^2 \, \tau_{(\omega,-\bar{\omega})}(r) \,
        G_{(\omega,-\bar{\omega})}(0;r;\omega) \,
        G_{(\omega,-\bar{\omega})}(-2\bar\omega;r;\bar\omega) }{r^{1+s+\omega-\bar{\omega}}}  \\
      & \hskip 26pt = G_{(\omega,-\bar{\omega})}(-2\bar\omega;1;\bar\omega)
          \sum_{r} \frac{\mu(r)^2 \tau_{(\omega,-\bar{\omega})}(r)
          \prod\limits_{p|r} \big( 1- \frac{1}{p^{1+\omega-\bar\omega}}\big)^{-1}
          \big(1-\frac{1}{p^{1-\bar\omega+\omega}}-\frac{1}{p^{1-2\bar{\omega}}}+\frac{1}{p}\big)^{-1} }{r^{1+s+\omega-\bar{\omega}}}  \\
      & \hskip 26pt = G_{(\omega,-\bar{\omega})}(-2\bar\omega;1;\bar\omega)
      \prod_p \bigg(1+\frac{\tau_{(\omega,-\bar{\omega})}(p)
      \big( 1- \frac{1}{p^{1+\omega-\bar\omega}}\big)^{-1}
          \big(1-\frac{1}{p^{1-\bar\omega+\omega}}-\frac{1}{p^{1-2\bar{\omega}}}+\frac{1}{p}\big)^{-1} }
          {p^{1+s+\omega-\bar{\omega}}} \bigg) \\
      & \hskip 26pt =: \zeta(1+s+\omega-\bar{\omega}) \,H_{(\omega,-\bar{\omega})}(s),
    \end{split}
  \]
  where $H_{(\omega,-\bar{\omega})}(s)$ is holomorphic for $\Re(s)>-1+2|\delta|$;
  and $H_{(\omega,-\bar{\omega})}(0)=1+ \mathcal O(|\omega|)$.
  It follows from  Perron's formula that the first statement in lemma \ref{lemma:r-sum+-} holds.
 The second statement follows by the partial summation formula.
\end{proof}

By the above lemma, we get
\begin{equation}\label{eqn:PsiV<+-}
  \begin{split}
    \Psi(\omega,-\bar{\omega}) & \cV_{(\omega,-\bar{\omega})}^{\leq}(\omega)
     = -\frac{K}{4} \left(\frac{K}{4\pi}\right)^{-2\bar\omega}
        \widetilde{\Phi}(1) \left(1-M^{-(\omega-\bar{\omega})(1-\Upsilon)}\right) \left(1+ \mathcal O\big(|\omega|\big) \right)    \\
    & \qquad \cdot \Bigg( M^{-2\bar\omega} \sum_{j=2}^{\infty} \frac{P^{(j)}(0)}{(-2\bar\omega)^{j}(\log M)^j} + M^{-2\bar\omega(1-\Upsilon)} \sum_{j=2}^{\infty}
    \frac{Q^{(j)}(0)}{(-2\bar\omega)^{j}(\log M^{1-\Upsilon})^j} \Bigg) \\
    & \hskip 110pt + \; \mathcal  O\left( \frac{(\log\log K)^3}{(\log K)^2} K  \right).
  \end{split}
\end{equation}
Then, by the same argument, we obtain
\begin{equation}\label{eqn:PsiV<-+}
  \begin{split}
    \Psi(-\omega,\bar{\omega}) & \cV_{(-\omega,\bar{\omega})}^{\leq}(\omega)
     = -\frac{K}{4} \left(\frac{K}{4\pi}\right)^{-2\omega}
        \widetilde{\Phi}(1) \left(1-M^{-(-\omega+\bar{\omega})(1-\Upsilon)}\right) \left(1+\mathcal O\big(|\omega|\big) \right)    \\
    & \qquad \cdot \Bigg( M^{-2\omega} \sum_{j=2}^{\infty} \frac{P^{(j)}(0)}{(-2\omega)^{j}(\log M)^j} +
        M^{-2\omega(1-\Upsilon)} \sum_{j=2}^{\infty}
        \frac{Q^{(j)}(0)}{(-2\omega)^{j}(\log M^{1-\Upsilon})^j} \Bigg) \\
    & \hskip 110pt + \; \mathcal O\left( \frac{(\log\log K)^3}{(\log K)^2} K  \right).
  \end{split}
\end{equation}

%
%

\subsection{Contribution of the long-range terms}


We first use lemma \ref{lemma:T} and equation \eqref{eqn:S=T}
to deduce the estimate for $S_{(\alpha,\beta)}(r;z)$
when $M^{1-\Upsilon}<r\leq M$.

\subsubsection{The case $(\alpha,\beta)=(\omega,\bar{\omega})$}

In this case we have
\[
  \begin{split}
    S_{(\omega,\bar\omega)}(r;z)
      & = \underset {s=0} {\Res} \;\frac{\mu(r) \, G_{(\omega,\bar\omega)}(s;r;z)}
                {s\, \zeta(1+s+\omega+\bar\omega)} \;
                \sum_{j=0}^{\infty} \frac{1}{(s\log M)^j}
                P^{(j)}\Big(\frac{\log M/r}{\log M}\Big)
     \;  + \; \mathcal O\left(\frac{E(r)}{\log^2 M}\right).
  \end{split}
\]
Consider the Taylor expansion
$$G_{(\omega,\bar\omega)}(s;r;z)/\zeta(1+s+\omega+\bar\omega) \; = \; a_0+a_1s+a_2s^2+a_3s^3+\cdots$$

Then we have
\[
  \begin{split}
     a_0 & = G_{(\omega,\bar\omega)}(0;r;z) (\omega+\bar\omega) \; +\; \mathcal O\left(E(r)|\omega|^2\right), \\
     a_1 & = G_{(\omega,\bar\omega)}(0;r;z) \; + \; \mathcal O\big(E(r)|\omega|\big), \\
     a_n & \ll_n E(r), \quad \mathrm{for}\ n\geq2.
  \end{split}
\]
It follows that
\[
  \begin{split}
    S_{(\omega,\bar\omega)}(r;z)
      & = \mu(r)\, G_{(\omega,\bar\omega)}(0;r;z)\, \left((\omega+\bar\omega)P\Big(\frac{\log M/r}{\log M}\Big) +
                \frac{1}{\log M} P^{'}\Big(\frac{\log M/r}{\log M}\Big) \right) \\
      & \hskip 100pt + \; \mathcal O\left(E(r)|\omega|^2 + \frac{E(r)}{\log^2 M}\right).
  \end{split}
\]
By \eqref{eqn:cV<>}, we have
\[
  \begin{split}
    \cV_{(\omega,\bar{\omega})}^{>}(\omega)
    & = \sum_{M^{1-\Upsilon}<r\leq M}
        \frac{\mu(r)^2 \, \tau_{(\omega,\bar\omega)}(r) \,G_{(\omega,\bar\omega)}(0;r;\omega) \, G_{(\omega,\bar\omega)}(0;r;\bar\omega)}
        {r^{1+\omega+\bar{\omega}}} \\
    &  \hskip 76pt \cdot \left((\omega+\bar\omega)P\Big(\frac{\log M/r}{\log M}\Big) +
                \frac{1}{\log M} P^{'}\Big(\frac{\log M/r}{\log M}\Big) \right)^2 \\
    & \hskip 156pt  + \; \mathcal O\bigg(\underset{1\leq r\leq M}{{\sum}^\flat}\frac{E(r)^3}{r}\bigg(|\omega|^3+\frac{|\omega|}{\log^2 K}\bigg)\bigg).
  \end{split}
\]
Now by lemma \ref{lemma:r-sum++}, we get
\[
  \begin{split}
    \cV_{(\omega,\bar{\omega})}^{>}(\omega)
    & =  \big(1+ \mathcal O\big(|\omega|\big)\big) \int\limits_{M^{1-\Upsilon}}^{M} \bigg((\omega+\bar\omega)P\Big(\frac{\log M/t}{\log M}\Big) \;+ \; \frac{1}{\log M} P^{'}\Big(\frac{\log M/t}{\log M}\Big) \bigg)^2 \frac{dt}{t^{1+\omega+\bar\omega}}\\
 &   \hskip 120pt
      \; + \; \mathcal O\left(\frac{(\log\log K)^3}{\log^2 K}\right).
  \end{split}
\]
Making the change of variable $\frac{\log M/t}{\log M} \mapsto x$, we see that
\[
  \begin{split}
    \cV_{(\omega,\bar{\omega})}^{>}(\omega)
    & = (\log M) \big(1+ \mathcal O\big(|\omega|\big)\big) \int\limits_{0}^{\Upsilon} M^{-(1-x)(\omega+\bar\omega)}  \cdot  \left((\omega+\bar\omega)P(x) +
            \frac{1}{\log M} P^{'}(x)\right)^2 dx\\
            & \hskip 120pt
      \; + \; \mathcal O\left(\frac{(\log\log K)^3}{\log^2 K}\right).
  \end{split}
\]
Integrating by parts, we obtain
\[
  \begin{split}
    \cV_{(\omega,\bar{\omega})}^{>}(\omega)
    & = \int\limits_{0}^{\Upsilon} (\omega+\bar\omega)P(x)^2 \, d M^{-(1-x)(\omega+\bar\omega)}
          \;  + \; 2\int\limits_{0}^{\Upsilon} M^{-(1-x)(\omega+\bar\omega)}
            (\omega+\bar\omega)P(x)P^{'}(x) \;dx \\
    & \hskip 130pt + \frac{1}{\log M} \int\limits_{0}^{\Upsilon} M^{-(1-x)(\omega+\bar\omega)}
            \left(P^{'}(x)\right)^2 dx
           \; + \; \mathcal O\left(\frac{(\log\log K)^3}{\log^2 K}\right)\\
    & \hskip -30pt
    = (\omega+\bar\omega)P(x)^2 M^{-(1-x)(\omega+\bar\omega)} \bigg|_{0}^{\Upsilon}
           \; + \; \frac{1}{\log M} \int_{0}^{\Upsilon} M^{-(1-x)(\omega+\bar\omega)}
            \left(P^{'}(x)\right)^2 dx \; + \; \mathcal O\left(\frac{(\log\log K)^3}{\log^2 K}\right)\\
    &\hskip -30pt = (\omega+\bar\omega)M^{-(1-\Upsilon)(\omega+\bar\omega)}
      \;  + \; \frac{1}{\log M} \int_{0}^{\Upsilon} M^{-(1-x)(\omega+\bar\omega)}
            \left(P^{'}(x)\right)^2 dx \; + \; \mathcal O\left(\frac{(\log\log K)^3}{\log^2 K}\right).
  \end{split}
\]
It then follows from \eqref{eqn:Psi} that
\begin{equation}\label{eqn:PsiV>++}
  \begin{split}
    \Psi(\omega,\bar{\omega})\cV_{(\omega,\bar{\omega})}^{>}(\omega)
    & = \frac{K}{4} \widetilde{\Phi}(1) M^{-(1+\Upsilon)(\omega+\bar{\omega})}\big(1+ \mathcal O\big(|\omega+\bar{\omega}|\big)\big) \\
    & \hskip -50pt   + \;\frac{K}{4}  \frac{\widetilde{\Phi}(1)}{(\omega+\bar\omega)\log M} \int_{0}^{\Upsilon}
            M^{-(1-x)(\omega+\bar\omega)}
            \left(P^{'}(x)\right)^2 dx \; +
            \; \mathcal  O\left( \frac{(\log\log K)^4}{\log K} K  \right).
  \end{split}
\end{equation}

\subsubsection{The case $(\alpha,\beta)=(-\omega,-\bar{\omega})$}
In this case we have
\[
  \begin{split}
    S_{(-\omega,-\bar\omega)}(r;z)
      & = \underset {s=0}{\Res}\; \frac{\mu(r) \, G_{(-\omega,-\bar\omega)}(s;r;z) \,\zeta(1+s+2z)}
                {s\, \zeta(1+s+z-\omega) \, \zeta(1+s+z-\bar\omega)} \;
                \sum_{j=0}^{\infty} \frac{1}{(s\log M)^j}
                P^{(j)}\Big(\frac{\log M/r}{\log M}\Big) \\
       & \hskip 40pt + \; \frac{\mu(r) \, G_{(-\omega,-\bar{\omega})}(-2z;r;z) \, r^{2z}}
                {\zeta(1-z-\omega) \, \zeta(1-z-\bar{\omega})}
                 M^{-2z} \sum_{j=2}^{\infty} \frac{P^{(j)}(0)}{(-2z)^{j+1}\log^j M} \;
       + \; \mathcal O\left(\frac{E(r)}{\log^2 K}\right).
  \end{split}
\]
Consider the Taylor expansion
$$\frac{G_{(-\omega,-\bar\omega)}(s;r;z) \,\zeta(1+s+2z)}
{\zeta(1+s+z-\omega) \,\zeta(1+s+z-\bar\omega)} \; = \; a_0+a_1s+a_2s^2+a_3s^3+\cdots$$

\pagebreak

Then we have
\[
  \begin{split}
     a_0 & = 0, \\
     a_1 & = \frac{G_{(-\omega,-\bar\omega)}(0;r;z)(3z^2+z(-\omega-\bar\omega)-\omega\bar\omega)}{(2z)^2}
            + \; \mathcal O\big(E(r)|\omega|\big), \\
     a_2 & = \frac{G_{(-\omega,-\bar\omega)}(0;r;z)(z+\omega)(z+\bar\omega)}{(2z)^3} \;  + \; \mathcal  O\big(E(r)\big), \\
     a_3 & = \frac{-G_{(-\omega,-\bar\omega)}(0;r;z)(z+\omega)(z+\bar\omega)}{(2z)^4}
          \;  + \; \mathcal O\big(E(r)|\omega|^{-1}\big).
  \end{split}
\]
It follows that
\[
  \begin{split}
    S_{(-\omega,-\bar\omega)}(r;z)
      & = \mu(r) \, G_{(-\omega,-\bar\omega)}(0;r;z)
            \bigg(\frac{z-\bar{z}}{2z\log M}P'\Big(\frac{\log M/r}{\log M}\Big) \\
      & \hskip 30pt +\; \frac{\omega+\bar\omega}{(2z\log M)^2} P''\Big(\frac{\log M/r}{\log M}\Big)
      -  \frac{\omega+\bar\omega}{(2z\log M)^3} P'''\Big(\frac{\log M/r}{\log M}\Big)  \bigg) \\
      & \hskip 50pt + \;\frac{\mu(r)G_{(-\omega,-\bar{\omega})}(-2z;r;z)r^{2z}}
                {\zeta(1-z-\omega)\zeta(1-z-\bar{\omega})}
                 M^{-2z} \sum_{j=2}^{\infty} \frac{P^{(j)}(0)}{(-2z)^{j+1}\log^j M} \\
      &\hskip 125pt +\; \mathcal O\left(E(r)|\omega|^2 + \frac{E(r)}{\log^2 M}\right).
  \end{split}
\]
By \eqref{eqn:cV<>}, we have
\[
  \begin{split}
    & \cV_{(-\omega,-\bar{\omega})}^{>}(\omega) \; = \hskip -7pt\sum_{M^{1-\Upsilon}<r\leq M}
        \frac{\mu(r)^2 \tau_{(-\omega,-\bar\omega)}(r) \, G_{(-\omega,-\bar\omega)}(0;r;\omega) \, G_{(-\omega,-\bar\omega)}(0;r;\bar\omega)}
        {r^{1+\omega+\bar{\omega}}}\\
        &
        \hskip 70pt
         \cdot
        \Bigg|
            \frac{\omega-\bar{\omega}}{2\omega\log M}P'\Big(\frac{\log M/r}{\log M}\Big)\;            + \; \frac{\omega+\bar\omega}{(2\omega\log M)^2} P''\Big(\frac{\log M/r}{\log M}\Big)
          \;  - \; \frac{\omega+\bar\omega}{(2\omega\log M)^3} P'''\Big(\frac{\log M/r}{\log M}\Big) \Bigg|^2 \\
          &
          \\
    & \hskip 30pt + 2\Re \sum_{M^{1-\Upsilon}<r\leq M}
        \frac{\mu(r)^2 \tau_{(-\omega,-\bar\omega)}(r) G_{(-\omega,-\bar\omega)}(0;r;\omega)G_{(-\omega,-\bar{\omega})}(-2\bar\omega;r;\bar\omega)}
        {r^{1+\omega-\bar{\omega}}\zeta(1-\omega-\bar\omega)\zeta(1-2\bar\omega)}
          \\
    & \hskip 70pt \cdot \Bigg(M^{-2\bar\omega}\sum_{j=2}^{\infty}\frac{P^{(j)}(0)}{(-2\bar\omega)^{j+1}\log^j M}\bigg) \bigg(
            \frac{\omega-\bar{\omega}}{2\omega\log M}P'\Big(\frac{\log M/r}{\log M}\Big)\\
    & \hskip 160pt
            + \frac{\omega+\bar\omega}{(2\omega\log M)^2} P''\Big(\frac{\log M/r}{\log M}\Big)
         \;   - \; \frac{\omega+\bar\omega}{(2\omega\log M)^3} P'''\Big(\frac{\log M/r}{\log M}\Big) \Bigg)
             \\
             &
             \\
    & \hskip 10pt +\sum_{M^{1-\Upsilon}<r\leq M}
        \frac{\mu(r)^2 \tau_{(-\omega,-\bar\omega)}(r) G_{(-\omega,-\bar{\omega})}(-2\omega;r;\omega)G_{(-\omega,-\bar{\omega})}(-2\bar\omega;r;\bar\omega)}
        {r^{1-\omega-\bar{\omega}}\zeta(1-2\omega)\zeta(1-\omega-\bar\omega)^2\zeta(1-2\bar\omega)}   \cdot
        \bigg|M^{-2\omega}\sum_{j=2}^{\infty}\frac{P^{(j)}(0)}{(-2\omega)^{j+1}\log^j M}\bigg|^2 \\
    & \hskip 170pt + \; \mathcal{O}\left(\frac{(\log\log K)^3}{\log^2 K}\right).
  \end{split}
\]

Next, lemma \ref{lemma:r-sum--} implies that
\[
  \begin{split}
    \cV_{(-\omega,-\bar{\omega})}^{>}(\omega)
    & =  \big(1+ \mathcal O\big(|\omega|\big)\big) \int\limits_{M^{1-\Upsilon}}^{M} \Bigg|
            \frac{\omega-\bar{\omega}}{2\omega\log M}P'\Big(\frac{\log M/t}{\log M}\Big)
         \;   + \; \frac{\omega+\bar\omega}{(2\omega\log M)^2} P''\Big(\frac{\log M/t}{\log M}\Big) \\
    & \hskip 215pt
        - \; \frac{\omega+\bar\omega}{(2\omega\log M)^3} P'''\Big(\frac{\log M/t}{\log M}\Big) \Bigg|^2 \frac{dt}{t^{1+\omega+\bar\omega}} \\
    & + 2\Re \left\{ \hskip-10pt\phantom{\Bigg |^{\int\limits^.}}    \frac{1+ \mathcal O(|\omega|)}{\zeta(1-\omega-\bar\omega)} \bigg(M^{-2\bar\omega}\sum_{j=2}^{\infty} \frac{P^{(j)}(0)}{(-2\bar\omega\log M)^j}\bigg)\right.
      \\
    &\hskip -60pt \left. \cdot \int\limits_{M^{1-\Upsilon}}^{M}  \Bigg(
            \frac{\omega-\bar{\omega}}{2\omega\log M}P'\Big(\frac{\log M/t}{\log M}\Big)
            + \frac{\omega+\bar\omega}{(2\omega\log M)^2} P''\Big(\frac{\log M/t}{\log M}\Big) \; - \; \frac{\omega+\bar\omega}{(2\omega\log M)^3} P'''\Big(\frac{\log M/t}{\log M}\Big) \Bigg) \frac{dt}{t^{1+\omega-\bar\omega}} \phantom{\Bigg |^{\int\limits^.}} \hskip -10pt \right\} \\
            &
            \\
    &\hskip -10pt  + \; \frac{1+ \mathcal O\big(|\omega|\big)}{\zeta(1-\omega-\bar\omega)} \left(M^{(1-\Upsilon)(\omega+\bar\omega)}-M^{\omega+\bar\omega}\right)
      \cdot  \left | M^{-2\omega}\sum_{j=2}^{\infty}\frac{P^{(j)}(0)}{(-2\omega\log M)^j}\right |^2 \; + \; \mathcal O\left(\frac{(\log\log K)^3}{\log^2 K}\right).
  \end{split}
\]

Making the change of variable $\frac{\log M/t}{\log M} \mapsto x$, we see that
\[
  \begin{split}
    \cV_{(-\omega,-\bar{\omega})}^{>}(\omega)
    & = \frac{1+ \mathcal O\big(|\omega|\big)}{\log M}\int\limits_{0}^{\Upsilon} M^{-(1-x)(\omega+\bar\omega)}\; \Bigg|
            \frac{\omega-\bar{\omega}}{2\omega}P'(x)
           \; + \; \frac{\omega+\bar\omega}{(2\omega)^2\log M} P''(x) \\
    & \hskip 250pt
        - \;\frac{\omega+\bar\omega}{(2\omega)^3(\log M)^2} P'''(x) \Bigg|^2 dx \\
    &\hskip -15pt  + 2\Re \left\{ \hskip-10pt\phantom{\Bigg |^{\int\limits^.}}   \frac{1+ \mathcal O\big(|\omega|)}{\zeta(1-\omega-\bar\omega\big)}
            \bigg(M^{-2\bar\omega}\sum_{j=2}^{\infty}\frac{P^{(j)}(0)}{(-2\bar\omega\log M)^j}\bigg)
        \int\limits_{0}^{\Upsilon} M^{-(1-x)(\omega-\bar\omega)}\right.  \\
    & \hskip 100pt
       \left.  \cdot \bigg( \frac{\omega-\bar{\omega}}{2\omega}P'(x)
            + \frac{\omega+\bar\omega}{(2\omega)^2\log M} P''(x)
            - \frac{\omega+\bar\omega}{(2\omega)^3(\log M)^2} P'''(x) \bigg) dx \hskip-10pt\phantom{\Bigg |^{\int\limits^.}}  \right\}
 \\
            &
            \\
    & \hskip -30pt + \; \frac{1+ \mathcal O\big(|\omega|\big)}{\zeta(1-\omega-\bar\omega)}
        \left(M^{(1-\Upsilon)(\omega+\bar\omega)}-M^{\omega+\bar\omega}\right)
       \cdot \left | M^{-2\omega}\sum_{j=2}^{\infty}\frac{P^{(j)}(0)}{(-2\omega\log M)^j}\right |^2    + \; \mathcal O\left(\frac{(\log\log K)^3}{\log^2 K}\right).
  \end{split}
\]
Hence by \eqref{eqn:Psi}, we obtain
\begin{equation}\label{eqn:PsiV>--}
  \begin{split}
    & \Psi(-\omega,-\bar{\omega})\cV_{(-\omega,-\bar{\omega})}^{>}(\omega) \;  = \; \frac{K}{4} \left(\frac{K}{4\pi}\right)^{-2(\omega+\bar{\omega})}
        \widetilde{\Phi}(1) \Big(1+ \mathcal O(|\omega|\big)\Big) \\
    &  \cdot \left[
    \hskip-10pt\phantom{\Bigg |^{\int\limits^.}}
     \frac{-1}{(\omega+\bar{\omega})\log M}
            \int\limits_{0}^{\Upsilon} M^{-(1-x)(\omega+\bar\omega)} \left |
            \frac{\omega-\bar{\omega}}{2\omega}P'(x)
            + \frac{\omega+\bar\omega}{(2\omega)^2\log M} P''(x)
            - \frac{\omega+\bar\omega}{(2\omega)^3(\log M)^2} P'''(x) \right |^2 dx \right. \\
    & \hskip 20pt + 2\Re   \left\{
    \hskip-10pt\phantom{\Bigg |^{\int\limits^.}}
       M^{-2\bar\omega} \sum_{j=2}^{\infty}\frac{P^{(j)}(0)}{(-2\bar\omega\log M)^j}
             \int\limits_{0}^{\Upsilon} M^{-(1-x)(\omega-\bar\omega)} \bigg(
            \frac{\omega-\bar{\omega}}{2\omega}P'(x)\right. \\
    & \hskip 200pt
        \left.    + \frac{\omega+\bar\omega}{(2\omega)^2\log M} P''(x)
            - \frac{\omega+\bar\omega}{(2\omega)^3(\log M)^2} P'''(x) \bigg) dx
            \hskip-10pt\phantom{\Bigg |^{\int\limits^.}}
            \right \} \\
    & \hskip 106pt + \left. \left(M^{(1-\Upsilon)(\omega+\bar\omega)}-M^{\omega+\bar\omega}\right)
        \left |M^{-2\omega}\sum_{j=2}^{\infty}\frac{P^{(j)}(0)}{(-2\omega\log M)^j}\right |^2
        \hskip-8pt\phantom{\Bigg |^{\int\limits^.}}
         \right]\\
         &\hskip 160pt  + \; \mathcal  O\left( \frac{(\log\log K)^4}{\log K} K  \right).
  \end{split}
\end{equation}

\subsubsection{The cases $(\alpha,\beta)=(\omega,-\bar{\omega})$ and $(\alpha,\beta)=(-\omega,\bar{\omega})$}

We first consider the case $(\alpha,\beta)=(\omega,-\bar{\omega})$.
By lemma \ref{lemma:T}, we have
\[
  \begin{split}
    S_{(\omega,-\bar\omega)}(r;\omega)
      & = \underset {s=0}{\Res} \; \frac{\mu(r) \, G_{(\omega,-\bar\omega)}(s;r;\omega)}
                {s\zeta(1+s+\omega-\bar\omega)}
                \sum_{j=0}^{\infty} \frac{1}{(s\log M)^j}
                P^{(j)}\Big(\frac{\log M/r}{\log M}\Big)
      \; + \; \mathcal O\left(\frac{E(r)}{\log^2 K}\right),
  \end{split}
\]
and
\[
  \begin{split}
    S_{(\omega,-\bar\omega)}(r;\bar\omega)
      & =  \underset {s=0} {\Res} \; \frac{\mu(r) \, G_{(\omega,-\bar\omega)}(s;r;\bar\omega) \, \zeta(1+s+2\bar\omega)}
                {s\zeta(1+s+\omega+\bar\omega) \, \zeta(1+s)}
                \sum_{j=0}^{\infty} \frac{1}{(s\log M)^j}
                P^{(j)}\Big(\frac{\log M/r}{\log M}\Big) \\
      & \quad + \frac{\mu(r) \, G_{(\omega,-\bar{\omega})}(-2\bar\omega;r;\bar\omega) \, r^{2\bar\omega}}
                {\zeta(1+\omega-\bar\omega)\zeta(1-2\bar{\omega})}
                 M^{-2\bar\omega} \sum_{j=2}^{\infty} \frac{P^{(j)}(0)}{(-2\bar\omega)^{j+1}\log^j M}
      \; + \; \mathcal O\left(\frac{E(r)}{\log^2 K}\right).
  \end{split}
\]
We first write the Taylor expansion of
$G_{(\omega,-\bar\omega)}(s;r;\omega)/\zeta(1+s+\omega-\bar\omega)$
as $a_0+a_1s+a_2s^2+a_3s^3+\cdots$.
Then we have
\[
  \begin{split}
     a_0 & = G_{(\omega,-\bar\omega)}(0;r;\omega) (\omega-\bar\omega) \; + \; \mathcal O\left(E(r)|\omega|^2\right), \\
     a_1 & = G_{(\omega,-\bar\omega)}(0;r;\omega) \;  + \; \mathcal O\big(E(r)|\omega|\big), \\
     a_2 & \ll_n E(r), \quad \mathrm{if}\ n\geq2.
  \end{split}
\]
Next, consider the Taylor expansion
$$\frac{G_{(\omega,-\bar\omega)}(s;r;\bar\omega) \,\zeta(1+s+2\bar\omega)}
{\zeta(1+s+\omega+\bar\omega) \,\zeta(1+s)} \; = \; b_0+b_1s+b_2s^2+b_3s^3+\cdots$$
Then we have
\[
  \begin{split}
     b_0 & = 0, \\
     b_1 & = \frac{G_{(\omega,-\bar\omega)}(0;r;\bar\omega)\, (\omega+\bar\omega)}{2\bar\omega}
           \; + \; \mathcal O\big(E(r)|\omega|\big), \\
     b_2 & = \frac{G_{(\omega,-\bar\omega)}(0;r;\bar\omega)\, (\bar\omega-\omega)}{(2\bar\omega)^2} \;  + \; \mathcal O\big(E(r)\big), \\
     b_3 & = \frac{-G_{(\omega,-\bar\omega)}(0;r;\bar\omega)\,(\bar\omega-\omega)}{(2\bar\omega)^3}
         \;   +\; \mathcal  O\left(E(r)|\omega|^{-1}\right).
  \end{split}
\]
So we get
\[
  \begin{split}
    S_{(\omega,-\bar\omega)}(r;\omega)
      & = \mu(r)G_{(\omega,-\bar\omega)}(0;r;\omega)
           \left( (\omega-\bar\omega)P\Big(\frac{\log M/r}{\log M}\Big)
           + \frac{1}{\log M}    P'\Big(\frac{\log M/r}{\log M}\Big) \right) \\
      & \hskip 265pt  + \; \mathcal O\left(\frac{\log\log K}{\log^2 K}\right),
  \end{split}
\]
and
\[
  \begin{split}
    S_{(\omega,-\bar\omega)}(r;\bar\omega)
      & = \mu(r)G_{(\omega,-\bar\omega)}(0;r;\bar\omega)
            \Bigg( \frac{\omega+\bar\omega}{2\bar\omega\log M}P'\Big(\frac{\log M/r}{\log M}\Big) \\
      & \hskip 30pt
                + \frac{\bar\omega-\omega}{(2\bar\omega\log M)^2}
                    P''\Big(\frac{\log M/r}{\log M}\Big)
                + \frac{\omega-\bar\omega}{(2\bar\omega\log M)^3}
                    P'''\Big(\frac{\log M/r}{\log M}\Big)\Bigg) \\
      & \hskip 50pt + \frac{\mu(r) \, G_{(\omega,-\bar{\omega})}(-2\bar\omega;r;\bar\omega) \, r^{2\bar\omega}}
                {\zeta(1+\omega-\bar\omega)}
                 M^{-2\bar\omega} \sum_{j=2}^{\infty} \frac{P^{(j)}(0)}{(-2\bar\omega\log M)^j}
       + \; \mathcal O\left(\frac{\log\log K}{\log^2 K}\right).
  \end{split}
\]
By \eqref{eqn:cV<>}, we have
\[
  \begin{split}
    & \cV_{(\omega,-\bar{\omega})}^{>}(\omega) \; = \sum_{M^{1-\Upsilon}<r\leq M}
        \frac{\mu(r)^2 \, \tau_{(\omega,-\bar\omega)}(r) \, G_{(\omega,-\bar\omega)}(0;r;\omega) \, G_{(\omega,-\bar\omega)}(0;r;\bar\omega)}
        {r^{1+\omega+\bar{\omega}}}
        \\
    &  \hskip 20pt \cdot \left(
    \hskip-10pt
    \phantom{\Bigg|^1}
    \left(
     (\omega-\bar\omega)P\Big(\frac{\log M/r}{\log M}\Big)
           + \frac{1}{\log M}    P'\Big(\frac{\log M/r}{\log M}\Big) \right)
           \bigg( \frac{\omega+\bar\omega}{2\bar\omega\log M}P'\Big(\frac{\log M/r}{\log M}\Big)\right. \\
    & \hskip 167pt   \left.  + \;\frac{\bar\omega-\omega}{(2\bar\omega\log M)^2}
                    P''\Big(\frac{\log M/r}{\log M}\Big)
                + \frac{\omega-\bar\omega}{(2\bar\omega\log M)^3}
                    P'''\Big(\frac{\log M/r}{\log M}\Big)
                    \phantom{\Bigg|^1} \hskip-10pt
                    \right) \\
                    &
                    \\
    & \quad\ + \sum_{M^{1-\Upsilon}<r\leq M}
        \frac{\mu(r)^2 \tau_{(\omega,-\bar\omega)}(r) G_{(\omega,-\bar\omega)}(0;r;\omega)G_{(\omega,-\bar{\omega})}(-2\bar\omega;r;\bar\omega)}
        {r^{1+\omega-\bar{\omega}}\zeta(1+\omega-\bar\omega)} \\
    & \hskip 100pt \cdot
        \bigg(M^{-2\bar\omega} \sum_{j=2}^{\infty} \frac{P^{(j)}(0)}{(-2\bar\omega\log M)^j}\bigg)
        \left( (\omega-\bar\omega)P\Big(\frac{\log M/r}{\log M}\Big)
           + \frac{1}{\log M}    P'\Big(\frac{\log M/r}{\log M}\Big) \right)
             \\
             &
             \\
    & \hskip 190pt +\; \mathcal O\left(\frac{(\log\log K)^3}{\log^2 K}\right).
  \end{split}
\]
Applying lemma \ref{lemma:r-sum--}, with
the change of variables from $\frac{\log M/t}{\log M}$ to $x$,
 together with \eqref{eqn:Psi}, we see that
\begin{equation}\label{eqn:PsiV>+-}
  \begin{split}
    &  \Psi(\omega,-\bar{\omega}) \cV_{(\omega,-\bar{\omega})}^{>}(\omega) \; =  \; - \frac{K}{4} \left(\frac{K}{4\pi}\right)^{-2\bar{\omega}}
        \widetilde{\Phi}(1) (1+\mathcal{O}(|\omega|)) \\
    & \hskip 20pt \cdot \left[ \frac{1}{(\omega-\bar{\omega})\log M}
            \int\limits_{0}^{\Upsilon} M^{-(1-x)(\omega+\bar\omega)}
            \left( (\omega-\bar\omega)(\log M) P(x) + P'(x)\right)
            \right. \\
    &  \hskip 150pt
            \cdot \Big( \frac{\omega+\bar\omega}{2\bar\omega}P'(x)
                + \frac{\bar\omega-\omega}{(2\bar\omega)^2\log M} P''(x)
                + \frac{\omega-\bar\omega}{(2\bar\omega)^3(\log M)^2}
                    P'''(x)\Big)  dx \\
    & \hskip 50pt \left.
     + \; M^{-2\bar\omega} \sum_{j=2}^{\infty}\frac{P^{(j)}(0)}{(-2\bar\omega\log M)^j}
        \int\limits_{0}^{\Upsilon} M^{-(1-x)(\omega-\bar\omega)}
            \cdot \Big( (\omega-\bar\omega)(\log M) P(x) + P'(x)\Big) \; dx \right] \\
    & \hskip 160pt + \; \mathcal O\left( \frac{(\log\log K)^4}{\log K} K  \right).
  \end{split}
\end{equation}
By the same argument we also obtain
\begin{equation}\label{eqn:PsiV>-+}
  \begin{split}
    &  \Psi(-\omega,+\bar{\omega}) \cV_{(-\omega,\bar{\omega})}^{>}(\omega) \; = \; - \frac{K}{4} \left(\frac{K}{4\pi}\right)^{-2\omega}
        \widetilde{\Phi}(1) (1+\mathcal{O}\big(|\omega|)\big) \\
    & \hskip 30pt  \cdot \left[ \frac{1}{(-\omega+\bar{\omega})\log M}
            \int\limits_{0}^{\Upsilon} M^{-(1-x)(\omega+\bar\omega)}
            \left( (-\omega+\bar\omega)(\log M) P(x) + P'(x)\right) \right. \\
    &  \hskip 80pt
            \cdot \Big( \frac{\omega+\bar\omega}{2\bar\omega}P'(x)
                + \frac{\omega-\bar\omega}{(2\bar\omega)^2\log M} P''(x)
                - \frac{\omega-\bar\omega}{(2\bar\omega)^3(\log M)^2}
                    P'''(x)\Big)  dx \\
    & \hskip 5pt \left.
     +\;  M^{-2\omega} \sum_{j=2}^{\infty}\frac{P^{(j)}(0)}{(-2\omega\log M)^j}
        \int\limits_{0}^{\Upsilon} M^{-(1-x)(-\omega+\bar\omega)}
            \cdot \Big( (-\omega+\bar\omega)(\log M) P(x) + P'(x)\Big) \,dx \right] \\
    & \hskip 120pt + \; \mathcal O\left( \frac{(\log\log K)^4}{\log K} K  \right).
  \end{split}
\end{equation}

\subsection{Conclusion}
Recall the harmonic weight
$$\omega_f := \frac{12\zeta(2)}{(k-1)}\cdot \frac{1}{ L(1,\operatorname{sym}^2 f)}.$$
Now $\Phi$ is a non-negative smooth function supported on $[1,2]$
such that $\Phi(t)\ll1$, and $\int\limits_{1}^{2}\Phi(t)dt\gg1$. It follows from (\ref{eq:harmonicsum}) that
 \begin{align} \label{A(kPhi-asymp}
    \cA(K, \Phi)
     & = \sum_{k\equiv2(4)}\Phi\left(\frac{k-1}{K}\right)\underset{f\in H_k}{{\sum}}\omega_f \\
     & = \sum_{k\equiv2(4)}\Phi\left(\frac{k-1}{K}\right) \left(1+ \mathcal O\left((2^{-k}\right)\right)
    \; = \; \frac{K}{4}\widetilde{\Phi}(1) \; + \; \mathcal  O\left(K^{-B}\right). \nonumber
  \end{align}

Combine \eqref{eqn:AtoS} and \eqref{eqn:PsiV<++}--\eqref{eqn:PsiV>-+} and make the change of variables
$\delta=\frac{u}{\log K}$,  $\; t=\frac{v}{\log K}$.
If we then take $M=K^{1-5\vartheta}$, it follows that
\begin{equation}\label{eqn:A<}
  \frac{1}{\cA(K,\Phi)}\cA\Big(\big\{|LM(1/2+\omega,f)|^2\big\}; K, \Phi\Big)
 \;  = \; \sV(u,v) \; + \; \mathcal O\left(\frac{(\log\log K)^4}{\log K}\right).
\end{equation}
Here
\begin{equation}\label{eqn:V()}
  \sV(u,v) := \sV_1(u,v) + \sV_2(u,v) + \sV_3(u,v),
\end{equation}

\begin{equation*}
  \begin{split}
    \sV_1(u,v) :=  1 +
            \frac{1}{2u(1-5\vartheta)}
            \int\limits_{0}^{\Upsilon} e^{-2u(1-x)(1-5\vartheta)}
            \left(P^{'}(x)\right)^2 dx,
  \end{split}
\end{equation*}
\begin{equation*}
  \begin{split}
    & \sV_2(u,v)  := e^{-4u}
        \left[  \hskip -156pt
          \phantom{\left|  \sum_{j=2}^{\infty}\frac{Q^{(j)}(0)}{(-2(u+iv)(1-5\vartheta)(1-\Upsilon))^j} \right |^2 }
         \frac{-1}{2u(1-5\vartheta)}
            \int\limits_{0}^{\Upsilon} e^{-2u(1-x)(1-5\vartheta)}
             \right.
             \\
            &
            \hskip 90pt \cdot \Bigg |
            \frac{iv}{u+iv}P'(x) \;
            +\;  \frac{u}{2(u+iv)^2(1-5\vartheta)} P''(x)
           \;  - \; \frac{u}{4(u+iv)^3(1-5\vartheta)^2} P'''(x) \Bigg|^2 dx  \\
    & \hskip 60pt + 2\Re \left\{   e^{-2(1-5\vartheta)(u-iv)} \sum_{j=2}^{\infty}\frac{P^{(j)}(0)}{(-2(u-iv)(1-5\vartheta))^j}
            \right. \\
    & \hskip 5pt
    \left.
     \cdot \int\limits_{0}^{\Upsilon} e^{-2iv(1-x)(1-5\vartheta)}\bigg( \frac{iv}{u+iv}P'(x)
            + \frac{u}{2(u+iv)^2(1-5\vartheta)} P''(x)
            - \frac{u}{4(u+iv)^3(1-5\vartheta)^2} P'''(x) \bigg) \;  dx\right \} \\
    & \hskip 100pt + \left(e^{2u(1-5\vartheta)(1-\Upsilon)}-e^{2u(1-5\vartheta)}\right)
        \left | \, e^{-2(u+iv)(1-5\vartheta)}\sum_{j=2}^{\infty}\frac{P^{(j)}(0)}{(-2(u+iv)(1-5\vartheta))^j} \,\right |^2 \\
    & \hskip 135pt + \left(1-e^{2u(1-5\vartheta)(1-\Upsilon)}\right)
            \left | e^{-2(u+iv)(1-5\vartheta)}\sum_{j=2}^{\infty}\frac{P^{(j)}(0)}{(-2(u+iv)(1-5\vartheta))^j} \right. \\
    & \hskip 165pt
        \left.  \left.  + \; e^{-2(u+iv)(1-5\vartheta)(1-\Upsilon)}
            \sum_{j=2}^{\infty}\frac{Q^{(j)}(0)}{(-2(u+iv)(1-5\vartheta)(1-\Upsilon))^j} \right |^2
      \phantom{\Bigg\{^{\int_0^1}}   \hskip -15pt
             \right] ,
  \end{split}
\end{equation*}
\begin{equation*}
  \begin{split}
    & \sV_3(u,v)  :=  - 2\Re \left\{ e^{-2(u+iv)} \left[ \frac{1}{-2iv(1-5\vartheta)}
            \int\limits_{0}^{\Upsilon} e^{-2u(1-x)(1-5\vartheta)} \Big( -2iv(1-5\vartheta)P(x) + P'(x)\Big)
            \right. \right. \\
    &  \hskip 67pt
            \cdot
            \Bigg( \frac{u}{(u+iv)}P'(x)\;
                + \; \frac{iv}{2(u+iv)^2(1-5\vartheta)} P''(x)
               \; -  \; \frac{iv}{4(u+iv)^3((1-5\vartheta))^2} P'''(x)\Bigg) \, dx \\
    & \hskip 10pt +  e^{-2(u+iv)(1-5\vartheta)}
        \sum_{j=2}^{\infty}\frac{P^{(j)}(0)}{(-2(u+iv)(1-5\vartheta))^j}
        \,  \int\limits_{0}^{\Upsilon} e^{2iv(1-x)(1-5\vartheta)}
        \Big( -2iv(1-5\vartheta) P(x) + P'(x)\Big) \, dx \\
    & \hskip 60pt + \left(1-e^{2iv(1-\Upsilon)(1-5\vartheta)}\right)
        \Bigg( e^{-2(u+iv)(1-5\vartheta)}\sum_{j=2}^{\infty}\frac{P^{(j)}(0)}{(-2(u+iv)(1-5\vartheta))^j} \\
    & \hskip 155pt
   \left. \left.         + \; e^{-2(u+iv)(1-5\vartheta)(1-\Upsilon)}
                \sum_{j=2}^{\infty}\frac{Q^{(j)}(0)}{(-2(u+iv)(1-5\vartheta)(1-\Upsilon))^j} \Bigg)
                 \phantom{\int\limits_{0}^{\Upsilon}}\hskip-7pt \right] \right\}.
  \end{split}
\end{equation*}

We  end this section by proving the following  upper bound for $\sV(u,v)$,
which will be used when $u$ is large.

\begin{lemma}\label{lemma:V-l}
 Choose
  \begin{align*}
  &   P(x)=3\left(\frac{x}{\Upsilon}\right)^2-2\left(\frac{x}{\Upsilon}\right)^3, \qquad
    Q(x) := 1-P\Big(\Upsilon+(1-\Upsilon)x\Big),\\
  & \vartheta=10^{-10}, \qquad
      \Upsilon=0.64, \qquad S= \frac{\pi}{4(1-\Upsilon)(1-20\vartheta)}, \qquad
      R=4.
    \end{align*}

  For $u\geq 10$ and $|v|\leq 5u$, we have
  \begin{equation*}
    \sV(u,v) \leq 1 + e^{-u/2}.
  \end{equation*}
\end{lemma}

\begin{proof}
  To prove the lemma, we will show that for $u\geq 10$ and $|v|\leq 5u$,
  \begin{equation}\label{eqn:Vj<-l}
    \begin{split}
       \sV_1(u,v) & \leq 1 + \frac{1}{2}e^{-u/2}, \\
       \sV_2(u,v) & \leq e^{-4u}, \\
       \sV_3(u,v) & \leq e^{-2u}.
    \end{split}
  \end{equation}
  In fact, much better bounds can be proved,
  but this will be good enough for our applications.

  Recall the definition of $\sV_j(u,v)$ above.
  Using the fact that $|P'(x)|\leq 3/(2\Upsilon)$ if $x\in[0,\Upsilon]$,
  one may easily obtain the first inequality in \eqref{eqn:Vj<-l}.
  Indeed, we have
  \begin{equation*}
    \begin{split}
      1 \; \leq \; \sV_1(u,v)  \; \leq \; 1 +
            \frac{\Upsilon}{2u(1-5\vartheta)}
             e^{-2u(1-\Upsilon)(1-5\vartheta)}
            \left(\frac{3}{2\Upsilon}\right)^2  \leq \; 1 + \frac{1}{2}e^{-u/2},
    \end{split}
  \end{equation*}
  for $u\geq10$.  Next we will consider $\sV_2(u,v)$.
  For $x\in[0,\Upsilon]$
  \[
    |P''(x)|\leq \frac{6}{\Upsilon^2},\qquad
    P'''(x)= \frac{12}{\Upsilon^3}.
  \]
  Hence
  \begin{equation*}
    \begin{split}
      |\sV_2(u,v)|
      & \leq e^{-4u}
        \Bigg\{ \frac{\Upsilon}{2u(1-5\vartheta)}
            e^{-2u(1-x)(1-5\vartheta)} \bigg(
            \Big(\frac{3}{2\Upsilon}\Big)^2  + \Big(\frac{6}{u\Upsilon^2}\Big)^2
            + \Big(\frac{12}{u^2\Upsilon^3}\Big)^2 \bigg) \\
      & \hskip 70pt + 2 \bigg(   e^{-2u(1-5\vartheta)}
            \Big(\frac{6}{u^2\Upsilon^2}+\frac{12}{u^3\Upsilon^3}\Big)
            \Upsilon  \Big( \frac{3}{2\Upsilon}  + \frac{6}{u\Upsilon^2}
            + \frac{12}{u^2\Upsilon^3} \Big) \bigg) \\
      & \hskip 150pt + 6e^{-2u(1-5\vartheta)(1-\Upsilon)}
        \bigg(\Big(\frac{6}{u^2\Upsilon^2}\Big)^2+\Big(\frac{12}{u^3\Upsilon^3}\Big)^2\bigg) \Bigg\}.
            \end{split}
  \end{equation*}
  For $u\geq10$, we have $|\sV_2(u,v)|\leq e^{-4u}$,
  which gives the second inequality in \eqref{eqn:Vj<-l}.
  Finally, we will bound $\sV_3(u,v)$.
  Note that
  \begin{equation*}\label{eqn:V3to}
    \sV_3(u,v) = \sV_{31}(u,v) + \sV_{32}(u,v),
  \end{equation*}
  where
  \begin{align*}
     & \sV_{31}(u,v)  :=  - 2\Re \left\{ e^{-2(u+iv)}\frac{u}{-2iv(u+iv)(1-5\vartheta)}
            \int\limits_{0}^{\Upsilon} e^{-2u(1-x)(1-5\vartheta)} \left(P'(x)\right)^2 dx \right\},\\
            &
            \\
    &  \sV_{32}(u,v)  :=  - 2\Re \left\{ e^{-2(u+iv)} \left[\;
            \int\limits_{0}^{\Upsilon} e^{-2u(1-x)(1-5\vartheta)} P(x) \bigg( \frac{u}{(u+iv)}P'(x)
              \right. \right.
              \\
      &  \hskip 146pt
                + \frac{iv}{2(u+iv)^2(1-5\vartheta)} P''(x)
                - \frac{iv}{4(u+iv)^3((1-5\vartheta))^2} P'''(x)\bigg)  dx
               \\
      &  - \frac{1}{2(1-5\vartheta)}
            \int\limits_{0}^{\Upsilon} e^{-2u(1-x)(1-5\vartheta)} P'(x) \bigg(  \frac{1}{2(u+iv)^2(1-5\vartheta)} P''(x)
                - \frac{1}{4(u+iv)^3((1-5\vartheta))^2} P'''(x)\bigg)  dx \\
      &\hskip 15pt  +  e^{-2(u+iv)(1-5\vartheta)}
        \sum_{j=2}^{\infty}\frac{P^{(j)}(0)}{(-2(u+iv)(1-5\vartheta))^j}  \int\limits_{0}^{\Upsilon} e^{2iv(1-x)(1-5\vartheta)}
        \Big( -2iv(1-5\vartheta) P(x) + P'(x)\Big) dx \\
      & \hskip 90pt  + \left(1-e^{2iv(1-\Upsilon)(1-5\vartheta)}\right)
        \bigg( e^{-2(u+iv)(1-5\vartheta)}\sum_{j=2}^{\infty}\frac{P^{(j)}(0)}{(-2(u+iv)(1-5\vartheta))^j} \\
      & \hskip 160pt
      \left. \left.
            + \; e^{-2(u+iv)(1-5\vartheta)(1-\Upsilon)}
                \sum_{j=2}^{\infty}\frac{Q^{(j)}(0)}{(-2(u+iv)(1-5\vartheta)(1-\Upsilon))^j} \bigg) \right] \right\}.
  \end{align*}
  By the same argument as above we can show that $|\sV_{32}(u,v)|\leq \frac{1}{2}e^{-2u}$.
  We also have
  \[
    \begin{split}
       |\sV_{31}(u,v)| & \leq \frac{2ue^{-2u}}{(1-5\vartheta)}
        \left(\; \int\limits_{0}^{\Upsilon} e^{-2u(1-x)(1-5\vartheta)} \left(P'(x)\right)^2 dx\right)
        \bigg| \Re\Big\{ \frac{e^{-2iv}}{-2iv(u+iv)} \Big\} \bigg| \leq \frac{1}{2} e^{-2u}.
    \end{split}
  \]
  This establishes the third inequality in \eqref{eqn:Vj<-l} and
   completes the proof of the lemma.
\end{proof}


\section{The harmonic mollified second moment away from the critical point}



In this section, we get a bound for
$\frac{1}{\cA(K, \Phi)}\cA\Big(\left\{ |LM(1/2+\delta+it,f)|^2\right\}; K, \Phi\Big)$
when $$\frac{\log\log K}{\log K}\ll \delta \leq 1/2+\frac{10 \log\log K}{(1-\Upsilon)\log K}.$$
We will follow the method of Ricotta \cite[Appendix A]{ricotta2006real},
which is based on a classical Phragm\'{e}n--Lindel\"{o}f-type convexity principle.

\begin{theorem}\label{thm:A>}
  Let $0 < \Upsilon < 1$  and $M=K^{1-5\vartheta}$,
  with $0<\vartheta<1/100$ being a small constant.
  If $\delta\gg \frac{\log\log K}{\log K}$, then for any $0<a<2(1-\Upsilon)$, we have
  \[
    \frac{1}{\cA(K, \Phi)}\cA\Big(\left\{ |LM(1/2+\delta+it,f)-1|^2\right\}; K, \Phi\Big)
    \; \ll_{B,a} \; (1+|t|)^B M^{-a\delta},
  \]
  for some constant $B>0$ depending only on $\vartheta$.
  If $\frac{\log\log K}{\log K}\ll \delta
  \leq 1/2+\frac{10 \log\log K}{(1-\Upsilon)\log K}$,
  then for any $0<a<2(1-\Upsilon)$, we have
  \[
    \frac{\cA\Big(\left\{ |LM(1/2+\delta+it,f)|^2\right\}; K, \Phi\Big)
}{\cA(K, \Phi)}    \; = \; 1 \; + \; \mathcal{O}_{B,a} \Big( (1+|t|)^B M^{-a\delta} \Big),
  \]
  for some constant $B>0$ depending only on $\vartheta$.
\end{theorem}

The proof of  theorem \ref{thm:A>} (to be given at the end of this section)
requires  the following three lemmas.

\begin{lemma}\label{lemma:A-s}
  Let $\vartheta,\Upsilon,M$ be as in theorem \ref{thm:A>}.
  If $\delta = \frac{\log\log\log K}{\log K}$, then we have
  \[
    \frac{\cA\Big(\left\{ |LM(1/2+\delta+it,f)|^2\right\}; K, \Phi\Big) }{\cA(K, \Phi)} \;  \ll_B  \; \big(1+|t|\big)^B,
  \]
  for some constant $B>0$ depending only on $\vartheta$.
\end{lemma}

\begin{proof}
  If $|t|\ll \frac{\log\log K}{\log K}$, then by the argument in \S\ref{sec:tsm}  and \S\ref{sec:msm} it follows that
  \[
     \frac{\cA\Big(\left\{ |LM(1/2+\delta+it,f)|^2\right\}; K, \Phi\Big)}{\cA(K, \Phi)} \; \ll \; 1,
  \]
  (see  \eqref{eqn:A<}, \eqref{eqn:V()}).
  So we may assume $|t|\gg \frac{\log\log K}{\log K}$.
  Let $\theta=\vartheta/2$.
  If $|t|\gg K^{\theta}$, then one can use the convexity bounds
  for $L(1/2+\delta+it,f)$ and $M(1/2+\delta+it,f)$ to deduce our claim,
  provided that $B>0$ is large enough.
  Consequently,  we only need to handle the case $\frac{\log\log K}{\log K} \ll |t| \ll K^{\theta}$.
  As in \S\ref{sec:msm}, by theorem \ref{thm:tsm}, we have
  \begin{equation*}
    \begin{split}
      &  \cA\Big(\left\{ |LM(1/2+\delta+it,f)|^2\right\}; K, \Phi\Big)   =  \underset{d}{{\sum}^\flat}\frac{1}{d^{1+2\delta}}
     \hskip-10pt    \underset{\substack{(m_1,n_1)=1\\(m_2,n_2)=1\\(m_1n_1m_2n_2,d)=1}}{{\sum}^\flat}
    \hskip-8pt
         \frac{\mu(m_1)\mu(m_2)F_{\Upsilon,M}(dm_1n_1)F_{\Upsilon,M}(dm_2n_2)}{(m_1n_1^2)^{1/2+\delta+it}(m_2n_2^2)^{1/2+\delta-it}}
          \\
      &\hskip 5pt \cdot  \left[\; \zeta(1+2\delta) \frac{\eta_{it}(m_1m_2)}{(m_1m_2)^{1/2+\delta}}
              \frac{K}{4} \int\limits_{0}^{\infty}\Phi(u)\, du \right. \; +
              \; \zeta(1-2\delta) \frac{\eta_{it}(m_1m_2)}{(m_1m_2)^{1/2-\delta}} \left(\frac{K}{4\pi}\right)^{-4\delta}
              \frac{K}{4} \int\limits_{0}^{\infty} \Phi(u)u^{-4\delta}\, du  \\
      &
      \left.
      \hskip 120pt - 2\Re \left\{\zeta(1-2it)\frac{\eta_\delta(m_1m_2)}{(m_1m_2)^{1/2-it}}
                \left(\frac{K}{4\pi}\right)^{-2\delta-2it}
                \frac{K}{4}\int\limits_{0}^{\infty}\Phi(u)u^{-2\delta-2it} du \right\} \right] \\
      & \hskip 220pt +\; \mathcal O\Big(K^{1-\vartheta+\ve}\Big) \\
      &\hskip 50pt =: \cS_1 + \cS_2 - 2 \Re \cS_3  + O_{\ve}\left(K^{1-\vartheta+\ve}\right).
    \end{split}
  \end{equation*}
  Now, following the same proof as in Hough \cite[\S5]{hough2012zero},
  for $t\ll K^\theta$, we obtain
  \[
     \frac{\cA\Big(\left\{ |LM(1/2+\delta+it,f)|^2\right\}; K, \Phi\Big)}{\cA(K, \Phi)}  \; \ll \;1.
  \]
  This completes the proof of the lemma.
\end{proof}

\begin{lemma}\label{lemma:A-l}
  Let $\vartheta,\Upsilon,M$ be as in theorem \ref{thm:A>}.
  If $\delta>1/2+\varepsilon$, then
  \[
    \frac{\cA\Big(\left\{ |LM(1/2+\delta+it,f)-1|^2\right\}; K, \Phi\Big)}{\cA(K, \Phi)} \;
    \ll_\ve  \; M^{-2(1-\Upsilon)(\delta-(1/2+\ve))}  \]
  for any $\ve>0$.
\end{lemma}

\begin{proof}
  From the shape of the mollifier \eqref{eqn:M(s)},
  we can deduce that
  \[
    LM(s,f) = 1 \; + \; \mathcal O\left( M^{(1-\Upsilon)(1+\ve-\Re(s))} \right),
  \]
  if $\Re(s)>1+\ve$ for any $\ve>0$.
  The lemma immediately follows from the above estimate.
\end{proof}

\begin{remark}
  If fact, by \eqref{eqn:M(s)}, for $\delta\geq \frac{1}{2}+\frac{9\log\log M^{1-\Upsilon}}{\log M^{1-\Upsilon}}$ and the choice $M=K^{1-5\vartheta}$, as in (\ref{eqn:A<}), we have
  \begin{equation} \label{eq:LMbound}
    LM(1/2+\delta+it,f) = 1 \; + \; \mathcal O\left( \frac{1}{\log K} \right).
 \end{equation}
  Indeed, let $b_f(n):=\sum\limits_{n=\ell m}\lambda_f(m)a_f(\ell)F_{\Upsilon,M}(\rad(\ell))$, then $LM(s,f)=\sum\limits_{n=1}^{\infty}b_f(n)/n^s$ if $\Re(s)>1$.
  Note that we have $b_f(1)=1$, $b_f(n)=0$ if $2\leq n\leq M^{1-\Upsilon}$,
  and $b_f(n)\ll \sum\limits_{n=\ell m}\tau(m)\tau(\ell) \ll \tau^3(n)$ for all $n\in\bbN$.
  So for $\delta\geq \frac{1}{2}+\frac{9\log\log M}{(1-\Upsilon)\log M}$, we have
  \[
    \begin{split}
       LM(1/2+\delta+it,f)-1 & \ll \sum_{n=M^{1-\Upsilon}}^{\infty}\frac{\tau^3(n)}{n^{1/2+\delta}}
        \;  \ll  \sum_{n=M^{1-\Upsilon}}^{\infty}\frac{\tau^3(n)}{n^{1+\frac{9\log\log n}{\log n}}} \\
         & \ll \sum_{n=M^{1-\Upsilon}}^{\infty}\frac{\tau^3(n)}{n(\log n)^{9}} \; \ll \int\limits_{M^{1-\Upsilon}}^{\infty}\frac{dt}{t(\log t)^{9}} \;\sum_{n\leq t}\tau^3(n)\\
         &
         \ll \frac{1}{\log K}.
    \end{split}
  \]
  Consequently, we may take $W_1=1+\frac{10 \log\log K}{(1-\Upsilon)\log K}$ in Selberg's lemma.
\end{remark}

We would like to thank Soundararajan for sending
his unpublished paper with Conrey \cite{conrey3000real}.
The following lemma is strongly based on \cite{conrey3000real}.
\begin{lemma}\label{lemma:firstmoment}
  Let $\vartheta,\Upsilon,M$ be as in theorem \ref{thm:A>}.
  If $\frac{\log\log K}{\log K}\ll \delta \leq 1/2+\frac{10 \log\log K}{(1-\Upsilon)\log K}$,
  then for any $0<a<2(1-\Upsilon)$, we have
  \[
    \frac{\mathcal{A}\Big( \big\{ LM(1/2+\delta+it,f) \big\};K,\Phi \Big)}{\mathcal{A}(K,\Phi)}\;
    = \; 1 + \mathcal{O}_{B,a} \Big( (1+|t|)^B M^{-a\delta} \Big),
  \]
  for some constant $B>0$ depending only on $\vartheta$.
\end{lemma}

\begin{proof}
  In the region $\Re(s)>1$, by \eqref{eqn:M(s)} we may write
  \[
    LM(s,f) = \sum_{n=1}^{\infty} \frac{1}{n^{s}} \left( \sum_{abc^2=n} \lambda_f(a)\lambda_f(b)\mu(b)\mu(bc)^2 F_{\Upsilon,M}(bc) \right).
  \]
  Using the Hecke relations we see that
  \[
    \begin{split}
      & \sum_{abc^2=n} \lambda_f(a)\lambda_f(b)\mu(b)\mu(bc)^2 F_{\Upsilon,M}(bc) = \sum_{abc^2=n} \sum_{d|(a,b)}\lambda_f\left(\frac{ab}{d^2}\right)\mu(b)\mu(bc)^2 F_{\Upsilon,M}(bc),
    \end{split}
  \]
  and setting $a=\alpha d$, $b=\beta d$, and $g=cd$, this becomes
  \[
    \begin{split}
      & \sum_{g^2|n} \lambda_f\left(\frac{n}{g^2}\right) \sum_{\alpha\beta=n/g^2}\mu(\beta g)^2 F_{\Upsilon,M}(\beta g) \sum_{cd=g}\mu(\beta d)   = \lambda_f(n) \sum_{\alpha\beta=n} \mu(\beta) F_{\Upsilon,M}(\beta),
    \end{split}
  \]
  since the terms with $g>1$ are easily seen to disappear. Thus
  \[
    LM(s,f) = \sum_{n=1}^{\infty} \frac{\lambda_f(n)}{n^{s}} c(n),
    \quad \textrm{where}\quad
    c(n) = \sum_{d|n} \mu(d) F_{\Upsilon,M}(d).
  \]
  We have $c(1)=1$; for $1<n\leq M^{1-\Upsilon}$, we have $c(n)= \sum_{d|n}\mu(d)=0$; and for $n>M^{1-\Upsilon}$,  we have $|c(n)|\leq \tau(n)$.

  We will first handle the case $\Re(s)= 1/2+\delta_0$ where
  $\delta_0 = 1/2+\frac{10 \log\log K}{(1-\Upsilon)\log K}$.
  Put $B(s,f):=LM(s,f)-1$. We consider
  \[
    \frac{1}{2\pi i} \int\limits_{3-i\infty}^{3+i\infty}\Gamma(w) B(w+s,f)X^w \, dw,
  \]
  where $X=K^{2-\vartheta}$.
  We shift the line of integration to $\Re(w)=-\delta_0+\delta_1$,
  where $\delta_1=\frac{\log\log\log K}{\log K}$.
  The pole at $w=0$ gives $B(s,f)$, and so we conclude that
  \[
    \begin{split}
      B(s,f) & = \frac{1}{2\pi i} \int\limits_{3-i\infty}^{3+i\infty} \Gamma(w)B(w+s,f)X^w \, dw
       \; - \; \frac{1}{2\pi i} \int\limits_{-\delta_0+\delta_1-i\infty}^{-\delta_0+\delta_1+i\infty} \Gamma(w)B(w+s,f)X^w \, dw \\
      & =: T_1(s,f) - T_2(s,f).
    \end{split}
  \]

  We first estimate the contribution of the $T_2(s,f)$ terms.
  By Cauchy's inequality and lemma \ref{lemma:A-s},
  for some constant $B>0$ we have
  \begin{align*}
      \frac{ \mathcal{A}\Big( \big\{ |T_2(s,f)| \big\};K,\Phi \Big)}{\mathcal{A}(K,\Phi)} & \ll \; X^{-\delta_0+\delta_1} \frac{ \mathcal{A}\left( \left\{ \int\limits_{-\delta_0+\delta_1-i\infty}^{-\delta_0+\delta_1+i\infty} |\Gamma(w)| |B(w+s,f)| \, |dw| \right\};K,\Phi \right)}{\mathcal{A}(K,\Phi)}
     \\
     & \ll \; X^{-\delta_0+\delta_1} \int\limits_{(-\delta_0+\delta_1)} |\Gamma(w)|  \frac{ \mathcal{A}\Big( \big\{ |LM(w+s,f)|\big\};K,\Phi \Big)}{\mathcal{A}(K,\Phi)}\;
      |dw| + X^{-\delta_0+\delta_1}  \\
     &  \ll \; (1+|t|)^B X^{-\delta_0+\delta_1}
     \ll  \;(1+|t|)^B K^{-2(1-2\vartheta)\delta_0}
     \ll \;(1+|t|)^B M^{-2(1-\Upsilon)\delta_0}.
  \end{align*}
  It remains now to estimate the $T_1$ contribution.
  Since $\frac{1}{2\pi i}\int_{(\alpha)}\Gamma(w)(X/n)^w dw = e^{-n/X}$,
  we see that
  \[
    \begin{split}
      T_1(s,f) & = \sum_{n=2}^{\infty} \frac{\lambda_f(n)c(n)}{n^s} e^{-n/X} \\
       & = \sum_{M^{1-\Upsilon}<n\leq X(\log K)^2} \frac{\lambda_f(n)c(n)}{n^s} e^{-n/X} + \mathcal{O}\Big(K^{-B}\Big).
    \end{split}
  \]
  In order to bound $\frac{1}{\mathcal{A}(K,\Phi)} \mathcal{A}\Big( \big\{ T_1(s,f) \big\};K,\Phi \Big)$,
  for $M^{1-\Upsilon}<n\leq X(\log K)^2$ we consider
  \[
    \frac{ \mathcal{A}\Big( \big\{ \lambda_f(n) \big\};K,\Phi \Big)
}{\mathcal{A}(K,\Phi)}  \;  =\; \frac{1}{\mathcal{A}(K,\Phi)} \sum_{k\equiv2(4)}\Phi\left(\frac{k-1}{K}\right)   \sum_{f\in H_k}  \omega_f \cdot \lambda_f(n).
  \]
  By lemma \ref{lemma:PTF} and lemma \ref{lemma:average-over-k}, we arrive at
  \[
    \begin{split}
         & \quad\ \frac{1}{\mathcal{A}(K,\Phi)} \sum_{k\equiv2(4)}\Phi\left(\frac{k-1}{K}\right)
    \left(-2\pi\sum_{c=1}^{\infty}\frac{S(1,n;c)}{c}J_{k-1}\left(\frac{4\pi\sqrt{n}}{c}\right)\right) \\
         & \hskip 60pt = -2\pi \frac{1}{\mathcal{A}(K,\Phi)} \sum_{c=1}^{\infty}\frac{S(1,n;c)}{c} \sum_{k\equiv2(4)}\Phi\left(\frac{k-1}{K}\right) J_{k-1}\left(\frac{4\pi\sqrt{n}}{c}\right)\\
         & \hskip 60pt = -2\pi \frac{1}{\mathcal{A}(K,\Phi)} \sum_{c=1}^{\infty}\frac{S(1,n;c)}{c}
         \Bigg( \frac{1}{4}\Phi\left(\frac{4\pi\sqrt{n}}{cK}\right)
          \\
         & \hskip 110pt   + \frac{\sqrt{c}K}{8\sqrt{\pi}n^{1/4}} \Im\left(e^{-2\pi i/8}e^{i\frac{4\pi\sqrt{n}}{c}}\check{\Phi}\left(\frac{cK^2}{8\pi\sqrt{n}}\right)\right)    + \mathcal{O}\bigg(\frac{\sqrt{n}}{cK^3}\bigg) \Bigg).
    \end{split}
  \]
  If $n\leq X(\log K)^2 \ll K^{2-\vartheta/2}$,
  then $\frac{4\pi\sqrt{n}}{cK}\ll K^{-\vartheta/4}$ and
  $\frac{cK^2}{8\pi\sqrt{n}}\gg cK$. So $\Phi\left(\frac{4\pi\sqrt{n}}{cK}\right) = 0$; and
  $\check{\Phi}\left(\frac{cK^2}{8\pi\sqrt{n}}\right)\ll (cK)^{-B}$
  by repeated integration by parts in the definition of $\check{\Phi}$.
  Thus we see that if $n\leq X(\log K)^2$ then
  \[
    \frac{\mathcal{A}\Big( \big\{ \lambda_f(n) \big\};K,\Phi \Big)
}{\mathcal{A}(K,\Phi)}  \;   \ll \; \frac{\sqrt{n}}{K^4} \; \ll \; K^{-3}.
  \]
  Hence
  \[
    \frac{\mathcal{A}\Big( \big \{ T_1(s,f) \big\};K,\Phi \Big)
}{\mathcal{A}(K,\Phi)}  \;   \ll \; K^{-2},
  \]
  and
  \[
   \frac{\mathcal{A}\Big( \big\{ LM(1/2+\delta_0+it,f)-1 \big\};K,\Phi \Big)}{\mathcal{A}(K,\Phi)} \;
    \ll_{B,a} \; (1+|t|)^B M^{-2(1-\Upsilon)\delta_0}.
  \]
  This completes the proof for $\Re(s)=1/2+\delta_0 = 1+\frac{10 \log\log K}{(1-\Upsilon)\log K}$.
  And the result for $\frac{\log\log K}{\log K}\ll \delta \leq 1/2+\frac{10 \log\log K}{(1-\Upsilon)\log K}$ now follows from lemma \ref{lemma:A-s} and
  the convexity argument.
\end{proof}

\pagebreak

Now we are ready to give the proof of theorem \ref{thm:A>}.
\begin{proof}[Proof of theorem \ref{thm:A>}]
  Let $\delta_2$ be a large fixed constant.
  By lemmas \ref{lemma:A-s} and \ref{lemma:A-l}, and a
  Phragm\'{e}n--Lindel\"{o}f-type convexity principle for
  subharmonic functions which can be found in Kowalski
  \cite[Lemma 25]{kowalski1998rank}, it follows that
  \[
  \frac{\cA\Big(\left\{ |LM(1/2+\delta+it,f)-1|^2\right\}; K, \Phi\Big)}{\cA(K, \Phi)} \;
    \ll_{B,\ve}  \; \big(1+|t|\big)^B M^{\alpha(\delta)},
  \]
  where $\alpha(\delta)$ is the linear function satisfying
  \[
    \alpha(\delta_2)=-2(1-\Upsilon)\Big(\delta_2-(1/2+\ve)\Big), \quad
    \textrm{and} \quad
    \alpha\Big(\frac{\log\log\log K}{\log K}\Big)=0.
  \]
  This leads to
  \[
    \begin{split}
      &\frac{\cA\Big(\left\{ |LM(1/2+\delta+it,f)-1|^2\right\}; K, \Phi\Big)}{\cA(K, \Phi)} \; \ll_{B,\ve} \;
      \big(1+|t|\big)^B M^{-\frac{2(1-\Upsilon)(\delta_2-(1/2+\ve))}{\delta_2-(\log\log\log K)/(\log K)}\big(\delta-\frac{\log\log\log K}{\log K}\big)},
    \end{split}
  \]
  for $(\log\log\log K)/(\log K) \leq \delta \leq \delta_2$.

  Now choose $\ve$ small enough and $\delta_2$ large enough.
  Combining this with lemma \ref{lemma:firstmoment},
  the proof of theorem \ref{thm:A>} immediately follows.
\end{proof}

\section{Proof of theorem \ref{thm:rz}}\label{sec:thmrz}

In this section, we will prove that for at least $60\%$
(counted with weight $\Phi$ for the sum over $k$ and
with harmonic weight for the sum over forms)
of the odd modular L-functions we have $L(\sigma,f)>0$
for $\sigma\in(1/2,1]$. 

We apply lemma \ref{lemma:selberg} with the choices
$$
H=\frac{S}{\log K}, \qquad W_0 = \frac{1}{2}-\frac{R}{\log K},\qquad
W_1=1+\frac{10 \log\log K}{(1-\Upsilon)\log K},\qquad \phi(s) = LM(s, f),
$$
where $R$ and $S$ are fixed positive parameters which will be chosen later.
It follows that
\begin{equation}\label{eqn:sumtoI}
  4S \sum_{\substack{\beta\geq\frac{1}{2}-\frac{R}{\log K}\\0\leq\gamma\leq \frac{2S}{3\log K} \\ L(\beta+i\gamma,f)=0}}
  \cos\left(\frac{\pi \,\gamma\, \log K}{2S}   \right)
        \sinh\left(\frac{\pi \big(R+(\beta-1/2)\log K\big)}{2S}\right)
 \; \leq \; I_1(f)+I_2(f)+I_3(f),
\end{equation}
where
\begin{equation}\label{eqn:I(f)}
  \begin{split}
     I_1(f) \; & := \int\limits_{-S}^{S} \cos\left(\frac{\pi t}{2S}\right)
        \log\left|LM\left(\frac{1}{2}-\frac{R}{\log K}+i\frac{t}{\log K},f\right)\right|dt,  \\
     I_2(f) & \; := \int\limits_{-R}^{(W_1-1/2)\log K} \sinh\left(\frac{\pi(u+R)}{2S}\right)
        \log\left|LM\left(\frac{1}{2}+\frac{u}{\log K}+i\frac{S}{\log K},f\right)\right|^2 du, \\
     I_3(f) & \; := -\Re \int\limits_{-S}^{S} \cos\left(\pi\frac{(W_1-1/2)\log K -R+it}{2iS}\right)
        \log LM\left(W_1+i\frac{t}{\log K},f\right) dt.
  \end{split}
\end{equation}

Now, in the sum over zeros on the left hand side of (\ref{eqn:sumtoI}),
the weight $\cos\left(\frac{\pi\gamma}{2H}\right)$
can be replaced by $1$.
Indeed, if $\gamma=0$, we have $\cos(0)=1$;
and if there is a zero $\rho=\beta+i\gamma$ of $L(s,f)$ with
$\beta\geq\frac{1}{2}-\frac{R}{\log K}$ and
$0<\gamma\leq \frac{2S}{3\log K}$,
then we know that $\bar\rho$ is also a zero of $L(s,f)$,
and the contribution of these two zeros is
$\geq \left[\cos\left(\frac{\pi\gamma}{2H}\right)+\cos\left(\frac{-\pi\gamma}{2H}\right)\right]
\sinh\left(\frac{\pi (R+(\beta-1/2)\log K)}{2S}\right)
\geq \sinh\left(\frac{\pi (R+(\beta-1/2)\log K)}{2S}\right)$.
It follows that
\begin{equation}\label{eqn:sumtoI2}
  4S \sum_{\substack{\beta\geq\frac{1}{2}-\frac{R}{\log K}\\0\leq\gamma\leq \frac{2S}{3\log K} \\ L(\beta+i\gamma,f)=0}}
        \sinh\left(\frac{\pi \big(R+(\beta-1/2)\log K\big)}{2S}\right)
 \; \leq \; I_1(f)+I_2(f)+I_3(f),
\end{equation}

\begin{claim}\label{claim}
  We have
  \begin{equation*}
    I_1(f)+I_2(f)+I_3(f) \; \geq \; \begin{cases}
    4S\sinh\left(\frac{\pi R}{2S}\right), & \textrm{for all $f\in H_k,$}\\
    &
    \\
    12S\sinh\left(\frac{\pi R}{2S}\right), &
    \begin{matrix} & \textrm{if $L(s,f)$ has a zero $\rho=\beta+i\gamma$}\\  &\textrm{with $\beta\in (1/2,1]$ and $|\gamma|\leq \frac{2S}{3\log K}$.}\end{matrix} \end{cases}
  \end{equation*}
\end{claim}

\begin{proof}[Proof of Claim]
  Since we only consider $f \in H_k$ with $k\equiv2\pmod{4}$,
  it follows from remark \ref{remark1} that $L(1/2, f) = 0$ for all these $f.$ Hence we always know that the left hand side of \eqref{eqn:sumtoI2}
  exceeds $4S\sinh\left(\frac{\pi R}{2S}\right)$.

  Now suppose that $L(\beta+i\gamma,f)=0$ for some $\beta\in (1/2,1]$
  and $|\gamma|\leq \frac{2S}{3\log K}$.
  If $L(s,f)$ has a zero with $\beta>\frac{1}{2}+\frac{R}{\log K}$,
  then the contribution from this zero to the left hand side of \eqref{eqn:sumtoI2} would be
  $\geq 4S\sinh\left(\frac{\pi R}{S}\right)\geq 8S\sinh\left(\frac{\pi R}{2S}\right)$,
  since $\sinh(2x)\geq 2\sinh(x)$ for $x\geq0$.
  Let us now assume that $L(s,f)$ has a zero with $\beta=\frac{1}{2}+\frac{\xi}{\log K}$
  for some $0<\xi\leq R$.  The functional equation then implies that there is also a zero with $\beta=\frac{1}{2}-\frac{\xi}{\log K}$,
  and together they contribute
  $4S\left(\sinh\left(\frac{\pi (R+\xi}{2S}\right)
  +\sinh\left(\frac{\pi (R-\xi}{2S}\right)\right)
  \geq 8S\sinh\left(\frac{\pi R}{2S}\right)$
  to the left hand side of \eqref{eqn:sumtoI2}. This is because
  the minimum value of $\sinh(x+y)+\sinh(x-y)$
  for $0\leq y\leq x$ is attained at $y=0$.
  Then together with the contribution from the zero at $s=1/2$, this proves the claim.
\end{proof}

Let us now define
$$\cN_0(K, \Phi) \; := \sum_{k\equiv2(4)}\Phi\left(\frac{k-1}{K}\right)
\hskip -30pt
\underset{\text{for some} \;\beta\in (1/2,1], \; |\gamma|\leq \frac{2S}{3\log K} } {\underset{L(\beta+i\gamma,f)\, = \, 0\;  } {\sum_{f\in H_k}}}  \hskip -30pt\omega_f.$$
By claim \ref{claim}, we have
\[
  \frac{1}{2}\cA(K, \Phi) + \cN_{0}(K, \Phi)
\;  \leq \;\frac{
        \cA\Big(\big\{I_1(f)+I_2(f)+I_3(f)\big\}; K, \Phi\Big)}{8S\sinh\left(\frac{\pi R}{2S}\right)}.
\]
That is
\[
  \frac{\cN_{0}(K, \Phi)}{\cA(K, \Phi)}
 \; \leq \;
        \frac{\cA\Big(\big\{I_1(f)+I_2(f)+I_3(f)\big\}; K, \Phi\Big)}{8S\sinh\left(\frac{\pi R}{2S}\right)\, \cA(K, \Phi)}
        -\frac{1}{2}.
\]
Since the weighted geometric mean is less than the weighted
arithmetic mean, we have that
\[
  \frac{\cA\Big(\left\{\log |LM(1/2+\delta+it,f)|^2\right\}; K, \Phi\Big)}{\cA(K, \Phi)} \;
  \leq \; \log\left( \frac{\cA\Big(\left\{|LM(1/2+\delta+it,f)|^2\right\}; K, \Phi \Big)}{\cA(K, \Phi)}\right).
\]
It follows that
\begin{equation}\label{eqn:N0toJ}
  \begin{split}
    & \frac{\cN_{0}(K, \Phi)}{\cA(K, \Phi)} \; \leq \; \frac{ J_1(K;\Phi) + J_2(K;\Phi)}{8S\sinh\left(\frac{\pi R}{2S}\right)}
        \;  +\;
        \frac{ \cA\Big(\big\{I_3(f)\big\}; \Phi, K\Big)}{8S\sinh\left(\frac{\pi R}{2S} \,\right) \cA(K, \Phi)}
      \;  - \;\frac{1}{2},
  \end{split}
\end{equation}
where
\begin{equation*}
  \begin{split}
     J_1(K;\Phi) & := \int\limits_{0}^{S}  \cos\left(\frac{\pi t}{2S}\right)
        \log \left( \frac{
        \cA\left(\left\{\Big|LM\Big(\frac{1}{2}-\frac{R}{\log K}+i\frac{t}{\log K},f\Big)\Big|^2\right\}; K, \Phi\right) }{\cA(K, \Phi)}\right) dt,  \\
     J_2(K;\Phi) & := \int\limits_{-R}^{(W_1-\frac12)\log K} \sinh\left(\frac{\pi(u+R)}{2S}\right)  \log \left( \frac{ \cA\left(\left\{\Big|LM\Big(\frac{1}{2}+\frac{u}{\log K}+i\frac{S}{\log K},f\Big)\Big|^2\right\}; K, \Phi\right)}{\cA(K, \Phi)}
 \right) du.
  \end{split}
\end{equation*}

From now on, we shall assume
\begin{equation}\label{eqn:S>}
  S \; \geq \; \frac{\pi}{4(1-\Upsilon)(1-20\vartheta)}.
\end{equation}
We first consider with $\frac{\cA\big(\{I_3(f)\}; K, \Phi\big)}{\cA(K, \Phi)}$.
By theorem \ref{thm:A>}, lemma \ref{lemma:firstmoment},
and \eqref{eqn:I(f)}, we have
\[
  \frac{\cA\left(\big\{I_3(f)\big\}; K, \Phi\right)}{8S\sinh\left(\frac{\pi R}{2S}\right) \cA(K, \Phi)}
       \; = \;  \mathcal{O}_{\ve}(K^{-\ve}),
\]
and
\begin{equation}\label{eqn:N0toJ12}
  \frac{\cN_{0}(K;\Phi)}{\cA(K, \Phi)}
\;  \leq \; \frac{J_1(K;\Phi) + J_2(K;\Phi)}{8S\sinh\left(\frac{\pi R}{2S}\right)}
       \;  - \; \frac{1}{2} \; + \; \mathcal{O}_{\ve}(K^{-\ve}).
\end{equation}
For $J_2(K;\Phi)$, we set
\[
  J_2(K;\Phi) = J_{21}(K;\Phi) + J_{22}(K;\Phi),
\]
where
\[
  \begin{split}
     J_{21}(K;\Phi) & := \int\limits_{-R}^{c_0\log\log K} \sinh\left(\frac{\pi(u+R)}{2S}\right)  \log \left( \frac{
        \cA\left(\left\{\Big|LM\Big(\frac{1}{2}+\frac{u}{\log K}+i\frac{S}{\log K},f\Big)\Big|^2\right\}; K, \Phi\right)}{\cA(K, \Phi)} \right) du, \\
     J_{22}(K;\Phi) & := \int\limits_{c_0\log\log K}^{(W_1-1/2)\log K} \sinh\left(\frac{\pi(u+R)}{2S}\right)  \log \left( \frac{      \cA\left(\left\{\Big|LM\Big(\frac{1}{2}+\frac{u}{\log K}+i\frac{S}{\log K},f\Big)\Big|^2\right\}; K, \Phi\right)}{\cA(K, \Phi)}
 \right) du.
  \end{split}
\]
By theorem \ref{thm:A>}, with the choices $a= \frac{2(1-\Upsilon)(1-10\vartheta)}{1-5\vartheta}$
and $S=\frac{\pi}{4(1-\Upsilon)(1-20\vartheta)}$, we have
\[
  \begin{split}
    J_{22}(K;\Phi)
    & \ll \int\limits_{c_0\log\log K}^{\log K} e^{\frac{\pi u}{2S}-2(1-\Upsilon)(1-10\vartheta)u} du \\
    & \ll \int\limits_{c_0\log\log K}^{\log K} e^{-20\vartheta(1-\Upsilon)u} du
    \;  \ll_{\vartheta,\Upsilon} \; (\log K)^{-20\vartheta(1-\Upsilon)c_0}.
  \end{split}
\]

Now by \eqref{eqn:A<}, and taking $c_0=\frac{S}{\pi}$,
we see that
\begin{equation}\label{eqn:N0toV}
  \begin{split}
    \frac{\cN_{0}(K, \Phi)}{\cA(K, \Phi)}
    & \; \leq \; \frac{1}{8S\sinh\left(\frac{\pi R}{2S}\right)}
        \left( \int\limits_{0}^{S}  \cos\left(\frac{\pi t}{2S}\right)
        \log\Big( \sV(-R,t) \Big) \;dt \right.\\
    & \hskip 75pt     \left. +
        \int\limits_{0}^{\infty} \sinh\left(\frac{\pi u}{2S}\right)
        \log\Big( \sV(u-R,S) \Big) \;du \right)
     -\frac{1}{2} + \mathcal{O}_{c}\Big((\log K)^{-c}\Big),
  \end{split}
\end{equation}
for some constant $0 < c \leq 20\vartheta(1-\Upsilon)c_0$.
Choose $c = 20\vartheta(1-\Upsilon)c_0$ and let $P(x)$, $Q(x)$, $\vartheta$, $\Upsilon$, $R$, and $S$ be as in lemma \ref{lemma:V-l}.
Then by a computer calculation of the integrals on the left hand side of (\ref{eqn:N0toV}), we get
\begin{equation}\label{eqn:N0}
  \frac{\cN_{0}(K, \Phi)}{\cA(K, \Phi)} \;\leq\; 0.3613,
\end{equation}
when $K$ is sufficiently large.
\vskip 10pt

\section{Proof of theorem \ref{thm:superpositivity2} } \label{sec:thmsp}

In this section the variables $\vartheta,   \Upsilon, R, S,$
and the polynomials $P(x)$, $Q(x)$, will be fixed as in lemma \ref{lemma:V-l}.
Let $J:=[C\log\log K]$, where $C$ is a large constant and $[x]$ means the largest integer less than $x$.  Set
$d := 2S/3$
and define the regions:
$$
  \mathcal{R}_j :=
  \begin{cases} \Big\{\beta+i\gamma \; \Big| \; \beta\geq \frac{1}{2}+\frac{jd}{\log K}, \;\;  |\gamma|\leq \frac{(j+1)d}{\log K}\Big\}, & \qquad  \mathrm{if} \; 1\leq j \leq J-1,\\
  &\\
 \Big\{\beta+i\gamma \; \Big| \; \beta\geq \frac{1}{2}+\frac{Jd}{\log K}, \;\;  |\gamma|\leq 1\Big\}, & \qquad \text{if} \; j = J,
   \end{cases}
$$
and the zero counting sum
$$\cN_j(K, \Phi) \; := \sum_{k\equiv2(4)}\Phi\left(\frac{k-1}{K}\right)
 \underset{ \text{one zero in}  \,\mathcal R_{j}} {\underset{L(s,f)\; \text{has at least} } {\sum_{f\in H_k}}}  \hskip -10pt\omega_f, \qquad \mathrm{if}\; 1\leq j \leq J.$$

For $1\leq j\leq J-1$, let $\cB_j$ be the rectangular box with vertices
$W_{0,j}\pm H_j$ and $W_1\pm H_j$,
where
\begin{equation}\label{eqn:W0j&Hj}
  W_{0,j} := \frac{1}{2}+\frac{jd/2}{\log K},\quad
  H_j := \frac{3(j+1)d/2}{\log K}.
\end{equation}
By lemma \ref{lemma:selberg} and the argument in \S\ref{sec:thmrz},
we have
\[
  \begin{split}
    &  4H_j\sinh\left(\frac{\pi j}{6(j+1)}\right) \frac{\cN_{j}(K, \Phi)}{\cA(K, \Phi)} \; \leq \; \frac{
    \cA\left( \left\{4H_j\sum_{\substack{\beta+i\gamma\in\cB_j \\ L(\beta+i\gamma,f)=0}}
    \cos\left(\frac{\pi \gamma}{2H_j}\right) \sinh\left(\frac{\pi(\beta-W_{0,j})}{2H_j}\right)\right\}; \; K, \; \Phi\right)}{\cA(K, \Phi)}
     \\
     &
     \\
    & \hskip 20pt\leq \int\limits_{0}^{H_j} \cos\left(\frac{\pi t}{2H_j}\right)
        \log \left( \frac{
        \cA\left(\left\{\left|LM\left(W_{0,j}+it,f\right)\right|^2\right\}; K, \Phi\right)}{\cA(K, \Phi)} \right) dt  \\
    & \hskip 50pt + \int\limits_{W_{0,j}}^{W_1} \sinh\left(\frac{\pi(u-W_{0,j})}{2H_j}\right)
        \log \left( \frac{
        \cA\left(\left\{\left|LM\left(u+iH_j,f\right)\right|^2\right\}; K, \Phi\right)}{\cA(K, \Phi)} \right) du
   \; + \; \mathcal O_{\varepsilon}\left( K^{-\varepsilon}  \right).
  \end{split}
\]
Consequently
\begin{equation}\label{eqn:Nj<}
  \begin{split}
    &  \frac{\cN_{j}(K, \Phi)}{\cA(K, \Phi)}   \; \leq \; \frac{1}{6(j+1)d\sinh\left(\frac{\pi j}{6(j+1)}\right)}
        \left[\int\limits_{0}^{\frac{3}{2}(j+1)d} \cos\left(\frac{\pi t}{3(j+1)d}\right)
        \log \Big( \sV(jd/2,t) \Big) dt  \right.\\
    & \hskip 60pt \left.
        + \int\limits_{0}^{\infty} \sinh\left(\frac{\pi u}{3(j+1)d}\right)
        \log \Big( \sV(u+jd/2,3(j+1)d/2) \Big) du \phantom{\int\limits_{0}^{\frac{3}{2}(j+1)d}} \hskip -27pt\right]
        + \; \mathcal{O}_{c}\Big((\log K)^{-c}\Big).
  \end{split}
\end{equation}
Let $c = 20\vartheta(1-\Upsilon)c_0$ as in (\ref{eqn:N0toV}).  By  a computer calculation of the integrals on the right hand side of (\ref{eqn:Nj<}), we obtained the following bounds:
\begin{equation}\label{eqn:Nj<-s}
  \begin{split}
   \frac{\cN_{1}(K, \Phi)}{\cA(K, \Phi)} \;  \leq \; 0.19441, \qquad  \frac{\cN_{2}(K, \Phi)}{\cA(K, \Phi)} & \; \leq \; 0.03891, \qquad   \frac{\cN_{3}(K, \Phi)}{\cA(K, \Phi)} \;  \leq \; 0.00989, \\
  \sum_{j=4}^{13} \;\frac{\cN_{j}(K, \Phi)}{\cA(K, \Phi)}  & \; \leq  \; 0.00439,
  \end{split}
\end{equation}
provided $K$ is sufficiently large.

To obtain similar bounds for $14\leq j \leq  J-1$, we will, instead, use lemma \ref{lemma:V-l}.
Now together with the fact $x\leq \sinh(x)\leq e^{x}$ for all $x>0$,
we have
\begin{equation*}
  \begin{split}
      \frac{\cN_{j}(K, \Phi)}{\cA(K,\Phi)} \; & \leq \; \frac{\frac{3}{2}(j+1)d}{6(j+1)d\sinh\left(\frac{\pi j}{6(j+1)}\right)}  e^{-jd/4}     \;    + \; \frac{ \int\limits_{0}^{\infty} e^{\frac{\pi u}{3(j+1)d}}  e^{-(u+jd/2)/2} \;du }{6(j+1)d\sinh\left(\frac{\pi j}{6(j+1)}\right)} \; + \;    \mathcal{O}_{c}\Big((\log K)^{-c}\Big)
       \\
    & \leq \; \frac{3(j+1)}{2\pi j}
        e^{-jd/4}
       \; + \; \frac{1}{\pi jd} \int\limits_{0}^{\infty} e^{\frac{\pi u}{3(j+1)d}}
        e^{-u/2-jd/4}\; du  \; + \;    \mathcal{O}_{c}\Big((\log K)^{-c}\Big) \\
    & \leq \;\frac{45}{28\pi}e^{-dj/4} + \frac{4}{\pi dj}e^{-dj/4}  \; + \;    \mathcal{O}_{c}\Big((\log K)^{-c}\Big)\\
    & \leq \; \frac{3}{5}e^{-dj/4}  \; + \;    \mathcal{O}_{c}\Big((\log K)^{-c}\Big).
  \end{split}
\end{equation*}

Hence we have
\begin{equation}\label{eqn:Nj<-l}
  \sum_{14\leq j\leq J-1} \frac{\cN_{j}(K, \Phi)}{\cA(K, \Phi)}
  \leq \frac{3}{5}\sum_{j=14}^{\infty} e^{-dj/4} \; + \; \mathcal O_c\left(\frac{C\log\log K}{(\log K)^c} \right)
 \; \leq \; 0.01212
\end{equation}
for suitable choice of $C$.
Note that by Hough \cite[Theorem 1.1]{hough2012zero}, we have
\begin{equation}\label{eqn:N_J}
  \frac{\cN_{J}(K, \Phi)}{\cA(K, \Phi)} \ll (\log K)^{-c},
\end{equation}
 when $C$ is large enough.
The choice of $C$, depends on
the absolute constant $\theta$ in Hough \cite[Theorem 1.1]{hough2012zero} but is not really important for our results.
Thus by \eqref{eqn:N0}, \eqref{eqn:Nj<-s}, \eqref{eqn:Nj<-l}, and \eqref{eqn:N_J},
we get
\[
  \frac{\cN(K, \Phi)}{\cA(K, \Phi)}
\;  \leq \; \sum_{j=0}^{J} \frac{\cN_{j}(K, \Phi)}{\cA(K, \Phi)}
\;  \leq \; 0.63,
\]
when $K, C$ is large enough.

It immediately follows from the above that
\begin{equation}\label{bound1} \cM(K, \Phi)=\cA(K, \Phi)-\cN(K, \Phi)
\; \geq \; 0.27 \cdot\cA(K, \Phi).
\end{equation}
By \cite{hoffstein1994coefficients}
and Goldfeld--Hoffstein--Lieman \cite{goldfeld1994effective} one may obtain the upper bound:
\begin{equation} \label{bound2}
\Phi\left(\frac{k-1}{K}\right) \omega_f =   \Phi\left(\frac{k-1}{K}\right) \frac{\zeta(2)}{(k-1)/12}
    \cdot\frac{1}{L(1,\operatorname{sym}^2 f)}
  \ll \frac{\log K}{K}.
\end{equation}
On the other hand, we have already shown the asymptotic formula (see (\ref{A(kPhi-asymp}))
\begin{equation} \label{bound3}
\cA(K, \Phi)  \; = \; \frac{K}{4}\widetilde{\Phi}(1) \; + \; \mathcal  O\left(K^{-B}\right).
\end{equation}
It now easily follows from definition \ref{defSums} and (\ref{bound1}), (\ref{bound2}), (\ref{bound3}), that the number of $f\in H_k$ with $k\equiv2\ (\mod4)$ such that
$L(s,f)$ has no zero $\rho=\beta+i\gamma$ with
$\beta\in (1/2,1]$ and $|\gamma|\leq \beta-1/2$ will be
$\gg K^2/\log K$.

\section{Acknowledgments}
Bingrong Huang would like to thank Jianya Liu and Wei Zhang  for their valuable advice and constant encouragement and Peter Sarnak for helpful conversations.
He also thanks the Department of Mathematics at Columbia University for its hospitality. The authors would also like to thank Zeev Rudnick for helpful comments.


\end{document}